\documentclass[ejs]{imsart}

\RequirePackage[OT1]{fontenc}
\RequirePackage{amsthm,amsmath,amsfonts,amssymb}
\RequirePackage[numbers]{natbib}
\RequirePackage[colorlinks,citecolor=blue,urlcolor=blue]{hyperref}
\usepackage{graphicx}

\usepackage{mathrsfs,dsfont,subfigure}
\usepackage{ifthen}
\usepackage{color}
\usepackage{array}
\usepackage{enumerate}
\usepackage{stmaryrd}
\usepackage{bbm,bm}
\usepackage{comment}
\usepackage{xspace}

\pubyear{2021}
\volume{0}
\issue{0}
\firstpage{1}
\lastpage{43}

\startlocaldefs
\numberwithin{equation}{section}
\theoremstyle{plain}

\endlocaldefs

\newtheorem{theorem}{Theorem}[section]
\newtheorem{lemma}[theorem]{Lemma}

\theoremstyle{remark}
\newtheorem*{example}{Example}

\newcommand{\E}{\mathrm{E}}

\definecolor{iblue}{rgb}{0.1,0,0.75}
\definecolor{ired}{rgb}{0.9,0,0.1}

\newcommand{\half}{\Small{\frac{1}{2}}}

\newcommand{\scrB}{{\mathscr B}}
\newcommand{\scrC}{{\mathscr C}}

\newcommand{\scrI}{{\mathscr I}}

\newcommand{\scrL}{{\mathscr L}}

\def\cC{\mathcal C}

\def\cF{\mathcal F}
\def\cG{\mathcal G}

\def\cK{\mathcal K}

\def\cP{\mathcal P}
\def\cR{\mathcal R}
\def\cS{\mathcal S}

\newcommand{\bbE}{{\mathbb E}}
\newcommand{\bbR}{{\mathbb R}}

\newcommand{\bbP}{{\mathbb P}}
\newcommand{\bbQ}{{\mathbb Q}}

\newcommand{\ie}{{\it i.e.}}

\newcommand{\eg}{{\it e.g.}}

\newcommand{\Levy}{L\'{e}vy}

\newcommand{\bc}{\begin{center}}
\newcommand{\ec}{\end{center}}
\newcommand{\be}{\begin{equation}}
\newcommand{\ee}{\end{equation}}
\newcommand{\ba}{\begin{array}}
\newcommand{\ea}{\end{array}}
\newcommand{\bean}{\setlength\arraycolsep{1pt}\begin{eqnarray*}}
\newcommand{\eean}{\end{eqnarray*}}
\newcommand{\bea}{\setlength\arraycolsep{1pt}\begin{eqnarray}}
\newcommand{\eea}{\end{eqnarray}}
\newcommand{\ben}{\begin{enumerate}}
\newcommand{\een}{\end{enumerate}}
\newcommand{\bed}{\begin{itemize}}
\newcommand{\eed}{\end{itemize}}

\def\half{\hbox{$1\over2$}}

\def\iidsim{ \stackrel{\rm iid}{\sim} }

\begin{document}

\begin{frontmatter}
\title{Posterior asymptotics in Wasserstein metrics on the real line
}
\runtitle{Posterior asymptotics in Wasserstein metrics}

\begin{aug}
\author{\fnms{Minwoo} \snm{Chae} 
\ead[label=e1]{mchae@postech.ac.kr}}

\address{Department of Industrial and Management Engineering\\
Pohang University of Science and Technology, South Korea
}

\author{\fnms{Pierpaolo} \snm{De Blasi} 
\ead[label=e2]{pierpaolo.deblasi@unito.it}}

\address{University of Torino and Collegio Carlo Alberto, Italy\\
}

\and

\author{\fnms{Stephen G.} \snm{Walker}
\ead[label=e3]{s.g.walker@math.utexas.edu}
}

\address{Department of Mathematics\\
University of Texas at Austin, USA\\
}

\runauthor{Chae, De Blasi and Walker}

\affiliation{Some University and Another University}

\end{aug}

\begin{abstract}
In this paper, we use the class of Wasserstein metrics to study asymptotic properties of posterior distributions. 
Our first goal is to provide sufficient conditions for posterior consistency. 
In addition to the well-known Schwartz's Kullback--Leibler condition on the prior, the true distribution and most probability measures in the support of the prior are required to possess moments up to an order which is determined by the order of the Wasserstein metric.
We further investigate convergence rates of the posterior distributions for which we need stronger moment conditions.
The required tail conditions are sharp in the sense that the posterior distribution may be inconsistent or contract slowly to the true distribution without these conditions.
Our study involves techniques that build on recent advances on Wasserstein convergence of empirical measures. We apply the results to some examples including a Dirichlet process mixture prior and conduct a simulation study for further illustration.
\end{abstract}

\begin{keyword}[class=MSC]
\kwd[Primary ]{62F15}
\kwd{62G20}
\kwd[; secondary ]{62G07}
\end{keyword}

\begin{keyword}
\kwd{Dirichlet process mixture}
\kwd{nonparametric Bayesian inference}
\kwd{posterior consistency}
\kwd{posterior convergence rate}
\kwd{Wasserstein metrics}
\end{keyword}

\setcounter{tocdepth}{1}

\tableofcontents
\end{frontmatter}

\section{Introduction\label{sec:intro}}

The Wasserstein distance originally arose in the problem of optimal transportation \citep{villani2003topics} and is often called the Kantorovich or transportation distance.
We refer to \cite{vershik2013long} for the history about this metric. For two Borel probability measures $P$ and $Q$ on the real line, the Wasserstein metric of order $p$, $p \in [1, \infty)$, is defined as
\bean 
	W_p (P,Q) = \inf_{\pi \in \scrC(P, Q)} \left( \int_{\bbR^2} |x-y|^p d\pi(x,y) \right)^{1/p},
\eean
where $\scrC(P, Q)$ is the set of every coupling $\pi$ of $P$ and $Q$, that is, a Borel probability measure on $\bbR^2$ with marginals $P$ and $Q$, respectively.

There are a wide number of applications of Wasserstein metrics, \eg\
Wasserstein generative adversarial networks (GAN; \cite{arjovsky2017wasserstein,gulrajani2017improved}), approximate Markov chain Monte Carlo (MCMC; \cite{rudolf2018perturbation}), distributionally robust optimization (DRO; \cite{kuhn2019wasserstein}) and clustering (\cite{biau2008performance, laloe2010quantization}). However, exhaustive study on statistical properties such as the convergence behavior of the empirical measure with respect to $W_p$ have been conducted only recently, see \cite{bobkov2019one, fournier2015rate, weed2019sharp, dereich2013constructive}.
In particular, the great success of Wasserstein GAN  in machine learning society accelerated the study of Wasserstein metrics in statistics community as a discrepancy measure between probabilities; \cite{singh2018nonparametric, liang2018well, biau2018some}. Recently, \cite{bernton2019jrssb} proposed the use of the Wasserstein distance in the implementation of Approximate Bayesian Computation (ABC) to approximate the posterior distribution. 
In nonparametric Bayesian inference, \cite{nguyen2013convergence, chae2019bayesianConsistency} used Wasserstein metrics to study asymptotic properties of posterior distributions, but $W_p$ was considered as a distance between mixing distributions rather than a distance between mixture densities themselves.
As a result, the Wasserstein metrics in these papers yielded a stronger topology than the total variation distance on the space of density functions. 
In general, $W_p$, $1 \leq p < \infty$, metrizes the weak convergence of probability measures in a bounded metric space.
Specifically, if the diameter of the underlying metric space is bounded by 1, one has the relationship $d_P^2 \leq W_1 \leq 2d_P \leq d_V$, where $d_P$ and $d_V$ are \Levy-Prokhorov and total variation distances, see \cite{gibbs2002choosing}.
In an unbounded metric space, the second and third inequalities do not hold because $W_1$ is not a bounded metric.

In this article, we utilize the Wasserstein distances to study asymptotic behavior of posterior distributions under the assumption that data are generated from a fixed true distribution and we focus on nonparametric Bayesian density estimation on the real line. To set the stage, let $X_1, \ldots, X_n$ be the observations which are independent and identically distributed random variables from the true distribution $P_0$ possessing a density $p_0$.
Let $\cF$ be a collection of probability densities in $\bbR$ equipped with the weak topology, and $\Pi$ be a prior distribution on $\cF$.
Then the posterior probability of a measurable set $A \subset \cF$ is given as
\be \label{eq:posterior}
	\Pi(p \in A \mid X_1, \ldots, X_n) = \frac{\int_A \prod_{i=1}^n p(X_i)/p_0(X_i) d\Pi(p)}{\int \prod_{i=1}^n p(X_i)/p_0(X_i) d\Pi(p)}
\ee
by the Bayes formula.
Throughout the paper, we allow the prior $\Pi$ to depend on the sample size $n$, but often abbreviate this dependency in the notation of both prior and posterior distributions.
If clarification is necessary, the prior and posterior will be denoted $\Pi_n$ and $\Pi_n(\cdot \mid X_1, \ldots, X_n)$, respectively.
The posterior distribution is said to be consistent with respect to a (pseudo-)metric $d$ if
\bean
	\Pi\Big( d(p, p_0) > \epsilon \mid X_1, \ldots, X_n\Big) \rightarrow 0
	~ \text{in probability for every $\epsilon > 0$},
\eean
where the convergence in probability is taken with respect to the true distribution $P_0$.
If $\epsilon$ is replaced by $\epsilon_n$ for some sequence $\epsilon_n \rightarrow 0$, the convergence rate of the posterior distributions is said to be at least $\epsilon_n$.
There is a huge amount of research articles concerning asymptotic properties of the posterior distribution.
We refer to the monograph \cite{ghosal2017fundamentals} for the history and details about this topic. 

Of key importance is the Kullback--Leibler (KL) support condition developed by \cite{schwartz1965bayes}.
A fixed prior $\Pi$ is said to satisfy the KL support condition if 
\be\label{eq:KL-support}
	\Pi\Big(p: K(p_0, p) < \epsilon \Big) > 0 \quad \text{for every $\epsilon > 0$},
\ee
where $K(p_0, p) = \int \log[p_0(x)/p(x)] dP_0(x)$ is the KL divergence.
If the prior depends on the sample size, the KL condition \eqref{eq:KL-support} can be replaced by
\be\label{eq:KL-support-n}
	\liminf_{n\rightarrow \infty} \Pi_n\Big(p: K(p_0, p) < \epsilon \Big) > 0 \quad \text{for every $\epsilon > 0$}.
\ee
Conditions \eqref{eq:KL-support} and \eqref{eq:KL-support-n} became standard for proving posterior consistency.
In particular, it gives a suitable lower bound of the denominator in \eqref{eq:posterior} and it implies posterior consistency in the weak topology, that is with respect to the \Levy-Prokhorov distance, see \cite{schwartz1965bayes} and Section 6.4 of \cite{ghosal2017fundamentals}. 
A variation of the KL support condition to obtain a convergence rate is developed by \cite{ghosal2000convergence}.
It is formally expressed as
\be \label{eq:KL-rate}
	\Pi(\cK_n) \geq e^{-n\epsilon_n^2} \quad \text{for all large enough $n$,}
\ee
where
\bean 
	\cK_n = \left\{ p\in\cF : \int \log \frac{p_0}{p} dP_0 \leq \epsilon_n^2, \int \Big(\log \frac{p_0}{p} \Big)^2 dP_0 \leq \epsilon_n^2  \right\}.
\eean

In literature, studies on posterior asymptotics have focused on strong metrics such as the total variation, Hellinger and uniform metrics.
For those purposes, some non-trivial conditions such as the bounded entropy or prior summability are assumed in addition to the KL conditions, see \cite{ghosal1999posterior, walker2004new, barron1999consistency, chae2017novel} for example.
On the other hand, it is surprising that careful analysis of the convergence rates with respect to a weak metric such as the \Levy-Prokhorov and Kolmogorov has not been studied in literature, considering that the KL support condition is sufficient for the consistency in those metrics.
\cite{chae2017novel} studied the convergence rate of the posterior distribution with respect to the \Levy-Prokhorov metric, but their rate $n^{-1/4}$ have a lot of room for improvement.
Furthermore, they used the \Levy-Prokhorov rate as a tool for proving the consistency in total variation, and did not focus on the convergence rate itself.

Wasserstein metrics $W_p$, $1 \leq p < \infty$ metrize weak convergence in a bounded space, but it generates a stronger topology in general.
Indeed, neither the KL support condition \eqref{eq:KL-support} nor \eqref{eq:KL-rate} are sufficient for posterior consistency with respect to $W_p$.
If $P_0$ is a standard Cauchy density, for example, $W_p(P, P_0) = \infty$ for any $P$ and $p \geq 1$.
Therefore, for any prior except the one putting all its mass on $P_0$, the posterior distribution is inconsistent with respect to $W_p$.
This simple example shows that tails or moments of probability measures play an important role for handling $W_p$.

For a sequence $P_n$ of probability measures, it is well-known that $W_p(P_n, P) \rightarrow 0$ if and only if $P_n$ converges to $P$ weakly and $M_p(P_n) \rightarrow M_p(P)$, see \cite{villani2003topics}, p.212, where $M_p(P) = \int |x|^p dP(x)$.
Therefore, for the Wasserstein consistency to hold, the posterior moment should converge to the true moment, see Theorem \ref{thm:moment}. However, while the moment consistency of frequentist's nonparametric estimators such as the the empirical distribution is straightforward, it is non-trivial to show that the posterior moment converges to the true moment even with a very popular prior such as a Dirichlet process mixture.
This is mainly because tails of probability measures in the support of the prior should be considered simultaneously.

To prove posterior consistency, we will leverage on the KL condition. We provide two different approaches which are of independent interest; see the proof of Theorem \ref{thm:consistency}. The first one targets directly posterior moment consistency and relies on a result from \cite{walker2004new}. The second one has less stringent conditions but the proof is more complicated. Specifically, we construct uniformly consistent tests based on the empirical distribution by exploiting suitable upper bounds of Wasserstein metrics. We then show that, to achieve posterior consistency with respect to $W_p$, moments of densities must be suitably bounded. In particular, the posterior needs to put most of its mass on distributions that possess moments up to an order determined by that of the Wasserstein metric. 
In practice, the posterior moment condition can be worked out by means of exponentially small prior probability on the complement set, cf. Lemma \ref{lem:ghosal2000prior}. In Section \ref{ssec:dpm} we provide an illustration in the specific example of Dirichlet process mixture prior.


Both approaches for posterior consistency can be extended to obtain suitable convergence rates with the KL condition \eqref{eq:KL-rate}.
While the first approach gives the convergence rate for the moment, the second approach gives the rate with respect to $W_p^p$ relying on slightly stronger moment conditions, see Theorems \ref{thm:moment-rate} and \ref{thm:Wp-rate}. For convergence rates with the second approach, we rely on new upper bounds on Wasserstein metrics that can be of independent interest, cf.\ Lemma \ref{lem:W-set}. 
Interestingly, the posterior moment conditions for consistency and convergence rates are nearly necessary, that is the posterior distribution may be inconsistent or contract slowly to the true distribution when they are not satisfied. 
Finally, we obtain convergence rates for the case $p=\infty$ in Theorem \ref{thm:W_infty}, for which we need to restrict to probability measures on a bounded space. 

To the best of our knowledge, this paper is the first result on posterior asymptotics with the Wasserstein metric.
The remainder of this paper is organized as follows.
Results on posterior consistency and its convergence rate with respect to $W_p$, for $1\leq p < \infty$, are considered in Sections \ref{sec:consistency} and \ref{sec:rates}, respectively.
Posterior asymptotics with respect to $W_\infty$ is studied in Section \ref{sec:rates-infty}.
Section \ref{sec:example} considers more details with specific examples.
Some numerical results complementing our theory are provided in Section \ref{sec:sim}.
Concluding remarks and proofs are given in Sections \ref{sec:conclusion} and \ref{sec:proof}, respectively.

\subsection*{Notation}
Before proceeding, we introduce some further notation; for two real numbers $a$ and $b$, their minimum and maximum are denoted by $a \wedge b$ and $a \vee b$, respectively.
Inequality $a \lesssim b$ means that $a$ is less than a constant multiple of $b$, where the constant is universal unless specified.
Upper cases such as $P$ and $Q$ refer to probability measures corresponding to the densities denoted by lower cases and vise versa.
The empirical measure based on $X_1, \ldots, X_n$ is denoted $\bbP_n$.
For a real-valued function $f$, its expectation with respect to $P$ is denoted $Pf$.
The expectation with respect to the true distribution is often denoted $\E f(X)$.
The restriction of $P$ onto a set $A$ is denoted $P|_A$.


\section{Consistency with respect to $W_p$}
\label{sec:consistency}

Recall that $W_p(P_n, P) \rightarrow 0$ if and only if $P_n$ converges weakly to $P$ and $M_p(P_n) \rightarrow M_p(P)$, see Theorem 7.12 of \cite{villani2003topics}.
Also, the KL support condition \eqref{eq:KL-support}, or \eqref{eq:KL-support-n}, guarantees posterior consistency with respect to the \Levy-Prokhorov metric which  induces the weak convergence.
Therefore, it is natural under the KL support condition to guess that posterior consistency with respect to $W_p$ is equivalent to the consistency of the $p$th moment, that is,
\be\label{eq:moment-consistency}
	\Pi \Big( \big|M_p(P) - M_p(P_0) \big| > \epsilon \mid X_1, \ldots, X_n \Big) \rightarrow 0 \quad \text{in probability for every $\epsilon > 0$.}
\ee
If \eqref{eq:moment-consistency} holds, we say that the posterior moment of order $p$ is consistent.
For $p=1$, the moment consistency can be easily implied by $W_1$-consistency by the help of the duality theorem by \cite{kantorovich1958space}, see also \cite{dudley1989real, de1982invariance, villani2008optimal}, which asserts that
\bean
	W_1(P, Q) = \sup_{f \in \scrL} \left| \int f(x) dP(x) - \int f(x) dQ(x)\right|,
\eean
where $\scrL$ is the class of every Lipschitz function whose Lipschitz constant is bounded by 1.
Since the map $x \mapsto |x|$ belongs to $\scrL$, we have that $\left| M_1(P) - M_1(Q) \right| \leq W_1(P, Q).$
Although such an explicit bound does not exist for $p > 1$, one can show that posterior consistency with respect to the Wasserstein distance is equivalent to the moment consistency under the KL support condition.

\begin{theorem} \label{thm:moment}
For a prior $\Pi$, suppose that the KL condition \eqref{eq:KL-support-n} holds. Then, the consistency of the $p$-th moment \eqref{eq:moment-consistency} is equivalent to that
\be \label{eq:consistency-W}
	\Pi\Big( W_p(P, P_0) \geq \epsilon \mid X_1, \ldots, X_n\Big) \longrightarrow 0
	\quad \text{in probability for every $\epsilon > 0$}.
\ee
\end{theorem}

We provide two different approaches for proving posterior consistency with respect to $W_p$ which are of independent interests.
The first approach relies on a result from \cite{walker2004new}; namely that if $\cC$ is a convex set of probability measures and $\inf_{P\in \cC}H(P_0,P)>0$ then $\Pi(\cC|X_1,\ldots,X_n)\to 0$ in probability, where $H(P, Q)$ denotes the Hellinger distance between $P$ and $Q$.
This approach directly uses Theorem \ref{thm:moment} by establishing the consistency of the $p$th moment \eqref{eq:moment-consistency}.
The proof based on this approach is very simple as it only needs a single application of the Cauchy--Schwarz inequality.
However, it requires the moment of order $2p$ to be bounded a posteriori.

The second approach constructs a uniformly consistent sequence of tests based on the convergence of empirical distribution.
The uniformity does not make any problem for the compact support case, \ie\ $P_0([-1,1])=1$ and $P([-1,1])=1$ for every $P$ in the support of the prior $\Pi$.
If probability measures in the support of the prior have unbounded support, however, problems may happen due to probability measures with large moments.
This problem can be avoided if the moments are suitably bounded a posteriori, as expressed through condition \eqref{eq:consistency-P} below. 
The second approach relies on a rather complicated proof, but it only needs the moment of order $p+\delta$, for some $\delta > 0$, to be bounded.

\medskip
\begin{theorem} \label{thm:consistency}
Assume that the prior $\Pi$ satisfies the KL condition \eqref{eq:KL-support-n}.
Furthermore, assume that there exist positive constants $K$ and $\delta$ such that $M_{p+\delta}(P_0) < \infty$
and
\be \label{eq:consistency-P}
	\Pi\Big( M_{p+\delta}(P) \leq K \mid X_1, \ldots, X_n \Big) \rightarrow 1 ~ \text{in probability.}
\ee
Then for every $\epsilon > 0$,
\bean
	\Pi\Big( W_p(P, P_0) > \epsilon \mid X_1, \ldots, X_n\Big) \longrightarrow 0 ~ \text{in probability}.
\eean
\end{theorem}

It should be emphasized that assumptions in Theorem \ref{thm:consistency} are nearly necessary.
Certainly, $M_p(P_0) < \infty$ is necessary.
Since the consistency with respect to $W_p$ entails the consistency of the $p$th moment by Theorem \ref{thm:moment}, it is also necessary that
\be \label{eq:consistency-P-delta0}
	\Pi\Big(
	M_p(P) \leq K
	\mid X_1, \ldots, X_n\Big) \rightarrow 1 
	\quad \text{in probability for some constant $K$.}
\ee
On the other hand, $M_p(P_0) < \infty$ and \eqref{eq:consistency-P-delta0} are not sufficient for the posterior distribution to be consistent with respect to $W_p$, as shown in the following example.

\medskip
\begin{example}
Let $P_0 = \delta_0$, $P_n = (1-n^{-1}) \delta_0 + n^{-1} \delta_{x_n}$ and $\Pi(\{P_0\}) = \Pi(\{P_n\}) = 1/2$, where $\delta_x$ is the Dirac measure at $x$ and $x_n = n^{1/p}$.
Obviously, the KL condition \eqref{eq:KL-support-n} holds.
Furthermore, $W_p^p(P_0, P_n) = M_p(P_n) = 1 < \infty$ and $M_{p+\delta}(P_n) = n^{\delta/p} \rightarrow \infty$ for every $\delta >0$.
Since $P_0(X_1=0, \ldots, X_n=0) = 1$ and $P_n(X_1=0, \ldots, X_n=0) = (1-n^{-1})^n \rightarrow e^{-1} > 0$, the posterior distribution is inconsistent with respect to $W_p$.
Here, condition \eqref{eq:consistency-P-delta0} holds, but \eqref{eq:consistency-P} is violated for any $\delta > 0$.
\end{example}

By Theorem \ref{thm:consistency}, the proof of the Wasserstein consistency boils down to
\be \label{eq:bounded-moment}
	\Pi\Big( M_p(P) \leq K \mid X_1, \ldots, X_n \Big)
	\longrightarrow 1 \quad \text{in probability}
\ee
for a constant $K$, a condition that seems easy to prove at first sight. However, the proof is not simple even with a well-known prior which puts all of its mass on the space of light-tailed distributions, that is, distribution with large or infinite tail index.
Here, if a distribution function $F$ satisfies $1-F(x) = x^{-\alpha} L(x)$ for large enough $x$, where $L(\cdot)$ is a slowly varying function satisfying $\lim_{y \rightarrow \infty} L(xy)/L(y) = 1$ for any $x > 0$, the positive constant $\alpha$ is called the (right) tail index of $F$, see \cite{li2019posterior} for a Bayesian consistency of the tail index.
It should be noted that a light-tail, \ie\ large tail index, does not guarantee a small value of moment, which makes the proof of posterior consistency in $W_p$ difficult. This is in stark contrast to that the moment of the empirical distribution can be trivially shown to be consistent. 
In Section \ref{ssec:dpm}, we are able to work out the case of Dirichlet process mixture prior by using Lemma \ref{lem:ghosal2000prior}, that is by establishing that the prior puts exponentially small mass to probability measures $P$ with $M_{p}(P) > K$. See Theorem \ref{thm:dpm-specific}.  


\section{Convergence rates with respect to $W_p$}
\label{sec:rates}

For a given rate sequence $\epsilon_n$, suppose that $\Pi(\cK_n) \geq e^{-n\epsilon_n^2}$ for every large enough $n$.
Based on this condition, which is used to find a lower bound of the integrated likelihood, the denominator in the expression \eqref{eq:posterior}, we will extend the results of Section \ref{sec:consistency} to obtain a convergence rate.
The main task in this section is to find additional assumptions required to achieve the convergence rate $\epsilon_n$.
An extension of the first proof for Theorem \ref{thm:consistency} requires the moment of order $2p$ as follows.

\begin{theorem}\label{thm:moment-rate}
Assume that the prior $\Pi$ satisfies the KL condition \eqref{eq:KL-rate} for a sequence $\epsilon_n$ with $\epsilon_n\to 0$ and $n\epsilon_n^2 \to \infty$.
Furthermore, assume that there exists a constant $K$ such that $M_{2p}(P_0) \leq K$ and 
\bean
	\Pi\Big(P:\, M_{2p}(P)>K \mid X_1,\ldots,X_n\Big) \to 0\quad\text{in probability.}
\eean
Then
\bean
	\Pi\Big(|M_p(P)-M_p(P_0)|> K'\epsilon_n \mid X_1,\ldots,X_n \Big)\to 0\quad \text{in probability}
\eean
for some constant $K' > 0$.
\end{theorem}

Note that $M_p(P)$ is a linear functional of $P$ for which the semi-parametric Bernstein--von Mises (BvM) theorem may hold, see \cite{castillo2015bernstein, rivoirard2012bernstein}.
In this case, the convergence rate of the marginal posterior distribution of $M_p(P)$ would be the parametric rate $n^{-1/2}$ even though the global posterior convergence rate $\epsilon_n$ may be slower.
However, while Theorem \ref{thm:moment-rate} is very general, the semi-parametric BvM theorem holds under rather strong conditions.
For example, the above mentioned papers consider only specific priors and relied on the assumption that $p_0$ is compactly supported and bounded away from zero.
It is sometimes possible to obtain the parametric convergence rate for the finite-dimensional parameter of interests without the semi-parametric BvM theorem.
However, the proof typically relies on the LAN (locally asymptotically normal) expansion of the log-likelihood, see \cite{bickel2012semiparametric, chae2019semiparametric}.

Next, we consider an extension of the testing approach.
To achieve the convergence rate $\epsilon_n$, we will construct a sequence of consistent test
\bean
	P_0 \phi_n \rightarrow 0
	\quad {\rm and} \quad
	\sup_{P \in \cF_n^c} P(1-\phi_n) \leq e^{-3n\epsilon_n^2},
\eean
where $\cF_n = \{P: W_p^p(P, P_0) \leq K \epsilon_n\} \cap \cF_0$.
Here, $\cF_0$ will be defined as a collection of probability measures whose tails and moments are suitably bounded.
Then, it will suffices for the desired result to show that $\Pi(\cF_0^c \mid X_1, \ldots, X_n) \rightarrow 0$ in probability.

A consistent sequence of tests will be constructed based on the convergence of the empirical distribution to the true distribution.
Note that there are well-known concentration inequalities of the form $P( W_p^p(\bbP_n, P) > \epsilon_n ) \leq \delta_n,$ where $\delta_n$ is a decaying sequence, and those inequalities might be directly used to define tests as
\bean
	\phi_n &=& \left\{\begin{array}{ll}
		1 & ~~\text{if $W_p^p(\bbP_n, P_0) > \epsilon_n$}
		\\
		0 & ~~\text{otherwise.}
	\end{array}\right.
\eean
However, such a simple approach does not give sharp convergence rates of the posterior distribution.
For example, if we apply the concentration inequality by \cite{fournier2015rate}, for any $P$ with $W_p^p(P, P_0) > 2^p \epsilon_n$ and $M_{2p+\delta}(P) < \infty$, we have
\be\begin{split} \label{eq:concentration-ineq}
	& P(1-\phi_n) = P\Big( W_p^p(\bbP_n, P_0) \leq \epsilon_n \Big)
	\leq P\Big( W_p^p(\bbP_n, P) \geq 2^{-(p-1)}W_p^p(P, P_0) - \epsilon_n \Big)
	\\
	& \leq P\Big( W_p^p(\bbP_n, P) \geq \epsilon_n \Big)
	\leq c_1 \bigg(e^{-c_2n\epsilon_n^2} + \frac{1}{n\epsilon_n^2} \frac{1}{(n\epsilon_n)^{\delta/2p}} \bigg)
\end{split}\ee
where $c_1$ and $c_2$ are constants.
Here, the constants $c_1$ and $c_2$ depends on the moments of $P$, so it is not easy to bound \eqref{eq:concentration-ineq} uniformly.
Furthermore, the second term in the right hand side of \eqref{eq:concentration-ineq} is of polynomial order in $n\epsilon_n^2$ which decays too slowly compared to $e^{-n\epsilon_n^2}$.
In turn, the use of $\phi_n$ would give a much slower convergence rate than $\epsilon_n$.

Theorem \ref{thm:Wp-rate} below is our main results concerning convergence rates of the posterior distribution.
Proof of Theorem \ref{thm:Wp-rate} relies on the set-up in \cite{fournier2015rate}.
In particular, Lemmas \ref{lem:D-compact} and \ref{lem:W1-unbounded} can be easily deduced from the results in \cite{fournier2015rate}.
We build on these two lemmas to develop some techniques whose details are different from \cite{fournier2015rate}.
As mentioned above, we need to construct a sequence of tests decaying with an exponential rate.
As far as we know, this is not possible with the proof technique used in \cite{fournier2015rate}. 
Given a bounded moment condition, we achieve this by the help of Lemma \ref{lem:W-set}.
The condition $\epsilon_n \geq \sqrt{(\log n)/n}$ in Theorem \ref{thm:Wp-rate} is assumed only for technical reason. 
Although we could not succeed to eliminate this condition, we believe the result is valid for any $\epsilon_n \downarrow 0$ with $n\epsilon_n^2 \rightarrow \infty$.

\medskip
\begin{theorem} \label{thm:Wp-rate}
Assume that the prior $\Pi$ satisfies the KL condition \eqref{eq:KL-rate} for a sequence $\epsilon_n$ with $\epsilon_n \downarrow 0$ and $\epsilon_n \geq \sqrt{(\log n)/n}$.
Furthermore, assume that there exist positive constants $K$ and $\delta$ such that $M_{2p+\delta}(P_0) < \infty$ and $\Pi( M_{2p+\delta}(P) \leq K \mid X_1, \ldots, X_n) \rightarrow 1$ in probability.
Then, for some constant $K' > 0$, 
\bean
	\Pi\Big( W_p^p(P, P_0) \geq K' \big\{\epsilon_n + I(p=1) \log \epsilon_n^{-1}\big\} \mid X_1, \ldots, X_n \Big) \longrightarrow 0
	\quad \text{in probability}.
\eean
\end{theorem}

\medskip
Assumptions in Theorems \ref{thm:Wp-rate} should be understood as sufficient conditions so that $\Pi(\cK_n) \geq e^{-n\epsilon_n^2}$ guarantees $\epsilon_n$ as the posterior convergence rate for any $\epsilon_n \gg n^{-1/2}$.
For the empirical measure to achieve the rate $n^{-1/2}$ with respect to $W_p^p$, the same moment condition $M_{2p+\delta}(P_0) < \infty$ is considered in \cite{fournier2015rate}.
They provided an example showing that this moment condition cannot be weakened in general.
As illustrated in the example at the end of this section, the posterior moment condition $\Pi( M_{2p+\delta}(P) \leq K \mid X_1, \ldots, X_n) \rightarrow 1$ cannot also be weakened.

When $p > 1$, Theorem \ref{thm:Wp-rate} gives a rate $\epsilon_n$ with respect to $W_p^p$ rather than $W_p$.
This result is more relevant to \cite{fournier2015rate} than \cite{bobkov2019one}.
In particular, condition $M_{2p+\delta}(P_0) < \infty$ is the same to Eq. (3) of \cite{fournier2015rate}, and much weaker than 
\be \label{eq:J_p}
	\int \frac{[F_0(x) (1 - F_0(x))]^{p/2}}{p_0(x)^{p-1}} dx < \infty
\ee
which is a necessary and sufficient condition in \cite{bobkov2019one} for that $\E[W_p(\bbP_n, P_0)] \asymp n^{-1/2}$.
When $p = 1$, $M_{2p+\delta}(P_0) < \infty$ is only slightly stronger than \eqref{eq:J_p}, which is reduced to $\int \sqrt{F_0(x) ( 1 - F_0(x))} dx < \infty$, where $F_0$ is the cumulative distribution function of $P_0$.
If $p > 1$, however, $M_{2p+\delta}(P_0) < \infty$ is much weaker than \eqref{eq:J_p} which may not be satisfied even when $P_0$ is compactly supported.
Note that if $P_0$ is standard normal, \eqref{eq:J_p} is satisfied if and only if $1 \leq p < 2$.
As mentioned in \cite{bobkov2019one}, the rate $\E[W_p(\bbP_n, P_0)] \asymp n^{-1/2}$ cannot be obtained under moment-type conditions considered in Theorem \ref{thm:Wp-rate}.
Therefore, we would need a stronger assumption such as \eqref{eq:J_p} to replace $W_p^p$ by $W_p$ in Theorem \ref{thm:Wp-rate}.
Since we are focusing on the moment-type condition in the present paper, we do not address more detail about condition \eqref{eq:J_p}.
Instead, we consider the metric $W_\infty$ in Section \ref{sec:rates-infty} with a stronger assumption.
Specifically, $P_0$ will be assumed to be supported on a bounded interval.
This is necessary to obtain the consistency with respect to $W_\infty$.
The result in Section \ref{sec:rates-infty} guarantees the rate $\epsilon_n$ with respect to $W_p$, not $W_p^p$, at least when probabilities are compactly supported.

Note that our approach does not guarantee the rate $n^{-1/2}$ which is minimax optimal and achieved by the empirical measure under some general conditions, \eg\ \cite{weed2019estimation, fournier2015rate, bobkov2019one}.
Our approach gives the rate $n^{-1/2}$ only if the prior puts sufficiently large mass around the KL neighborhood of $p_0$.
This is mainly because our approach relies on the general approach of \cite{ghosal2000convergence} for which the KL condition $\Pi(\cK_n) \gtrsim e^{-n\epsilon_n^2}$ plays an important role to determine the rate.
Also, note that the testing approach only gives sharp rates when the distance is compatible with the natural statistical distance of the model, the Hellinger distance in our case, see \cite{hoffmann2015adaptive} for extensive discussion on this point.
In this regards, it might not be possible to obtain sharp rates based on the testing approach.
Hence, a different approach would be necessary to achieve the rate $n^{-1/2}$, \eg\ the approach in \cite{hoffmann2015adaptive, yoo2017adaptive}.
Another possible approach would be to utilize the functional Bernstein--von Mises theorem.
Specifically, the approach given in \cite{castillo2014bernstein}, combining with the Kantorovich-Rubinstein representation, might give the rate $n^{-1/2}$ at least for $W_1$, and further limiting distribution of the posterior distribution.
Note that the above papers are limited to specific priors and probability measures on a bounded set, while the present paper focuses on the moment condition for the posterior convergence rate.
With these approaches, it would be highly interesting to investigate sufficient conditions to achieve the rate $n^{-1/2}$.

For $p=1$, an additional logarithmic term in Theorem \ref{thm:Wp-rate} can be eliminated if we assume a slightly stronger condition, which is satisfied if $p_0(x) \leq K |x|^{-(3 +\delta)}$ and $\Pi(p(x) \leq K |x|^{-(3 +\delta)} ~ \forall x \mid X_1, \ldots, X_n) \rightarrow 1$ in probability for some positive constants $K$ and $\delta$, see Theorem \ref{thm:Wp-rate-general} for details.
Since Theorem \ref{thm:Wp-rate-general} relies on some technical assumptions, we defer its statement to Section \ref{sec:proof}, and provide here a simpler statement, Theorem \ref{thm:Wp-rate}.
Proofs of these theorems are quite similar.

Finally, we note that moment conditions in Theorem \ref{thm:Wp-rate} cannot be weakened to $\delta < 0$ as shown in the following example.

\begin{example}
Let $P_0 = \delta_0$, $P_n = (1-n^{-1})\delta_0 + n^{-1} \delta_{x_n}$ and $\Pi(\{P_0\}) = \Pi(\{P_n\}) = 1/2$.
Certainly, the KL condition \eqref{eq:KL-rate} holds for any $\epsilon_n \gg n^{-1/2}$.
Note that the likelihood given $P_n$ equals to $(1-n^{-1})^n$ $P_0$-almost-surely, which converges to $e^{-1} > 0$.
If $x_n = n^{1/(2p) + \delta}$ for small enough $\delta > 0$, then $W_p^p(P_0, P_n) = n^{-(1/2 - p\delta)}$.
Therefore, the posterior distribution is consistent with respect to $W_p^p$, but the rate of convergence is strictly slower than $\sqrt{(\log n)/n}$.
In this example, note that $\int |x|^{2p} dP_n(x) = n^{2\delta p} \rightarrow \infty$, so the posterior moment condition in Theorem \ref{thm:Wp-rate} is not satisfied.
\end{example}


\section{Convergence rates with respect to $W_\infty$}
\label{sec:rates-infty}

Since $W_p(P,Q)$ monotonically increases in $p$, one may define 
  $W_\infty(P, Q) = \lim_{p\rightarrow\infty} W_p(P, Q)$
which, according to \cite{givens1984class}, corresponds to 
\bean
	W_\infty(P,Q) = \inf \Big\{\epsilon > 0: P(A) \leq Q(A^\epsilon), \; \forall A\in \cR \Big\},
\eean
where $A^\epsilon = \{x: |x-y| < \epsilon ~ \text{for some} ~ y \in A\}$ is the $\epsilon$-enlargement of $A$ and $\cR$ is the set of all Borel subsets of $\bbR$.
This representation of $W_\infty$ bears similarities with the \Levy-Prokhorov metric
\bean
	d_P(P,Q) = \inf \Big\{\epsilon > 0: P(A) \leq Q(A^\epsilon)+\epsilon, \; \forall A\in \cR \Big\}
\eean
which metrizes the weak convergence.

The metric $W_\infty$ induces a much stronger topology than the weak topology even in a bounded metric space.
In an unbounded space, if the tail index of two probability measures $P$ and $Q$ are different, then $W_\infty(P, Q)$ is typically infinity.
For example, if $P$ and $Q$ are Student's $t$-distributions with $\nu_1$ and $\nu_2$ degrees of freedom with $\nu_1 \neq \nu_2$, then $W_\infty(P, Q) = \infty$.
Therefore, it is meaningless to study asymptotics with $W_\infty$ in an unbounded space.

In this section, we assume that $P_0$ is supported in the unit interval $[0,1]$, and so are all probability measures in the support of the prior.
Our benchmarking assumption is  $\inf_{x\in [0,1]} p_0(x) \geq c_0$ for some constant $c_0 > 0$, which is a necessary and sufficient condition for that $P_0[W_\infty(\bbP_n, P_0)] \asymp n^{-1/2}$, see \cite{bobkov2019one}.

\begin{theorem} \label{thm:W_infty}
Suppose that $p_0$ is a density on $[0,1]$ and $\inf_{x\in [0,1]} p_0(x) \geq c_0$ for some constant $c_0 > 0$.
Also, assume that the prior $\Pi$ satisfies the KL condition \eqref{eq:KL-rate} for a sequence $\epsilon_n$ with $\epsilon_n \downarrow 0$ and $\epsilon_n \geq \sqrt{(\log n)/n}$ and $\Pi(P([0,1]) = 1) = 1$.
Then, for some constant $K > 0$, 
\be \label{eq:W_inf-rate}
	\Pi\bigg( W_\infty(P, P_0) > K \epsilon_n ~\Big|~ X_1, \ldots, X_n\bigg) \longrightarrow 0
	\quad \text{in probability}.
\ee
\end{theorem}



\section{Examples}
\label{sec:example}

In this section, we consider the posterior moment condition \eqref{eq:bounded-moment} with two examples.
In the first example, we illustrate the idea of a novel approach handling the second moment condition without full technical details.
The approach relies on a special property of gamma distributions.
The second example considers higher order moments, and concrete posterior convergence rates are derived.

Note that \eqref{eq:bounded-moment} holds trivially if the prior satisfies
\be \label{eq:prior-tail}
	\Pi\Big( p: p(x) \leq K' t_p(x) ~ \forall x \Big) = 1
\ee
for some $K'$, where $t_p$ is the density of the Student's $t$ distribution with $p$ degrees of freedom.
Such a prior can be easily constructed by conditioning well-known priors by the event in the left hand side of \eqref{eq:prior-tail}.
Although the prior probability for this event would be close to 1 with most priors and large enough $K$, this conditioning might be unnatural in practice.

\subsection{Mixture of gamma distributions}
\label{ssec:gamma}

Suppose it is required to establish weak consistency alongside a functional constraint; such as
$$\Pi\left(p: \int x^2 p(x)\, dx>K\mid X_1, \ldots, X_n\right)\to 0$$
in probability, for some finite $K>0$, with the prior $\Pi$ on density functions on $(0,\infty)$. 
This would be for establishing $W_1$ consistency. 

With the usual Kullback--Leibler support condition, we write the posterior as
$$d\Pi(p\mid X_1, \ldots, X_n)= \frac{\prod_{i=1}^n\alpha(x_i) \frac{p(x_i)}{p_0(x_i)} d\Pi(p) }
{ \int \prod_{i=1}^n \alpha(x_i)\frac{p(x_i)}{p_0(x_i)} d\Pi(p)   },$$
for some function $\alpha$. Using standard arguments,
the $P_0$-expectation of the numerator over the set $A$  is 
$$\int_A\left(\int \alpha dP\right)^n\,d\Pi(p)$$ and the reciprocal of the denominator is upper bounded, i.e.
$$e^{nd}\,e^{n\int \log(1/\alpha)\,dP_0}\quad \mbox{a.s.}$$
for all large $n$, for any $d>0$; see \cite{schwartz1965bayes} for details about this argument.
Also note 
$$\exp\left\{n\int \log(1/\alpha)\,dP_0\right\}\leq \left(\int \alpha^{-1}\,dP_0\right)^n.$$ 
Hence, if $A=\{p:\int \alpha^{-1}\,dP>K\}$ and we construct the prior $\Pi$ so that
$$\int \alpha^{-1}\, dP>K\Longrightarrow \int \alpha\,dP<\epsilon<\left(\int \alpha^{-1}\,dP_0\right)^{-1},$$
then, a.s. for all large $n$, using the Markov inequality and the Borel--Cantelli lemma,
$$\Pi(A\mid X_1, \ldots, X_n)\leq e^{nd}\,\epsilon^n  \left(\int\alpha^{-1}\,dP_0\right)^n\to 0\quad\mbox{a.s.}$$
We obtain the second moment result by taking $\alpha(x)=1/x^2$ and so we need to ensure for the prior,
for any $\epsilon>0$, there exists a $K<\infty$ such that $\int x^2\,dP>K$ implies $\int x^{-2}\,dP<\epsilon$.

Let $\kappa_0$ be the true second moment and assume we can construct the model $p(x)$ such that for any $\epsilon>0$ there exists a $K>0$ such that
$$\int x^2 p(x)\, dx>K\Longrightarrow \int x^{-2} p(x)\,d x<\epsilon<1/\kappa_0.$$
This also implies $K>\kappa_0$. Such an example arises with the gamma distribution,
so consider 
$p(x)=\Gamma(a)^{-1\,}x^{a-1}e^{-x},$ where
we have $\E\,X^2=a(a+1)$ and $\E\,X^{-2}=[(a-1)(a-2)]^{-1}$. 
Hence, $\E\,X^2>K$ implies $\E\,X^{-2}<\epsilon$ for some suitably large $K$, for $a>2$.


For a more general nonparametric  model, consider
$$p(x)=\sum_{j=1}^M w_j\Gamma(x\mid a_j,b),$$
a mixture of gamma distributions. 
We can assign priors to $M$ and $w$ but to describe the prior for $a=(a_1,\ldots,a_M)$ and $b$ there is no loss in generality in fixing them. Now
$$\E\,X^2=\sum_{j=1}^M w_j\,\frac{a_j(1+a_j)}{b^2}$$
and, if, as assumed, $a_j>2+\delta$, then
$$\,\E\,X^{-2}=\sum_{j=1}^M w_j\frac{b^2}{(a_j-1)(a_j-2)}.$$
To obtain our required condition, we take the prior so that if a single $a_j(1+a_j)/b^2>K$ then it is true for all $j$.
This ensures that if $\E\,X^2>K$ then $a_j(1+a_j)/b^2>K$ for all $j$ and then it is also true that 
$b^2/((a_j-1)(a_j-2))<\xi/K$, for all $j$, for some $\xi<\infty$, which is fixed. Indeed, $\xi=(2+\delta)(3+\delta)/(\delta(1+\delta))$, so $\E\,X^{-2}<\xi/K$.

Hence, we take the prior for $(a,b)$ as
$$\pi(a,b)=\pi(b)\left[(1-q)\prod_{j=1}^M g_{c-}(a_j|b)+q\prod_{j=1}^M g_{c+}(a_j|b)\right],$$
where $g_{c-}$ is a density on $(0,c)$ and $g_{c+}$ a density on $(c,\infty)$ for some $c$. Here,
$g_{c-}(a_j|b)$ puts all the mass on $a_j(1+a_j)<cb^2$ and $g_{c+}(a_j|b)$ puts all the mass on $a_j(1+a_j)>cb^2$.
In practice, we can take $c$ so large that the part of the prior which contributes to the posterior will only be the $g_{c-}$ component.


\subsection{Dirichlet process mixture}
\label{ssec:dpm}

Consider a Dirichlet process mixture prior
\be\label{eq:dpm}
	p(x) = \int \phi_\sigma(x-z) dG(z),
	\quad
	G \sim {\rm DP}(\alpha H),
\ee
where DP$(\alpha H)$ denotes the Dirichlet process with base measure $\alpha H$, $\phi_\sigma(x) = \sigma^{-1} \phi(x/\sigma)$ and $\phi$ is the standard normal density.
In practice, an inverse gamma prior is usually imposed for $\sigma^2$, but we consider a fixed sequence $\sigma=\sigma_n \rightarrow 0$ for technical convenience.
Note that the sequence $\sigma_n$ controls the convergence rate.
Specifically, with a suitable sequence $\sigma_n$, one can prove that $\Pi(\cK_n) \geq e^{-n\epsilon_n^2}$, see \cite{ghosal2007posterior, shen2013adaptive}.
While these papers extensively studied posterior convergence rates with respect to the Hellinger metric, posterior moments have not been studied thoroughly.
Only Section 8 of \cite{ghosal2007posterior} slightly touched the tail mass of the posterior distribution.
However, their result relies on the assumption that $P_0$ is compactly supported, and cannot be directly used to bound the posterior moments.

Note that the posterior moment condition \eqref{eq:bounded-moment} is similar to
\be \label{eq:bounded-tail}
	\Pi\Big( P(B_m) \leq K 2^{-pm} ~ \forall m \geq 0 \mid X_1, \ldots, X_n \Big) \longrightarrow 1 ~ \text{in probability},
\ee
where $B_0 = (-1,1]$ and $B_m = (-2^m, 2^m] \backslash (-2^{m-1}, 2^{m-1}]$ for $m \geq 1$.
Since
\be \label{eq:moment-prob}
	\frac{1}{2^p} \sum_{m \geq 1} 2^{pm} P(B_m) \leq M_p(P) \leq \sum_{m \geq 0} 2^{pm} P(B_m)
\ee
for any probability measure $P$ and $p \geq 1$, \eqref{eq:bounded-moment} is implied by
\bean
	\Pi\Big( P(B_m) \leq K' 2^{-(p+\delta) m} ~ \forall m \geq 0 \mid X_1, \ldots, X_n \Big) \longrightarrow 1 ~ \text{in probability}
\eean
for some positive constants $K'$ and $\delta$.

Suppose that $P_0(B_m) \lesssim 2^{-pm}$ for every $m\geq 0$.
Under the assumption that $\Pi(\cK_n) \geq e^{-n\epsilon_n^2}$, it is not difficult to show that
\bean
	\Pi\Big( |P(B_m) - P_0(B_m)| \gtrsim \epsilon_n \mid X_1, \ldots, X_n \Big) \rightarrow 0 ~ \text{in probability}
\eean
for every $m \geq 0$.
If $2^{-pm} \geq \epsilon_n$, or equivalently $m \leq p^{-1} \log_2 \epsilon_n^{-1}$, then the posterior probability
\bean
	\Pi\Big( P(B_m) \lesssim 2^{-pm} \mid X_1, \ldots, X_n \Big)
\eean
will be close to 1 for large enough $n$ with high $P_0$-probability.
More generally, one can show that
\bean
	\Pi\Big( P(B_m) \lesssim 2^{-pm} ~ \forall m \leq p^{-1} \log_2 \epsilon_n^{-1} \mid X_1, \ldots, X_n \Big) \rightarrow 1 ~~ \text{in probability}.
\eean
If $m > p^{-1} \log_2 \epsilon_n^{-1}$, however, one cannot bound $P(B_m)$ by $2^{-pm}$ because the convergence rate $\epsilon_n$ is larger than $2^{-pm}$.
In this case, the prior must play a role, that is, the prior probability that $P(B_m) \gtrsim 2^{-pm}$ should be small.
In fact, this prior probability should be exponentially small, with an order $e^{-cn\epsilon_n^2}$ for some constant $c> 0$, to guarantee that the posterior probability also decays, cf. Lemma \ref{lem:ghosal2000prior}.
To this aim we will make use of 
\bean
	G(B_m) \sim {\rm Beta} \Big(\alpha H(B_m), \alpha \big(1-H(B_m)\big) \Big)
\eean
for every $m \geq 0$, which in particular implies that the prior expectation of $G(B_m)$ equals $H(B_m)$.
If $H$ is a normal distribution (any $H$ with sub-Gaussian tail would actually work), the prior expectation of $G(B_m)$ is much smaller than $2^{-pm}$ for every large enough $m$.

\begin{theorem} \label{thm:dpm-specific}
Let $H$ be the normal distribution with mean $\mu_H$ and variance $\sigma^2_H$.
Let $\Pi$ be a Dirichlet process mixture prior \eqref{eq:dpm} with $\alpha > 1$ and $\sigma=\sigma_n = n^{-1/5}$. 
Also, suppose that $p_0$ is twice continuously differentiable with a sub-Gaussian tail, and satisfies $\int (p_0''/p_0)^2 + (p_0'/p_0)^4  dP_0< \infty$.
Then, for any $p < 4$, there exists a constant $K$ such that
\bean
	\Pi \Big( M_p(P) \leq K \mid X_1, \ldots, X_n\Big) \rightarrow 1 \quad \text{in probability.}
\eean
\end{theorem}

If the prior and $P_0$ satisfy conditions in Theorem \ref{thm:dpm-specific}, the posterior distribution is consistent with the rate $n^{-2/5}$ with respect to $W_p^p$ for $p < 2$, up to a logarithmic factor.
Once $P_0$ possesses a smoother density, it is possible to prove the consistency of higher order moments, see Lemma \ref{lem:dpm} for more details.

If we impose a prior on $\sigma^2$, it can be deduced from the proof that the assertion of Lemma \ref{lem:dpm} is still valid provided that $\sigma^2$ is bounded a posteriori, that is, $\Pi ( \sigma^2 > K \mid X_1, \ldots, X_n) \rightarrow 0$ for some constant $K > 0$.
Under mild assumptions, the posterior distribution of $\sigma^2$ will be concentrated around 0 unless $P_0$ itself is a location mixture of normal distributions.
If $P_0$ is a location mixture of normal distributions, the posterior probability that $\sigma^2 > \sigma^2_0 + \epsilon$ vanishes, where $\sigma_0^2$ is the true parameter.

\section{Numerical study}
\label{sec:sim}

Although theoretical results given in previous sections provide reasonable sufficient conditions for the Wasserstein consistency, those conditions are not easy to verify in practice.
With a DP mixture prior, for example, the rate $\epsilon_n$ determined by $\Pi(\cK_n) \geq e^{-n\epsilon_n^2}$ plays an important role for the consistency with respect to $W_p$.
However, it is very difficult to find exact rate $\epsilon_n$ satisfying $\Pi(\cK_n) \asymp e^{-n\epsilon_n^2}$.
Note also that if $P_0$ has an unbounded support, the posterior distribution is typically inconsistent with respect to $W_\infty$.
Since $W_p \uparrow W_\infty$ as $p \uparrow \infty$, the posterior distribution will be consistent with respect to $W_p$ only for small values of $p$, where the threshold value depends on $\epsilon_n$.
Perhaps the most interesting cases would be $p=1$ or $p=2$, so in this section, we empirically show that the posterior distribution tends to be consistent with respect to $W_1$ and $W_2$ with popularly used priors.

We consider DP mixtures of Gaussian priors described in Section \ref{ssec:dpm}.
Instead of a decaying sequence $\sigma_n$, we put an inverse gamma prior on $\sigma^2$ as usual in practice.
Specifically, we used $H = N(\mu_H, \sigma_H)$, $\sigma^2 \sim \Gamma^{-1}(\beta,\lambda)$ and $\alpha \sim \Gamma(\beta_\alpha, \lambda_\alpha)$ with $\sigma_H = \beta = \lambda = \beta_\alpha = \lambda_\alpha = 1$ and $\mu_H=0$, where $\beta$ and $\lambda$ denotes the shape and rate parameters of the gamma distribution.
In addition to the location mixture, we also consider a location-scale mixture
\bean
	p(x) = \int \phi_\sigma(x-z) dG(z, \sigma), \quad G \sim {\rm DP}(\alpha H),
\eean
where $H$ is the normal-inverse gamma distribution.
In this case, we used $H = \text{N-}\Gamma^{-1}(\mu_H, \sigma_H, \beta, \lambda)$ and $\alpha \sim \Gamma(\beta_\alpha, \lambda_\alpha)$ with $\sigma_H = \beta = \lambda = \beta_\alpha = \lambda_\alpha = 1$ and $\mu_H=0$, where $(X, Y) \sim \text{N-}\Gamma^{-1}(\mu, \sigma, \beta, \lambda)$ means that $X\mid Y \sim N(\mu, Y/\sigma)$ and $Y \sim \Gamma^{-1}(\beta,\lambda)$.
Note that an inverse gamma distribution has a tail of polynomial order, so with a location-scale mixture, the prior probability that $P(B_m) \geq 2^{-pm}$ may not be too small.

There are several computational algorithms sampling from a posterior distribution based on a Dirichlet process mixture prior, see \cite{neal2000markov, kalli2011slice} and references therein.
Unfortunately, given a posterior sample $P$, it is very difficult to compute the Wasserstein distance $W_p(P, P_0)$, see Theorem 3 of \cite{kuhn2019wasserstein}.
Instead of directly calculating $W_p(P, P_0)$, we can easily generate a Markov chain sample $Y_1, \ldots, Y_N$ from the posterior predictive distribution $\int p(x) d\Pi(p\mid X_1, \ldots, X_n)$.
Then, the corresponding empirical distribution $\widetilde\bbP_N$ can be used as a proxy of the posterior predictive distribution.
Note that the empirical distribution from an ergodic Markov chain, as well as the one from an iid sample, contracts to the stationary distribution with respect to the Wasserstein metrics, see \cite{fournier2015rate}.
However, it is still not easy to compute $W_p(\widetilde\bbP_N, P_0)$.
To evaluate $W_p(\widetilde\bbP_N, P_0)$, we first approximate $P_0$ by a discrete measure $\bbQ_M$ and find $W_p(\widetilde \bbP_N, \bbQ_M)$.
If $M$ is a multiple of $N$, one can easily find exact value of $W_p(\widetilde \bbP_N, \bbQ_M)$ based on the following lemma taken from \cite{bobkov2019one}.

\medskip
\begin{lemma}
For given two collections of real numbers $x_1 \leq \cdots \leq x_N$ and $y_1 \leq \cdots \leq y_N$, let $\bbP$ and $\bbQ$ be the corresponding empirical measures.
Then, for any $p \geq 1$, 
\bean
	W_p^p(\bbP, \bbQ) = \frac{1}{N} \sum_{k=1}^N |x_k - y_k|^p.
\eean
\end{lemma}

To approximate $P_0$ by $\bbQ_M$, assume for a moment that $P_0$ is symmetric about the origin.
For an even integer $M$, let $x_k = q(1/2 + k/M)$ for $k=0, \ldots, M/2-1$ and $\bbQ_M$ be the probability measure such that $\bbQ_M(\{x_0\}) = 2/M$, $\bbQ_M(\{x_k\}) = 1/M$ and $\bbQ_M(\{-x_k\}) = 1/M$ for $k \geq 1$,  where $q:(0,1) \rightarrow \bbR$ is the quantile function of $P_0$.
Then,
\be\label{eq:discrete-approx}
	W_p^p(P_0, \bbQ_M) \leq 
	2\int_{x_{M/2-1}}^{\infty} |x-x_{M/2-1}|^p dP_0(x) + \frac{2}{M/2-1}\sum_{k=1}^{M/2-1} |x_k - x_{k-1}|^p.
\ee
Since $W_p(P_0, \bbQ_M) \rightarrow 0$ as $M \rightarrow \infty$, one can approximate $W_p(\widetilde\bbP_N, P_0)$ by $W_p(\widetilde\bbP_N, \bbQ_M)$.
For a non-symmetric $P_0$, a similar approximation $\bbQ_M$ can be obtained after replacing the origin by the median.
For various true distributions--standard uniform, standard normal, Laplace, Student's $t$ with 20, 10, 5 degrees of freedoms--the approximation error, the upper bound of $W_p(P_0, \bbQ_M)$, is depicted in Figure \ref{fig:discrete-approx}.
When $p=1$ and $p=2$, the approximation of $P_0$ by $\bbQ_M$ is quite accurate for all cases.
On the other hand, for $p=4$ and $p=8$, the approximation is not reliable unless the support of the true distribution is bounded.

\begin{figure}[t]
\centering
\subfigure[Uniform]{\scalebox{0.2}{\includegraphics{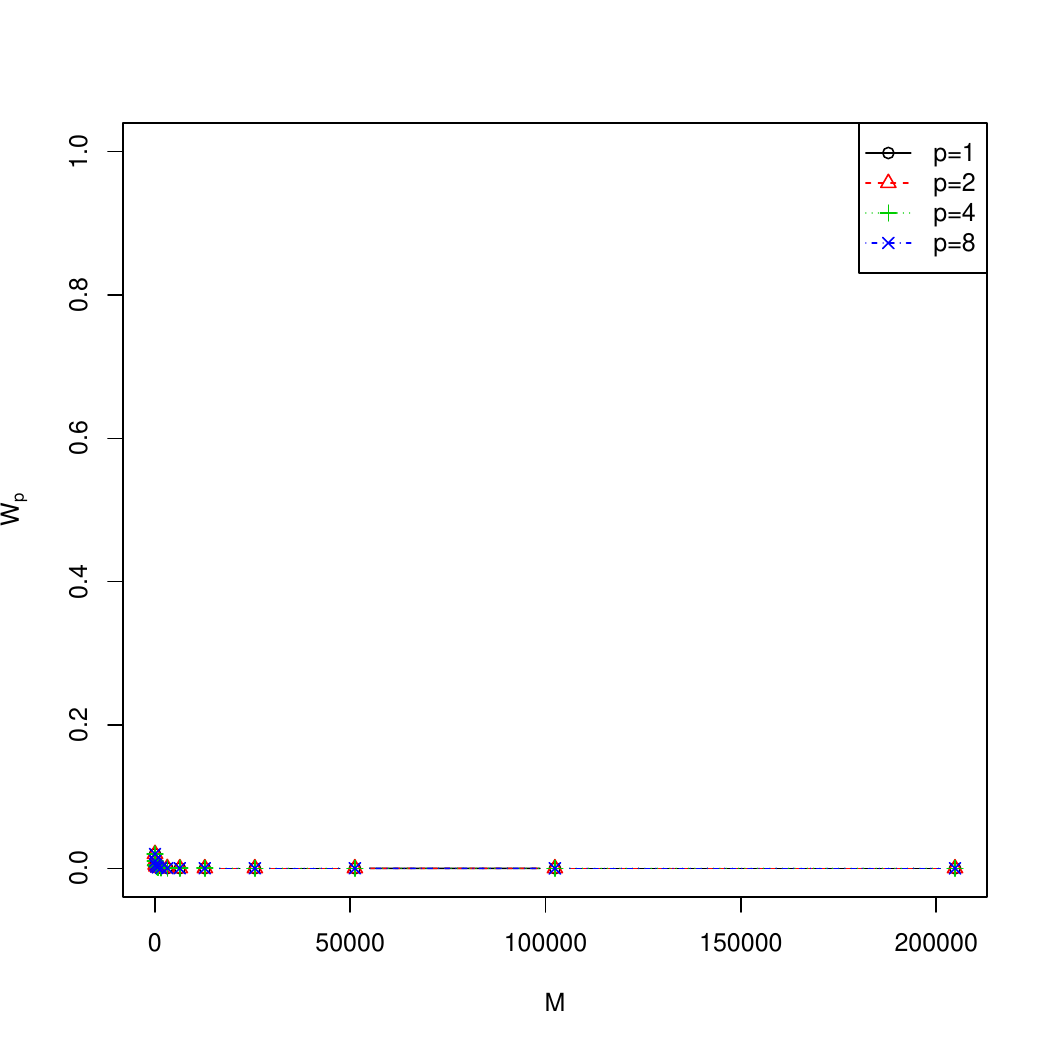}}}
\subfigure[Standard normal]{\scalebox{0.2}{\includegraphics{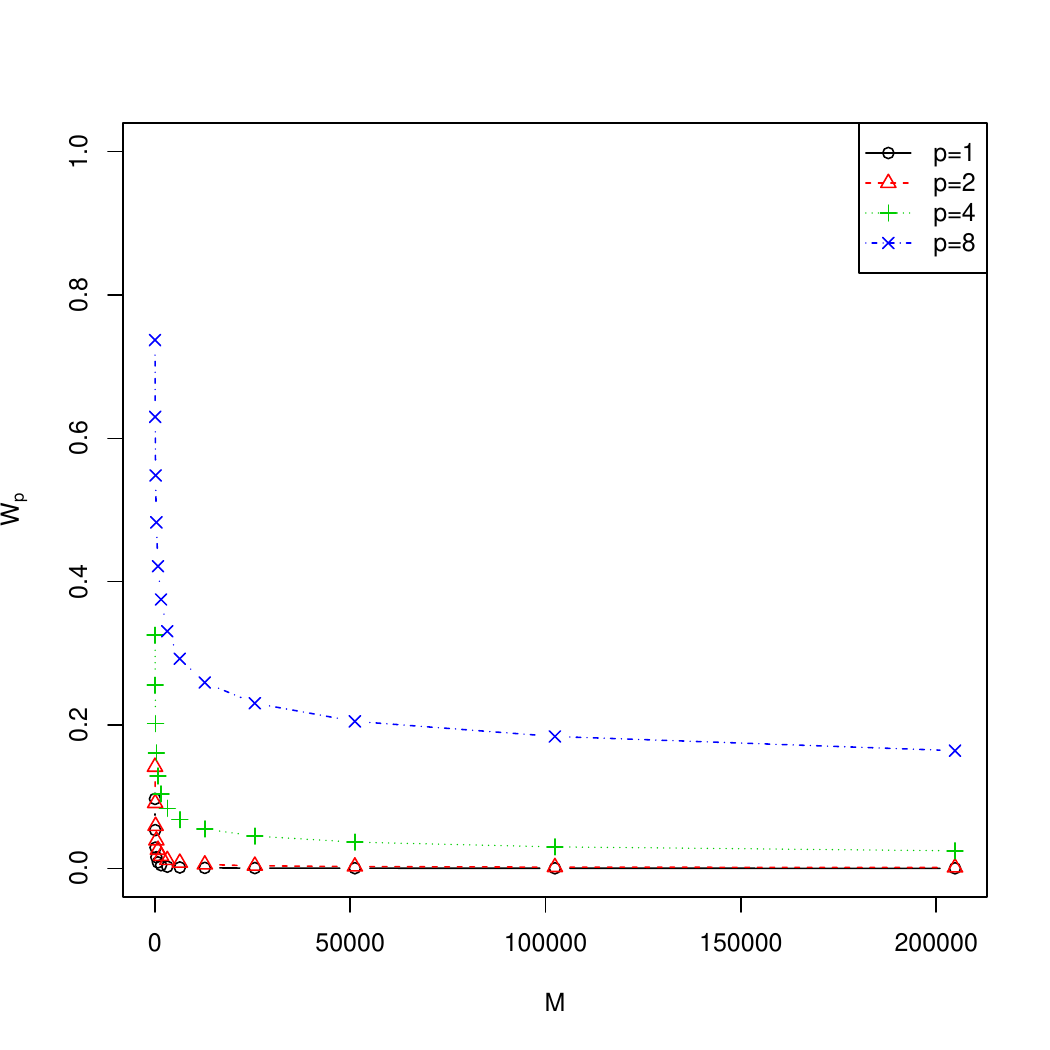}}}
\subfigure[Laplace]{\scalebox{0.2}{\includegraphics{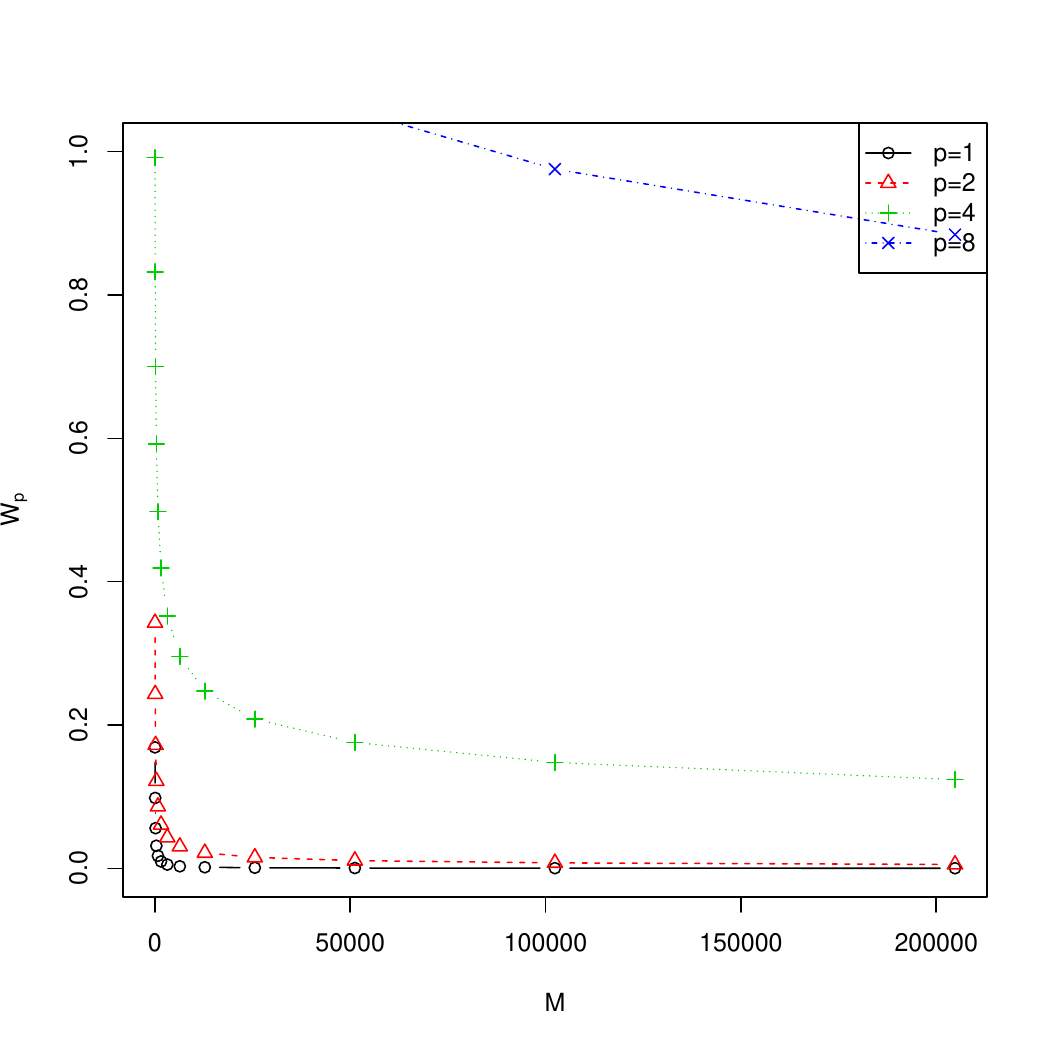}}}
\subfigure[Student's $t$ with 20 df]{\scalebox{0.2}{\includegraphics{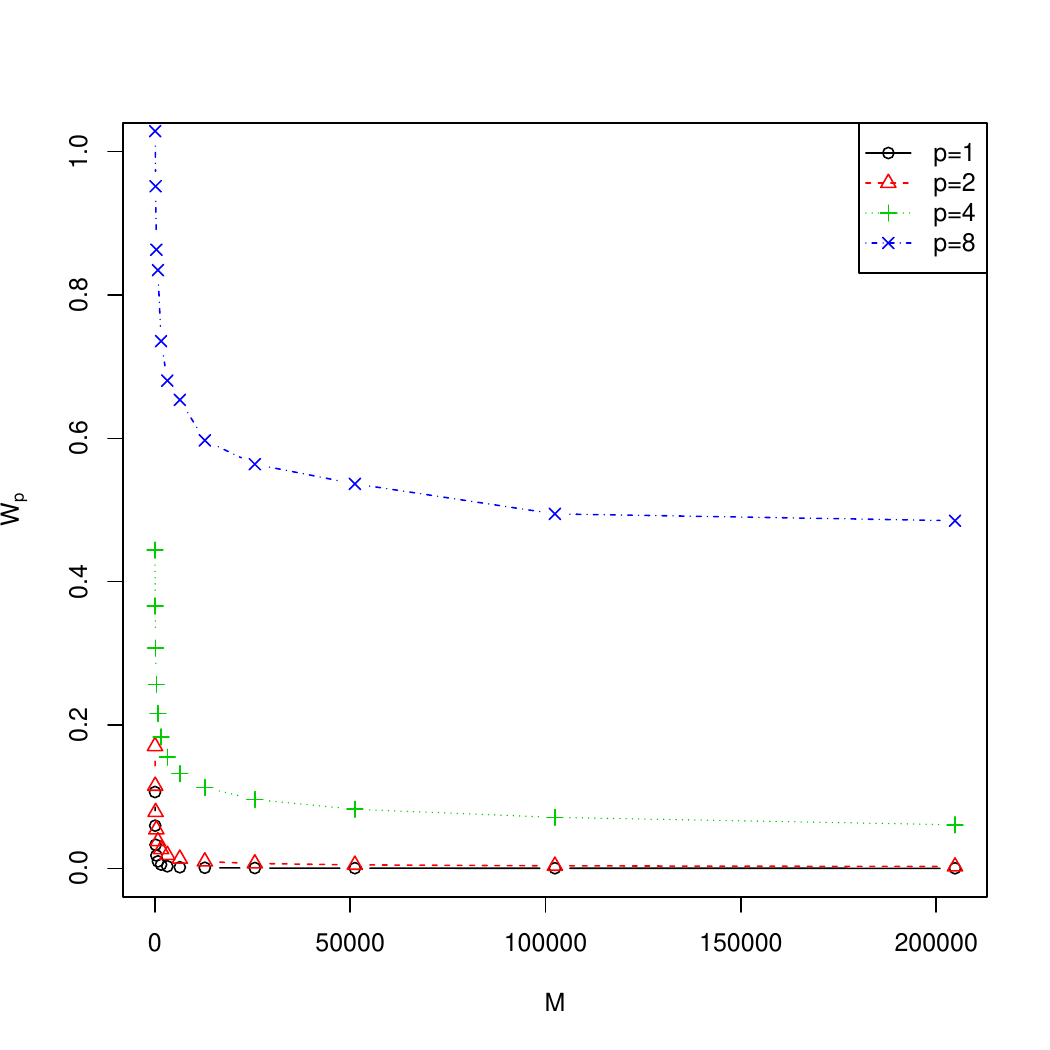}}}
\subfigure[Student's $t$ with 10 df]{\scalebox{0.2}{\includegraphics{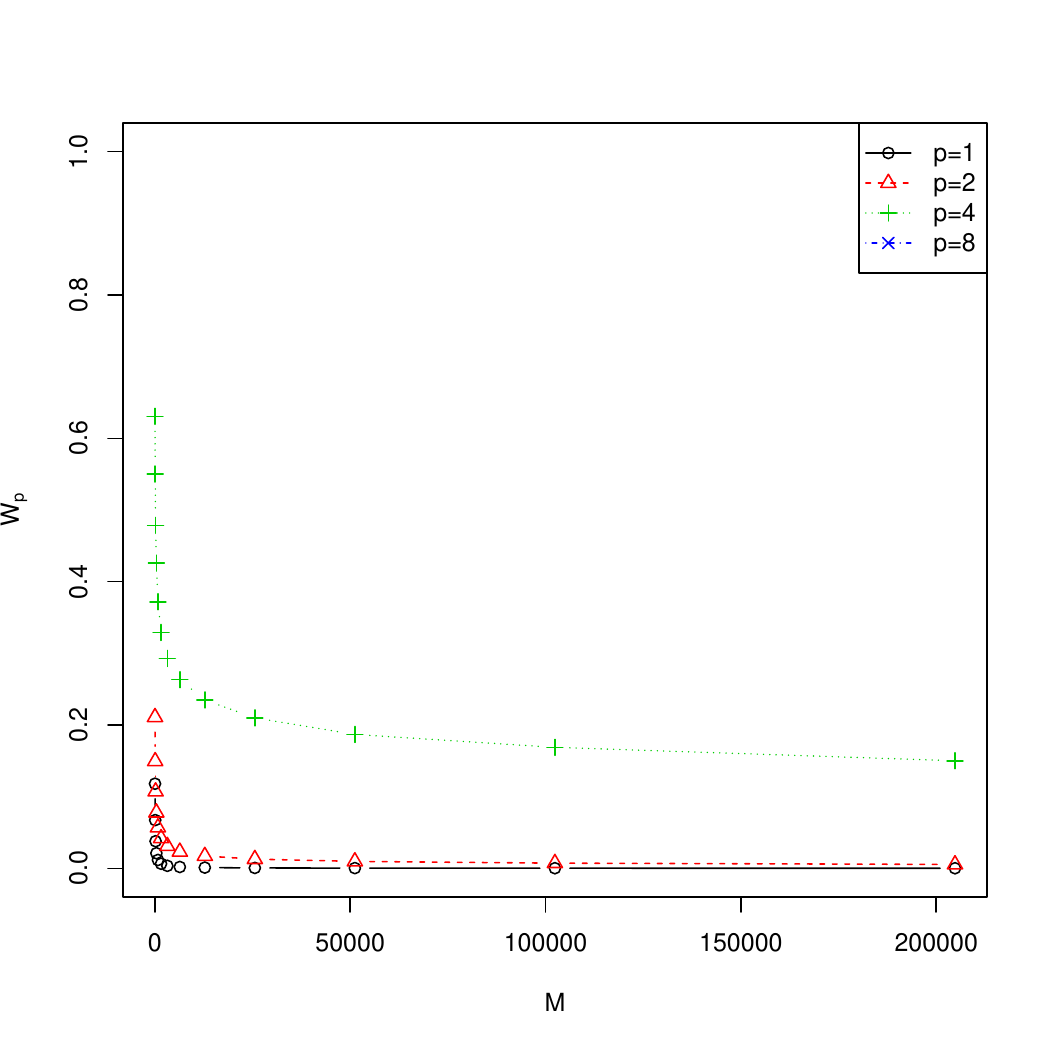}}}
\subfigure[Student's $t$ with 5 df]{\scalebox{0.2}{\includegraphics{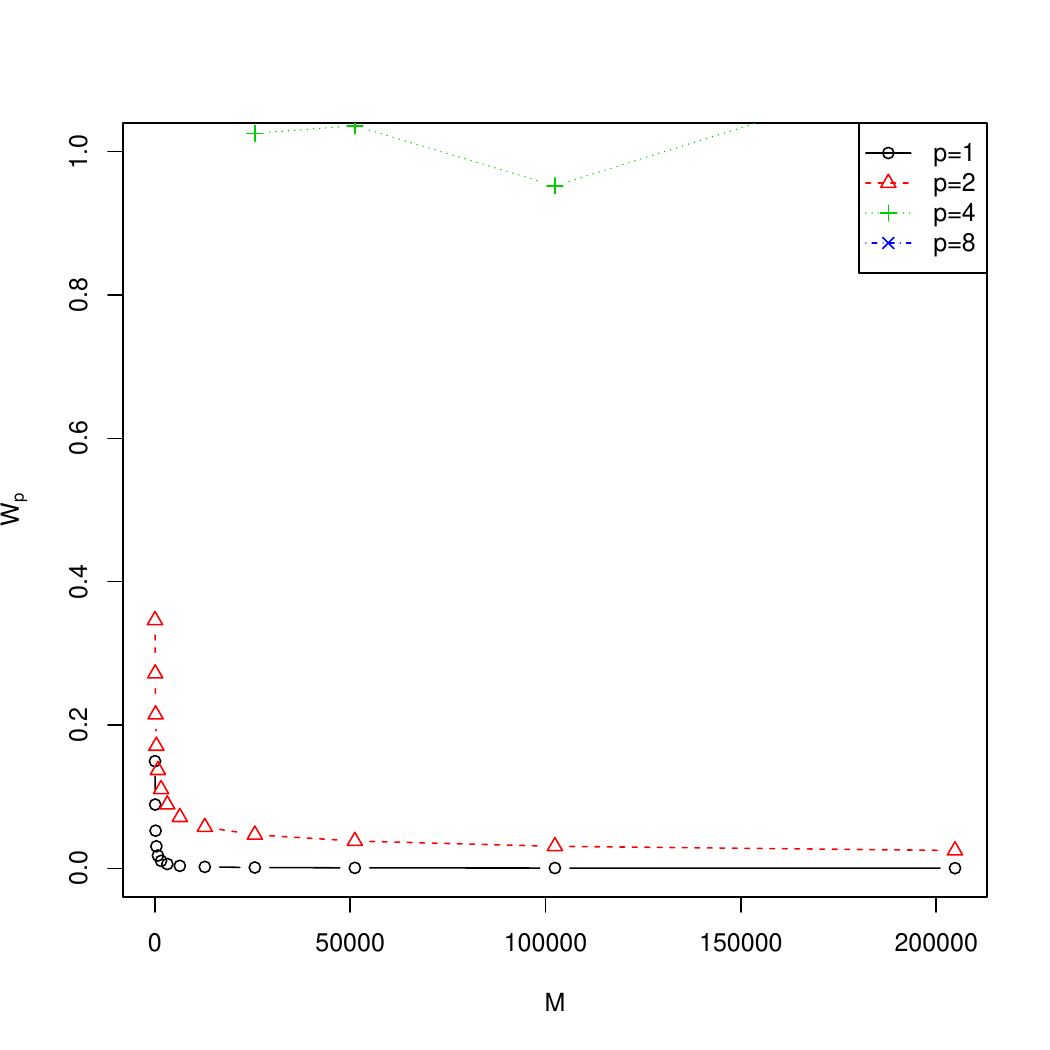}}}
\caption{The upper bound of $W_p(P_0, \bbQ_M)$ for various true distributions.}
\label{fig:discrete-approx}
\end{figure}

With the above six true distributions, we generated $n=50, 100, 200, \ldots, 6400$ samples and obtained $N=10^4$ MCMC samples from the posterior predictive distributions after $1000$ burn-in periods.
Then, we evaluated the Wasserstein distance $W_p(\widetilde \bbP_N, \bbQ_M)$ between the empirical distribution $\widetilde\bbP_N$ of MCMC sample and the discrete approximation $\bbQ_M$ of $P_0$ with $M=2\times 10^5$.
We considered $p=1$ and $p=2$ only because because the approximation by $\bbQ_M$ is not reliable for large $p$.
We repeated the above procedure for 100 times and the median among 100 repetitions are depicted in Figures \ref{fig:W-dpm1} and \ref{fig:W-dpm2}.
As can be seen, the posterior predictive distributions become closer to the approximation $\bbQ_M$ of the true distribution as the sample size increases.
Interestingly, it seems that the location-scale mixture prior also gives consistent posterior distributions with respect to both $W_1$ and $W_2$ for all cases.
Figure \ref{fig:W-dpm3} shows similar results with a location mixture prior with different hyperparameter $H = N(0, 10^4)$.
Note that a normal distribution with large variance is a natural choice for $H$ in practice.
The results in Figure \ref{fig:W-dpm3} shows that the posterior distribution seems to be consistent with respect to $W_2$, but more samples are needed to dominate prior probabilities on the tail.
This is because some posterior predictive samples might be very large when the number of observation is small, and $W_2(\widetilde\bbP_N, \bbQ_M)$ is more sensitive to these large samples than $W_1(\widetilde\bbP_N, \bbQ_M)$.

\begin{figure}[t]
\centering
\subfigure[$P_0=$Uniform, location]{\scalebox{0.2}{\includegraphics{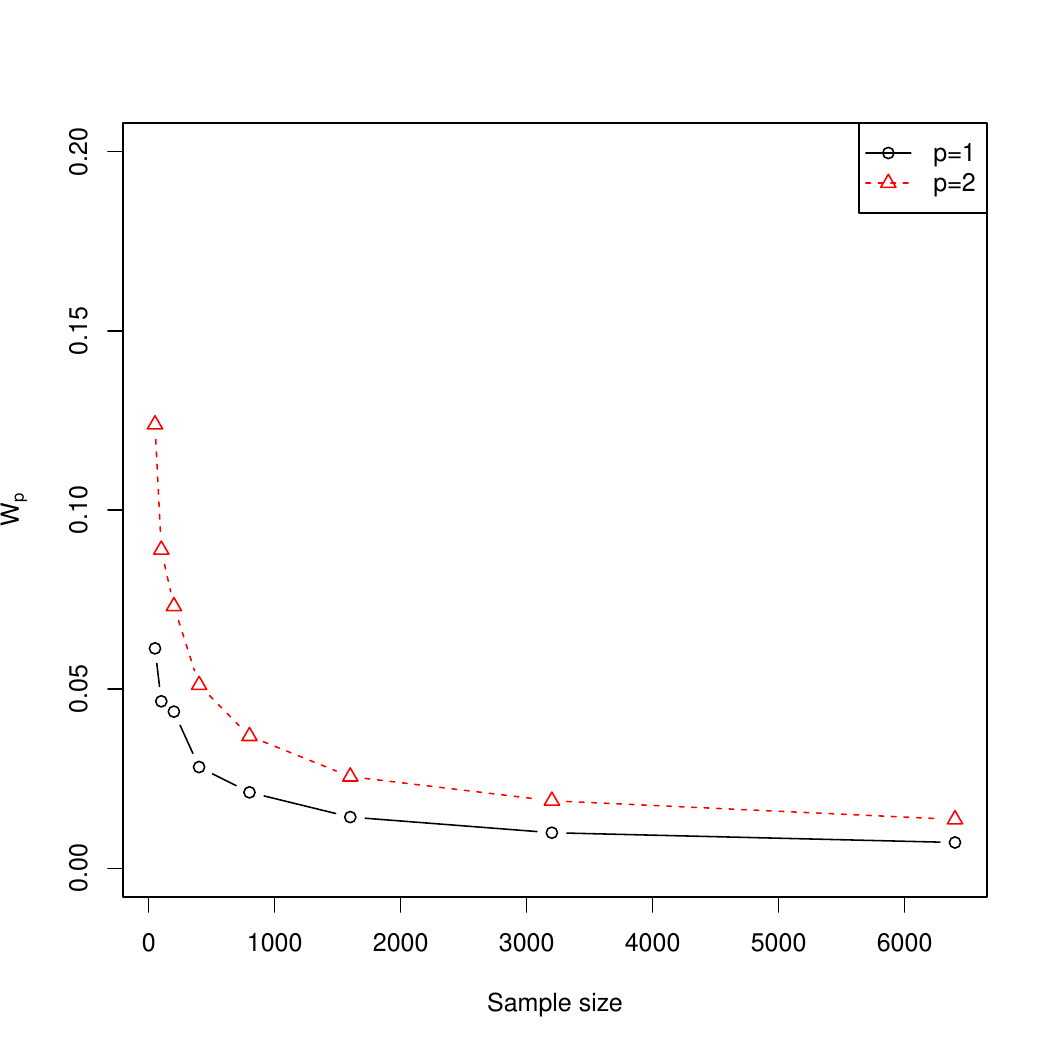}}}
\subfigure[$P_0=$Normal, location]{\scalebox{0.2}{\includegraphics{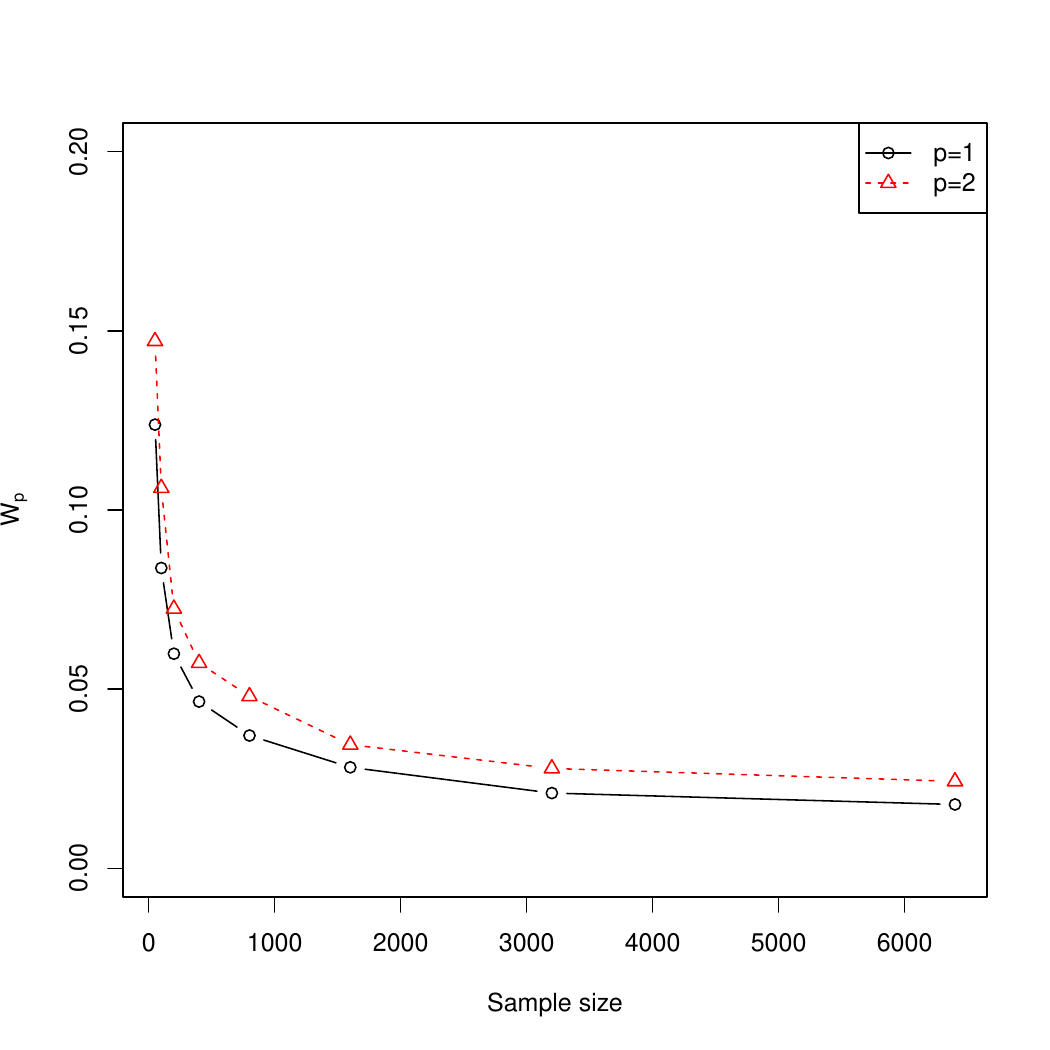}}}
\subfigure[$P_0=$Laplace, location]{\scalebox{0.2}{\includegraphics{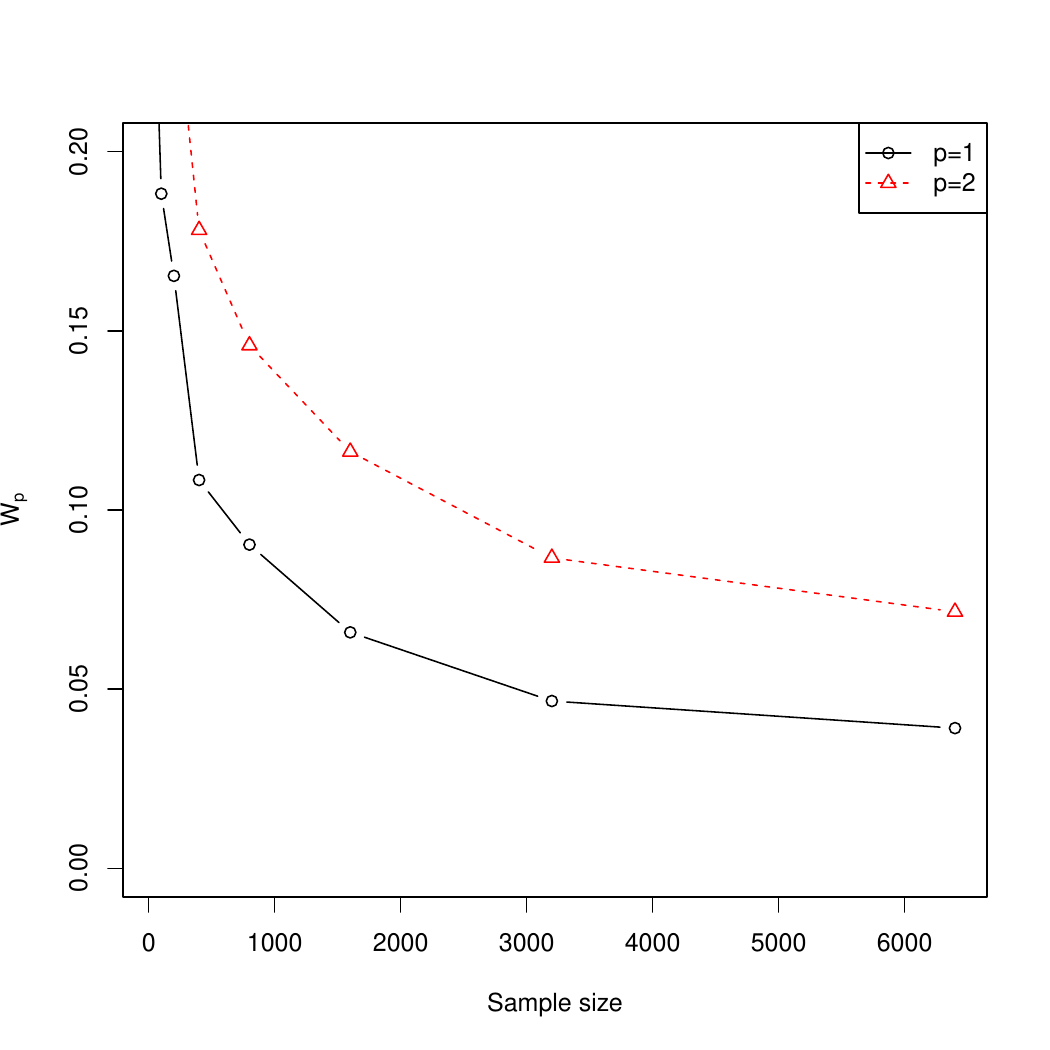}}}
\subfigure[Location-scale mixture]{\scalebox{0.2}{\includegraphics{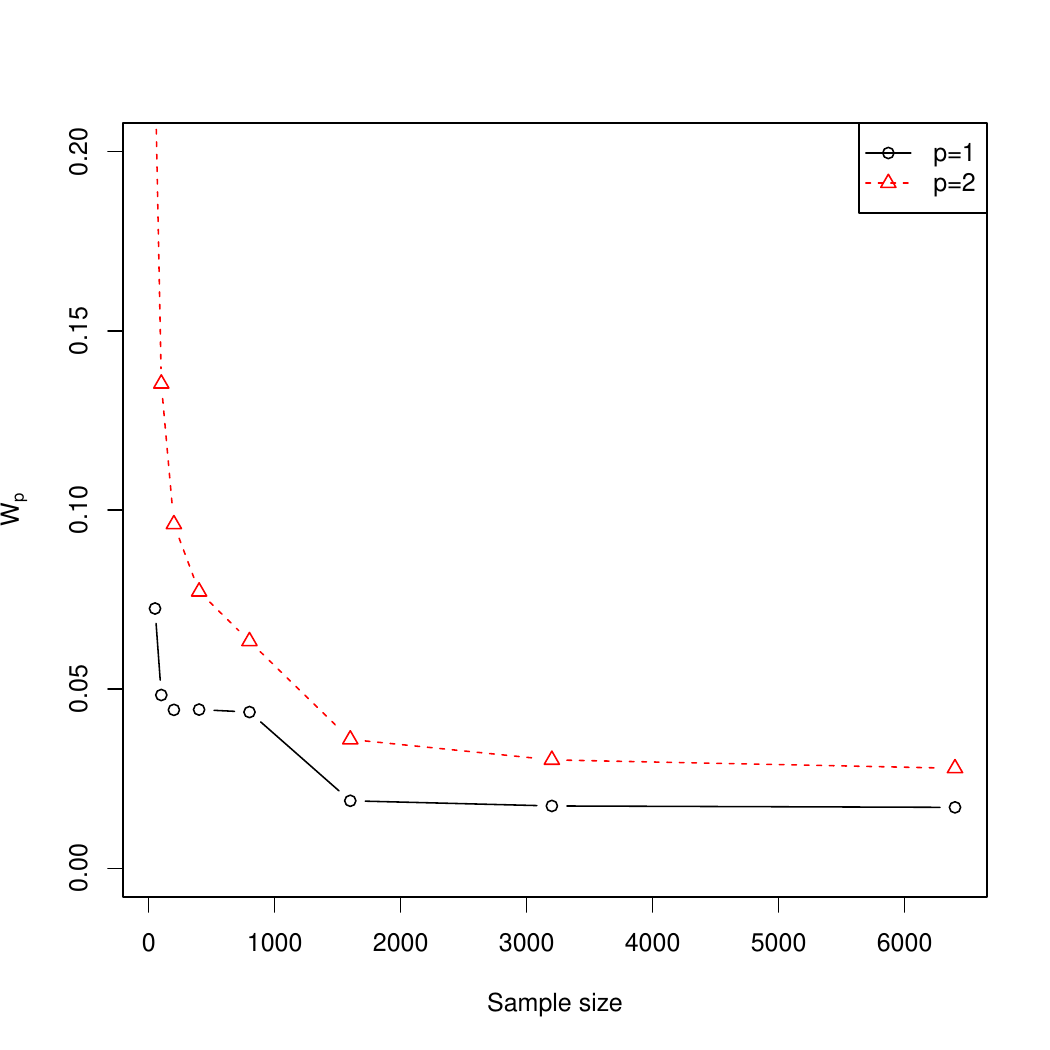}}}
\subfigure[Location-scale mixture]{\scalebox{0.2}{\includegraphics{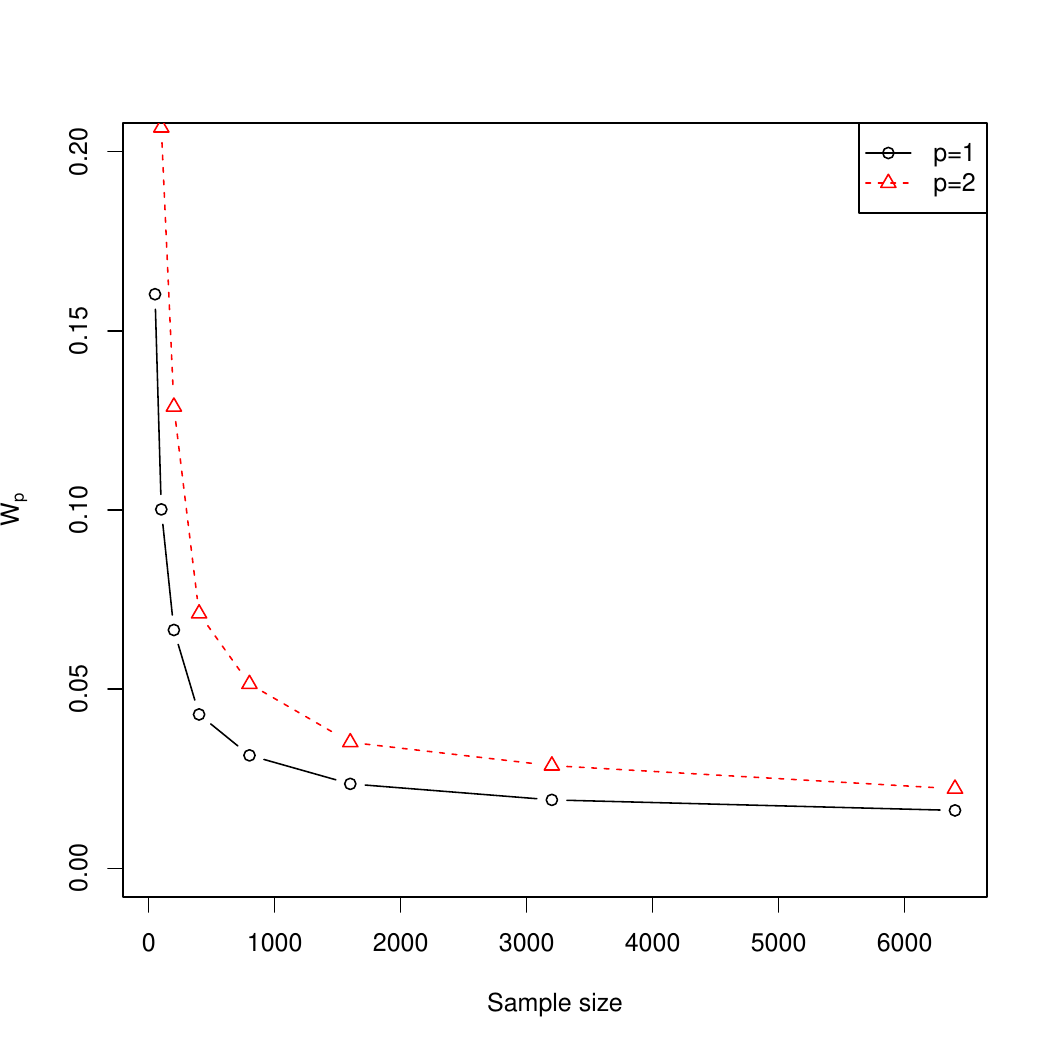}}}
\subfigure[Location-scale mixture]{\scalebox{0.2}{\includegraphics{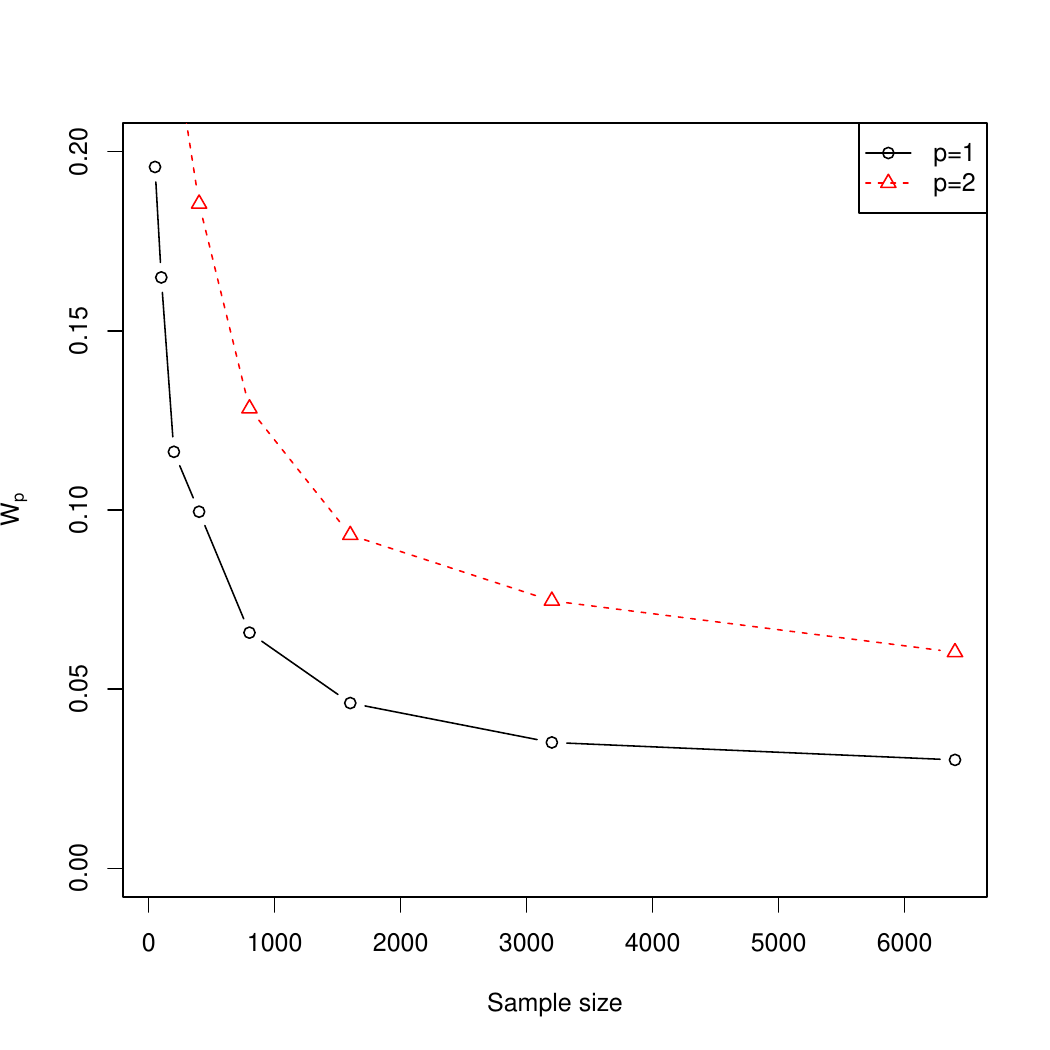}}}
\caption{Wasserstein distances between the true distribution--uniform (left), normal (middle), and Laplace (right)--and the posterior predictive distributions based on location (upper) and location-scale mixtures (lower) of Gaussians.}
\label{fig:W-dpm1}
\end{figure}

\begin{figure}[t]
\centering
\subfigure[$P_0= t$ with 20 df, location]{\scalebox{0.2}{\includegraphics{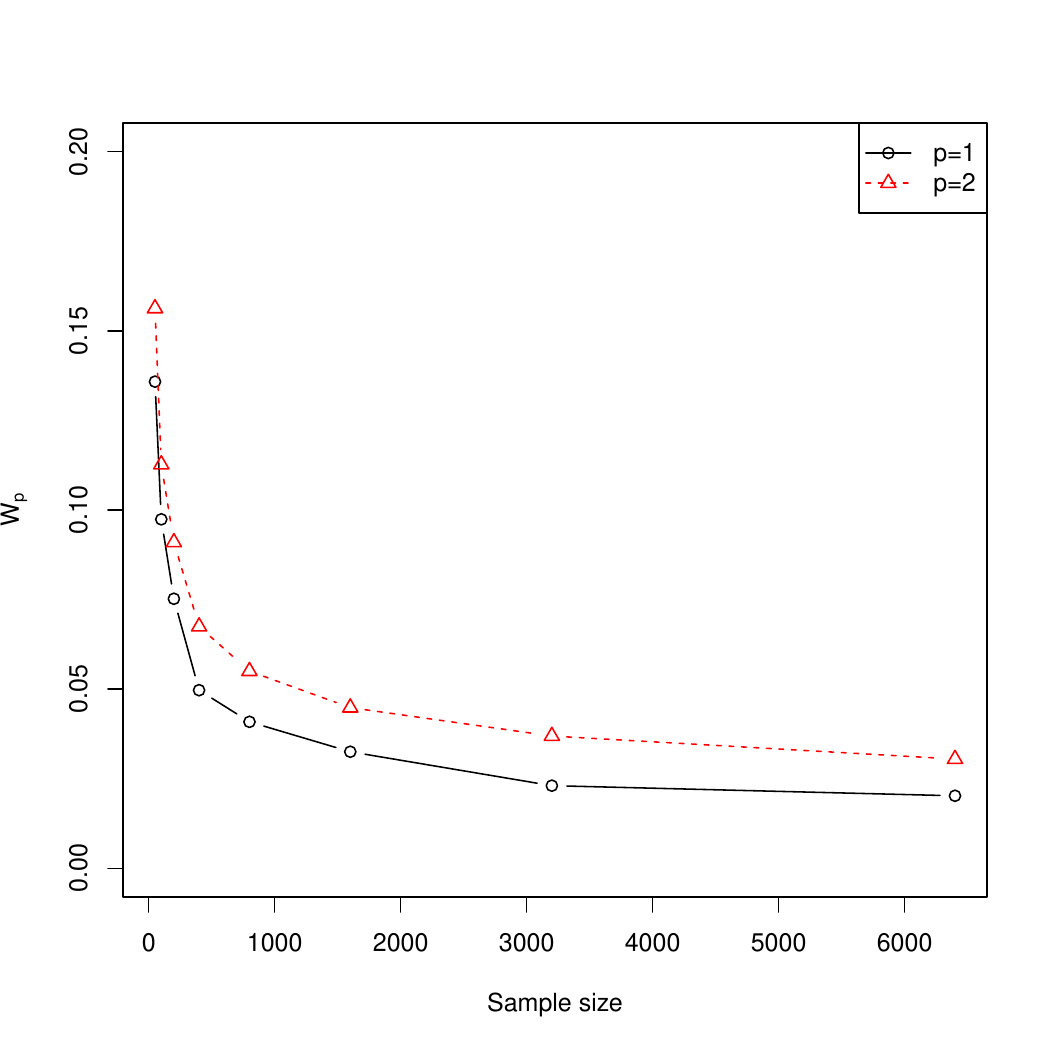}}}
\subfigure[$P_0= t$ with 10 df, location]{\scalebox{0.2}{\includegraphics{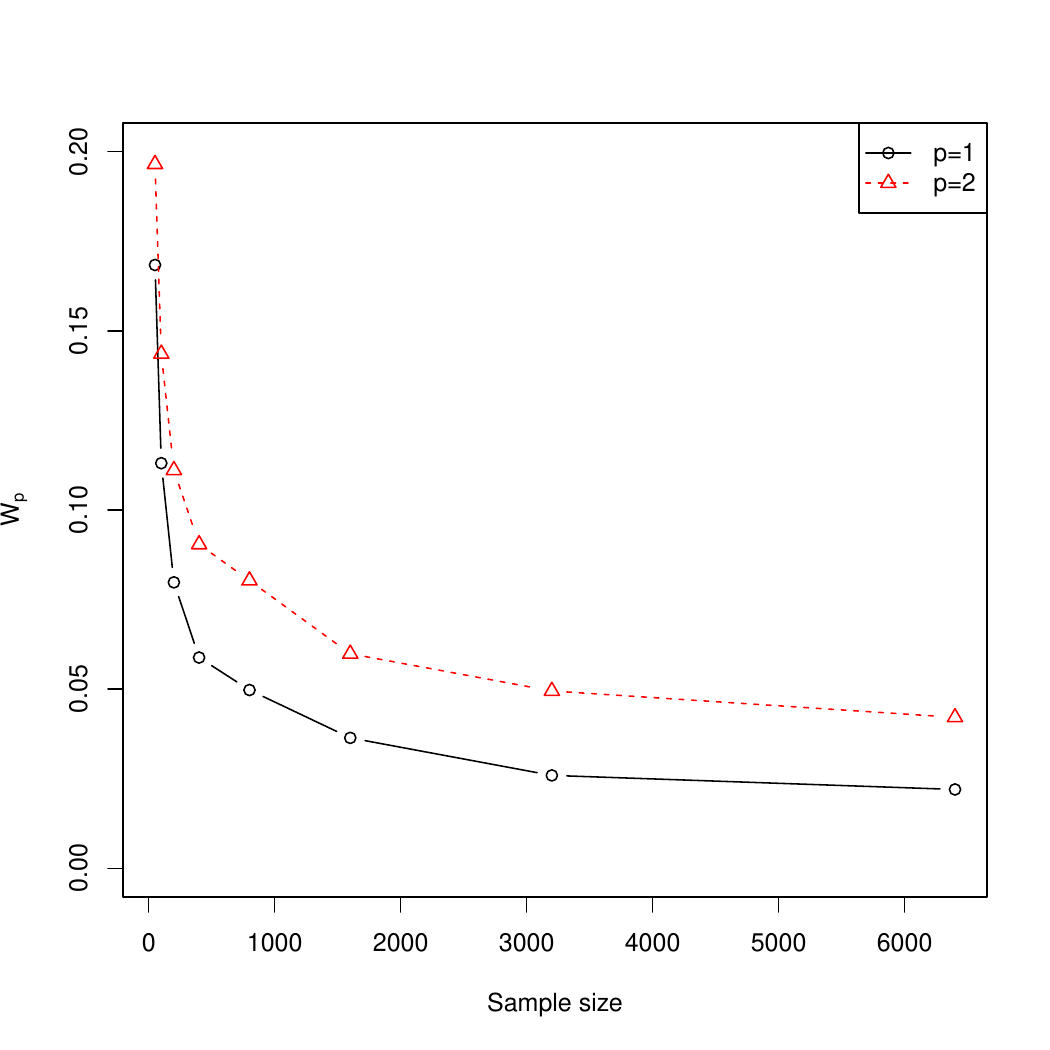}}}
\subfigure[$P_0= t$ with 5 df, location]{\scalebox{0.2}{\includegraphics{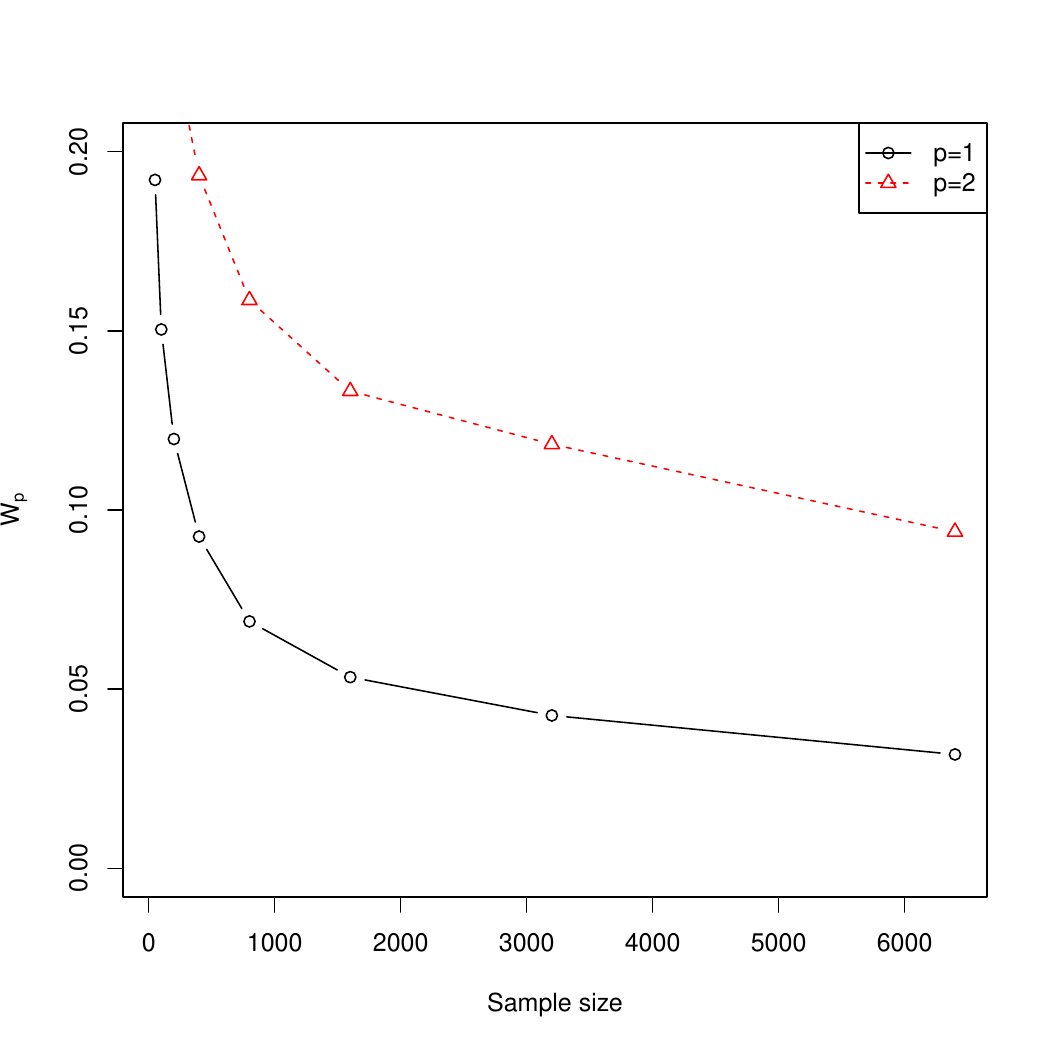}}}
\subfigure[Location-scale mixture]{\scalebox{0.2}{\includegraphics{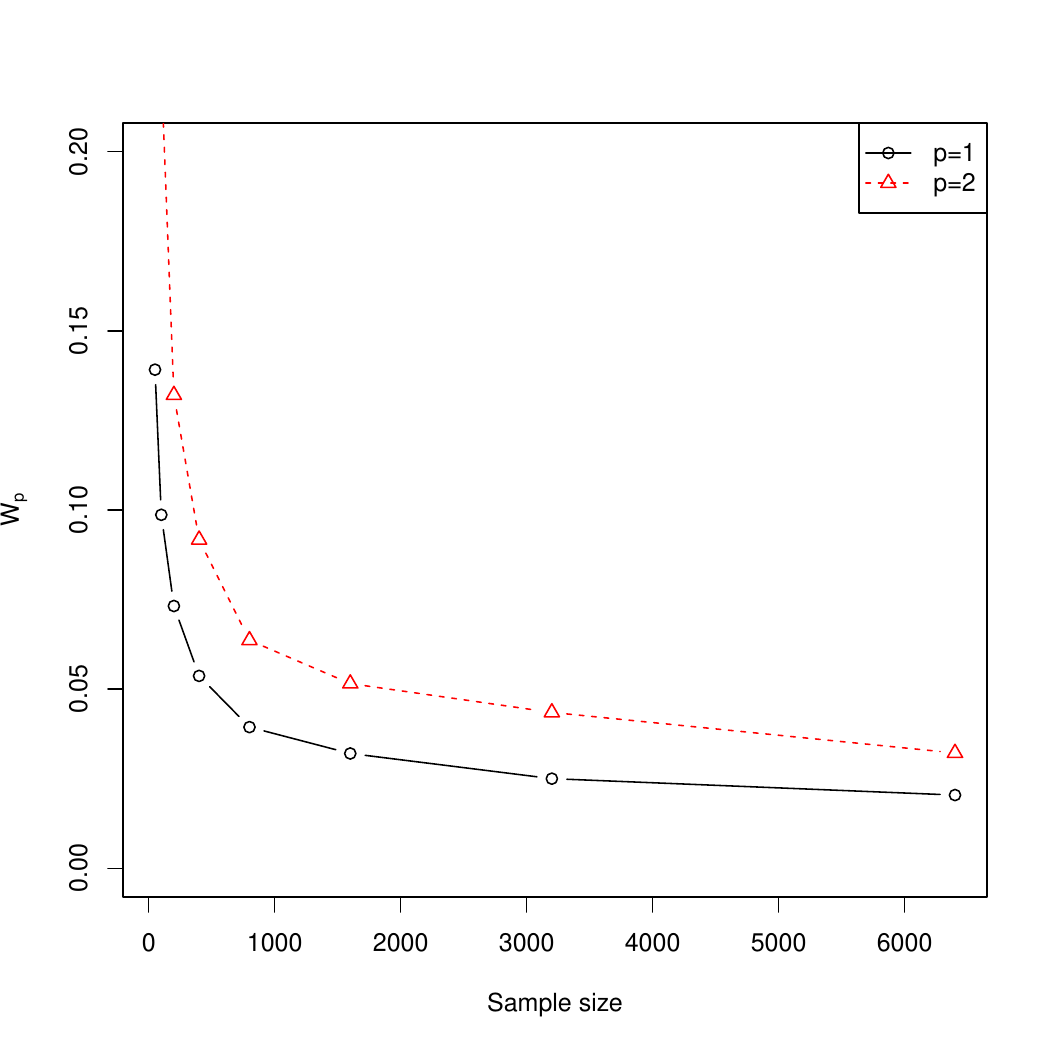}}}
\subfigure[Location-scale mixture]{\scalebox{0.2}{\includegraphics{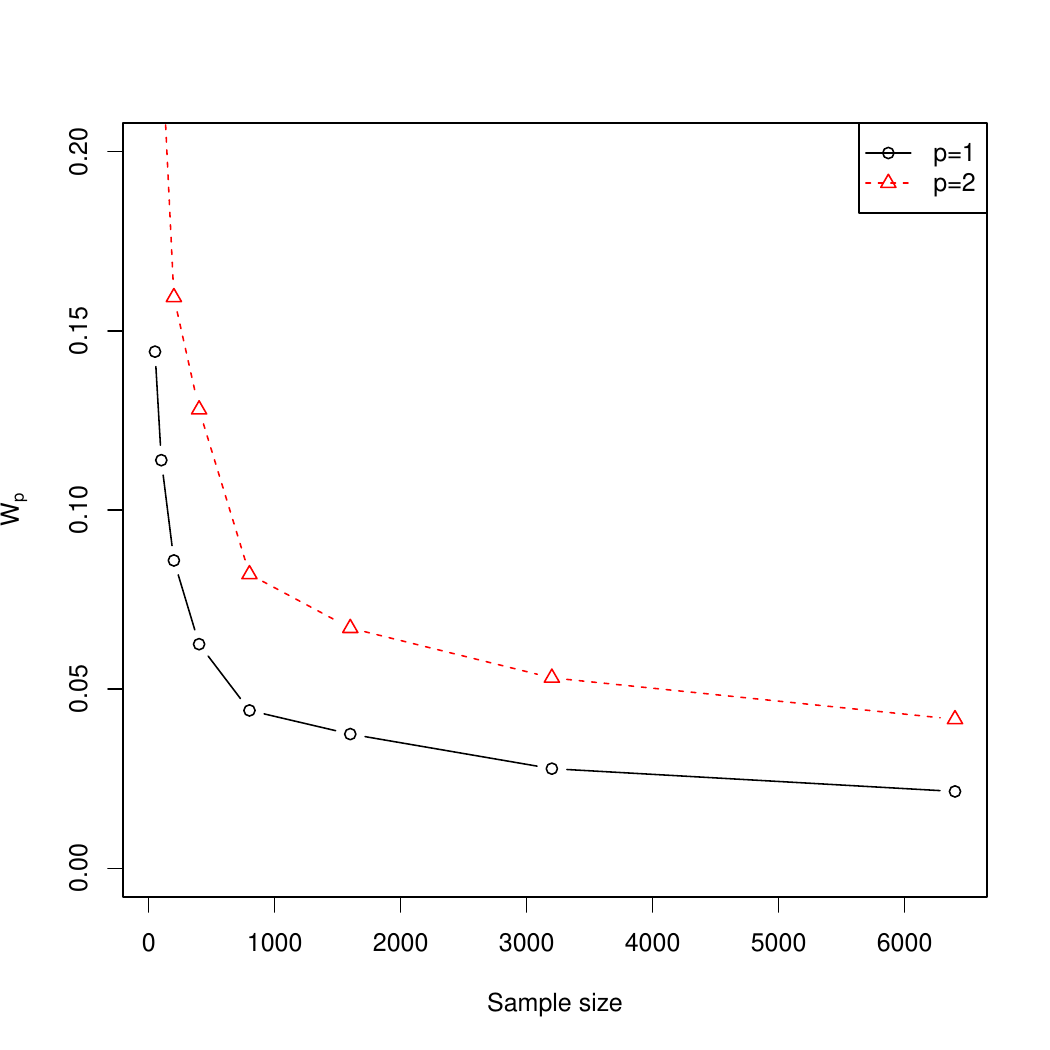}}}
\subfigure[Location-scale mixture]{\scalebox{0.2}{\includegraphics{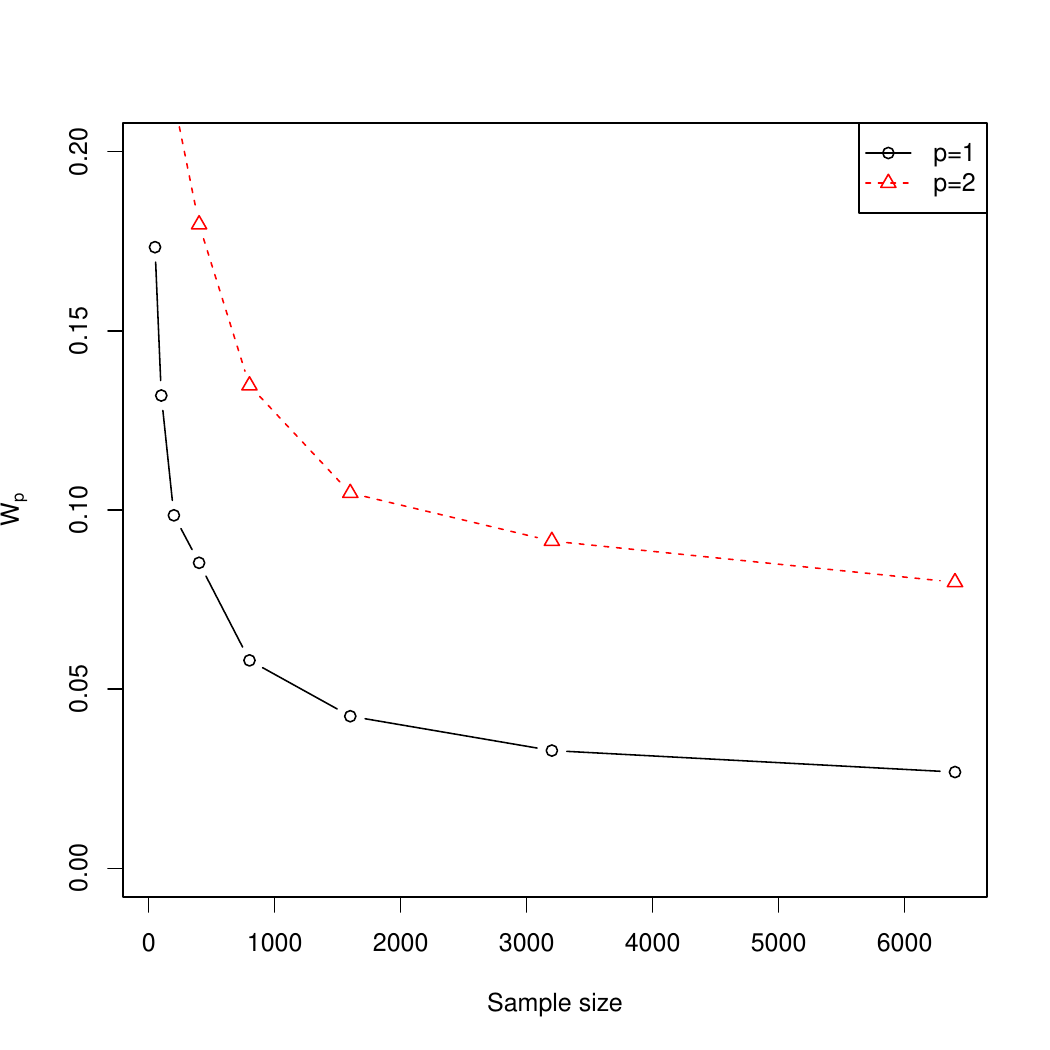}}}
\caption{Wasserstein distances between the true distribution--Student's $t$ distribution with 20 (left), 10 (middle), and 5 (right) degrees of freedom--and the posterior predictive distributions based on location (upper) and location-scale mixtures (lower) of Gaussians.}
\label{fig:W-dpm2}
\end{figure}

\begin{figure}[t]
\centering
\subfigure[$P_0=$ Uniform]{\scalebox{0.2}{\includegraphics{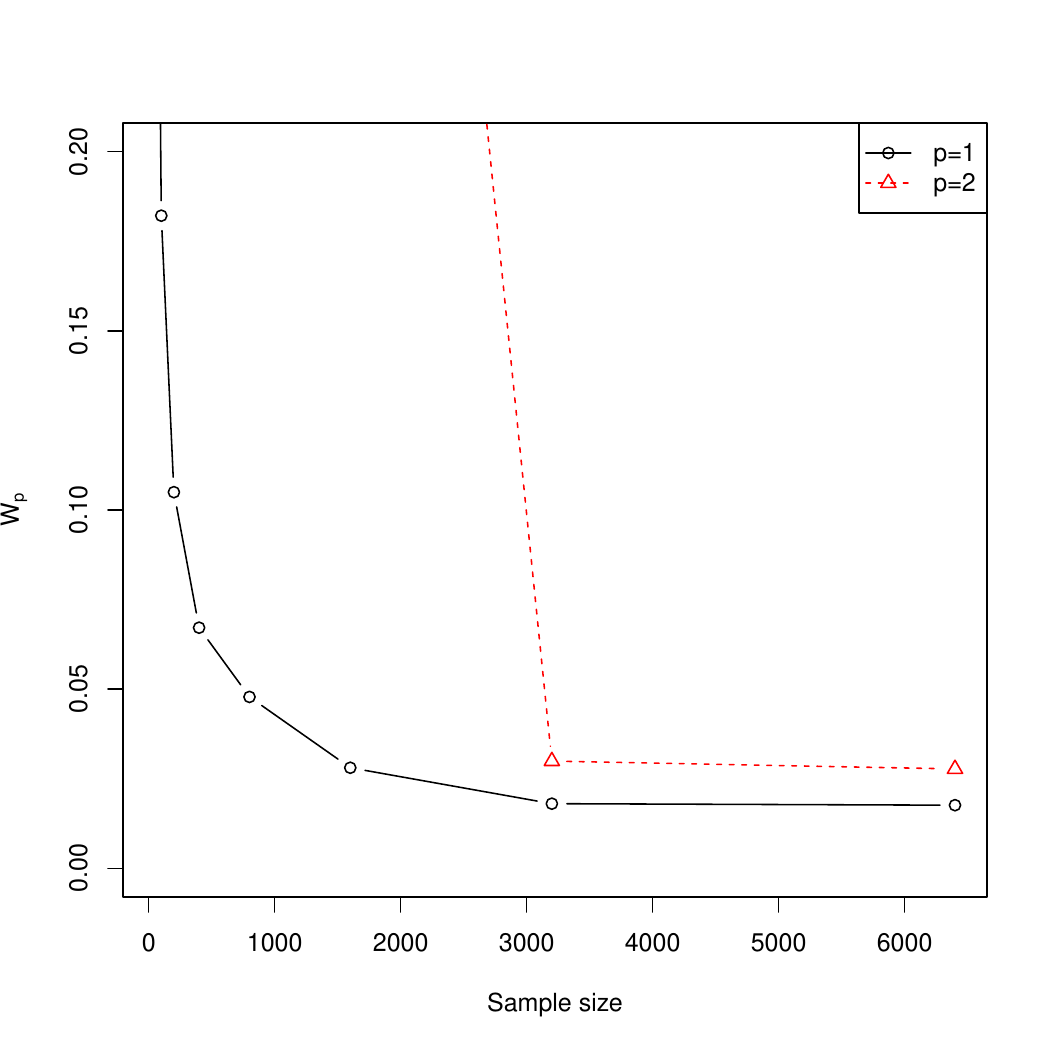}}}
\subfigure[$P_0=$ Normal]{\scalebox{0.2}{\includegraphics{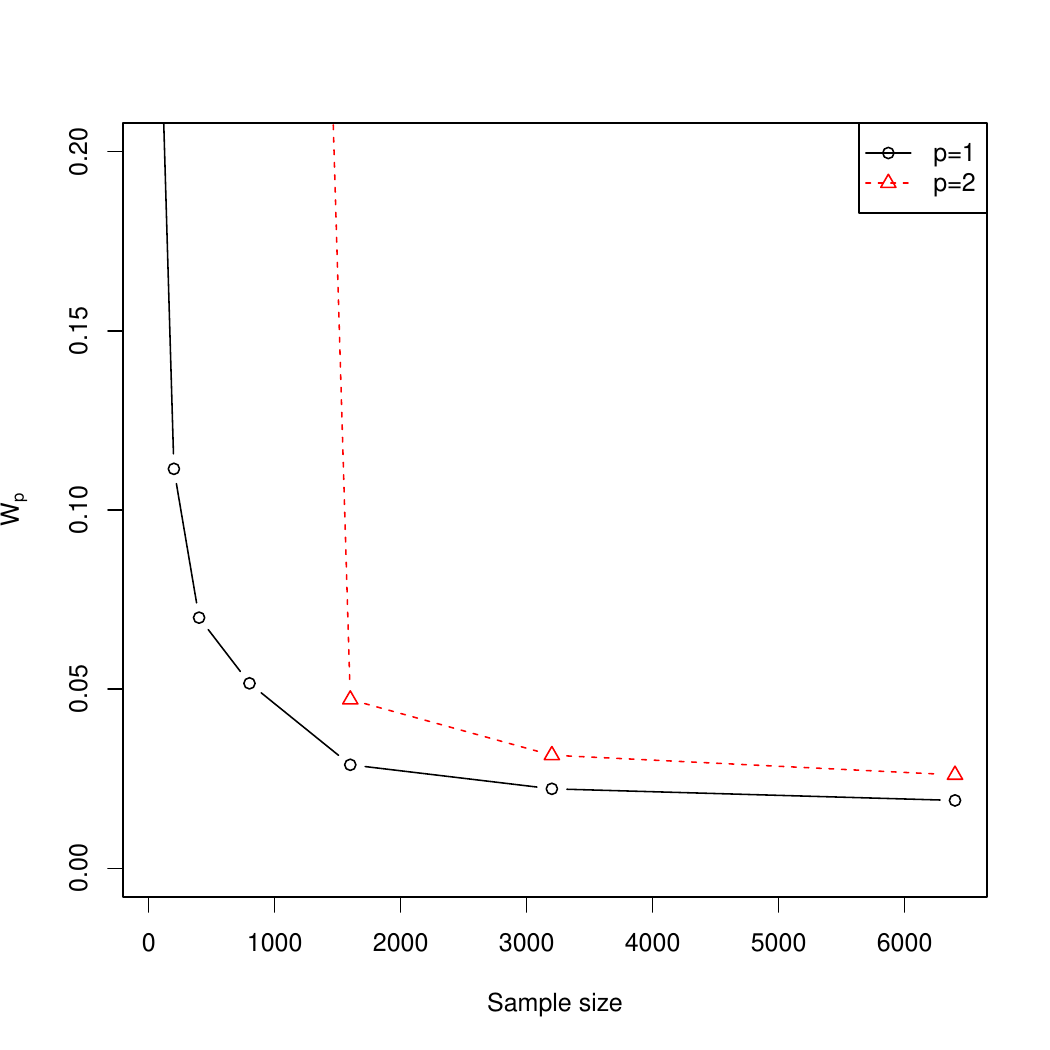}}}
\subfigure[$P_0$ Laplace]{\scalebox{0.2}{\includegraphics{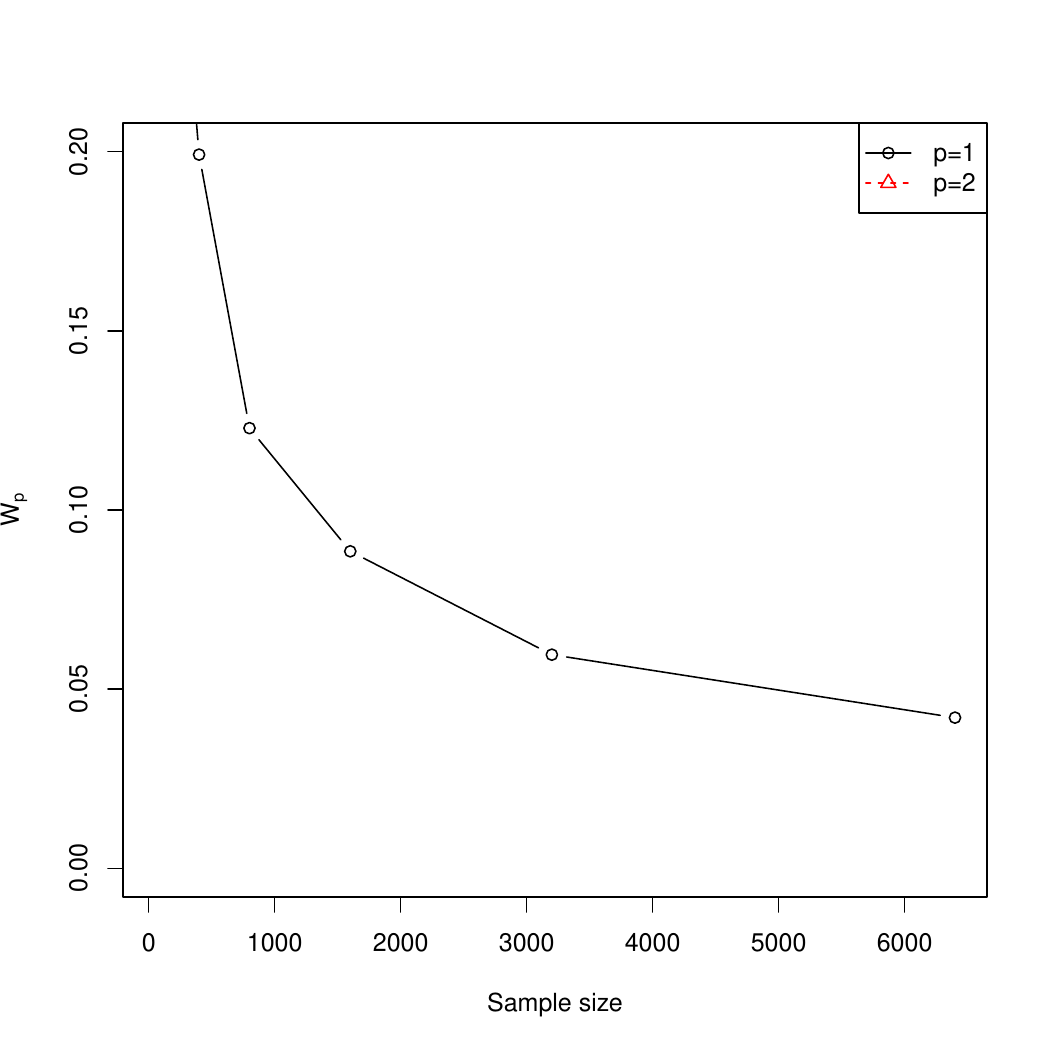}}}
\subfigure[$P_0= t$ with 20 df]{\scalebox{0.2}{\includegraphics{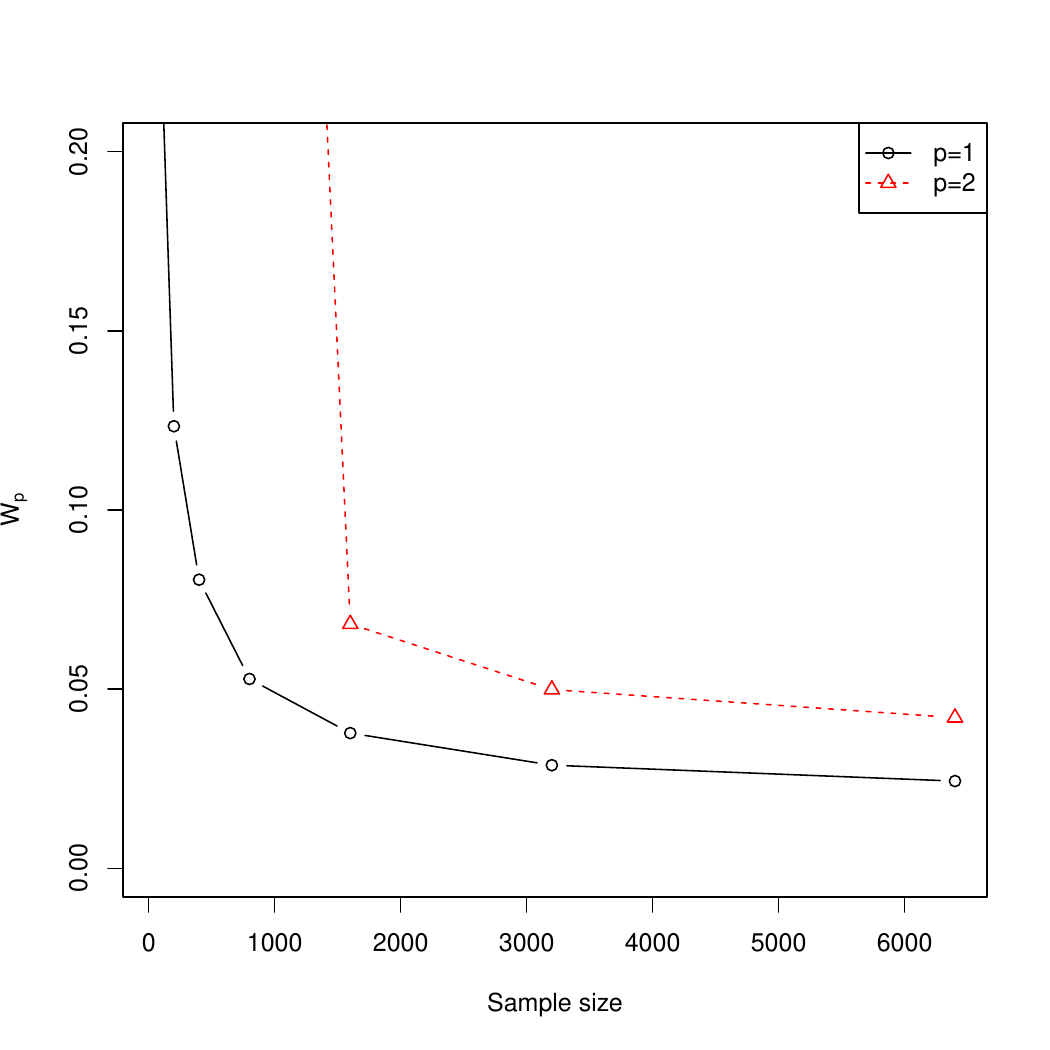}}}
\subfigure[$P_0= t$ with 10 df]{\scalebox{0.2}{\includegraphics{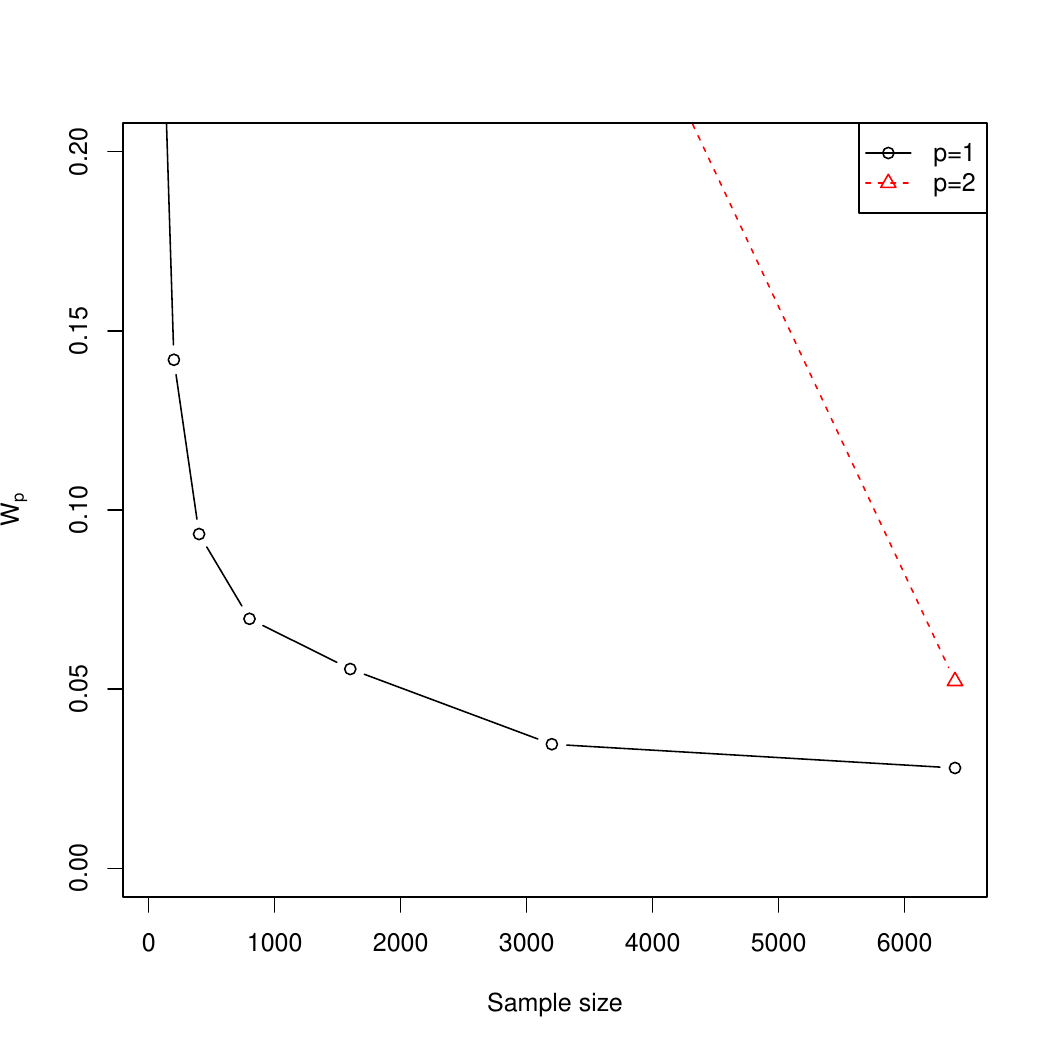}}}
\subfigure[$P_0= t$ with 5 df]{\scalebox{0.2}{\includegraphics{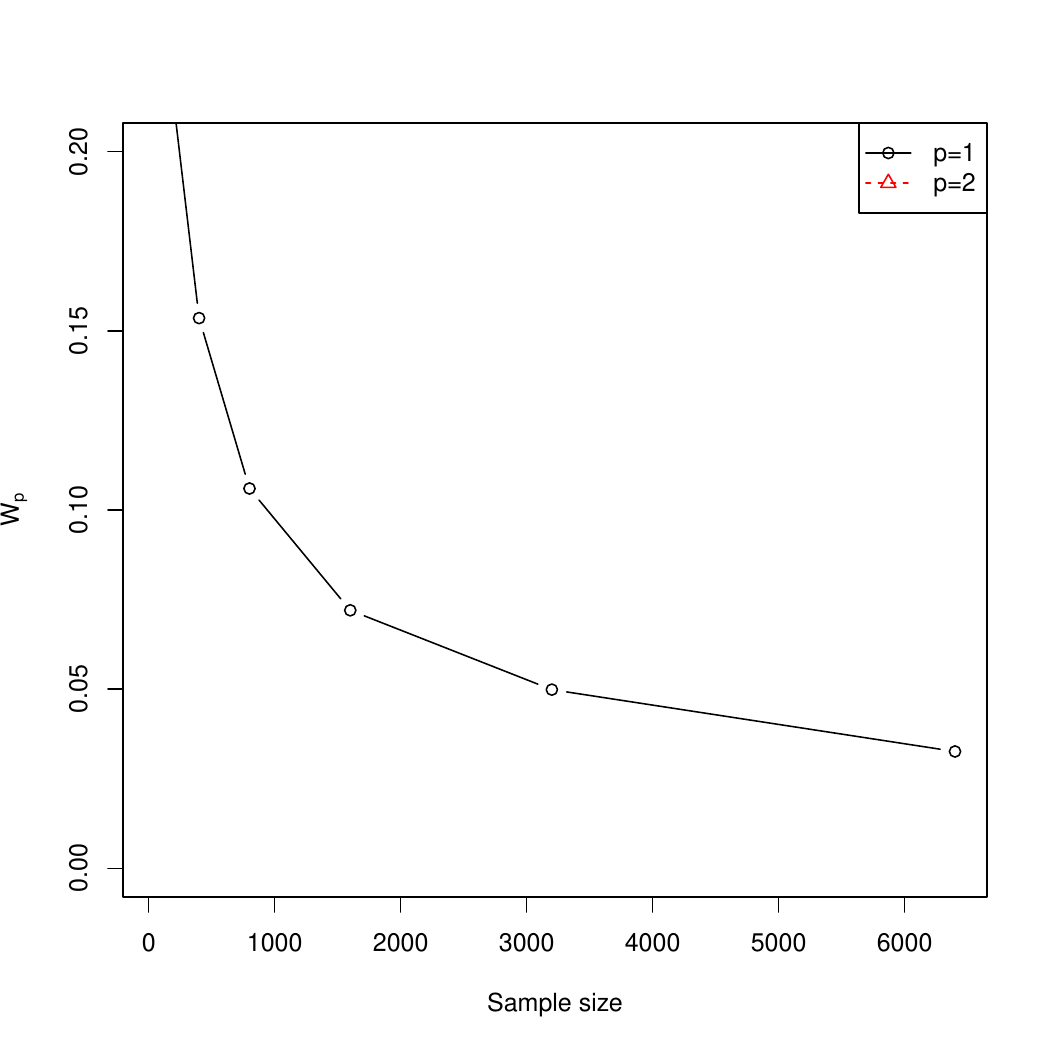}}}
\caption{Wasserstein distances between the true distributions and the posterior predictive distributions based on the location mixture with $H = N(0, 10^4)$.}
\label{fig:W-dpm3}
\end{figure}

\section{Discussion}
\label{sec:conclusion}

In this paper, we provided sufficient conditions for posterior consistency with respect to the Wasserstein metrics and the convergence rate to be $\epsilon_n$ in addition to the well-known KL conditions.
Based on our main theorem, the posterior probability that $W_p^p(P, P_0) \gtrsim \epsilon_n$ vanishes if $M_{2p+\delta}(P)$ is bounded by a constant for some $\delta > 0$ with high posterior probability.
A similar moment condition has been used in \cite{fournier2015rate} to show that $W_p^p(\bbP_n, P_0) \asymp n^{-1/2}$ with high probability.
The moment condition cannot be weakened in general as illustrated in our examples.
Under a stronger condition \eqref{eq:J_p}, which is a necessary and sufficient condition for $W_p(\bbP_n, P_0) \asymp n^{-1/2}$, we conjecture that the posterior probability that $W_p(P, P_0) \gtrsim \epsilon_n$ would vanish.

We note that asymptotic results given in this paper might be utilized to obtain posterior consistency and its convergence rate with respect to strong metrics such as the total variation.
For this, the key is to obtain posterior convergence rate in the Wasserstein metric and bound the total variation between smooth densities by a power of the Wasserstein metric.
More precisely, if $P$ and $Q$ possess smooth Lebesgue densities $p$ and $q$, one can prove that $\|p-q\|_1 \lesssim W_p^\alpha (P, Q)$ for some $\alpha > 0$, see \cite{chae2020wasserstein} for a sharp inequality.
This is a certain reverse inequality because the total variation generates a stronger topology than $W_p$ in the space of all probability measures on a bounded metric space.
This kind of reverse inequality and related theory for posterior consistency are the main motivation of the present paper, which was firstly considered in \cite{chae2017novel}.


We conclude by discussing an example where total variation consistency and a mild condition implies Wasserstein consistency. This is a non--trivial finding. For a given kernel density function $k$ on $\bbR$, consider a location mixture
\be \label{eq:location-mixture}
	p(x) = \int k(x-z) dG(z),
\ee
which is often called a convolution.
A prior $\Pi$ on $p$ can be induced from a prior on the mixing distribution $G$.
With slight abuse of notation, we use the same notation $\Pi$ for the prior of $G$.
Suppose that the true distribution is also of the form \eqref{eq:location-mixture}, that is $p_0(x) = \int k(x-z) dG_0(z)$ for some probability measure $G_0$.
In this case, posterior consistency with respect to the total variation automatically implies the consistency in $W_2$.
Suppose that $k$ is symmetric about the origin, $\int x^2 k(x) dx < \infty$, and that $\tilde k(t) \neq 0$ for every $t\in \bbR$, where $\tilde k$ is the Fourier transform of $f$ defined as $\tilde k(t) = \int e^{-itx} k(x) dx$.

\begin{theorem} \label{thm:convolution}
Suppose that the kernel $k$ satisfies the assumption described above.
Assume also that $M_2(G_0) < \infty$ and $\Pi( M_2(G) < \infty) = 1$.
Then, $\E\Pi( \|p - p_0\|_1 > \epsilon \mid X_1, \ldots, X_n) \to 0$ for every $\epsilon > 0$ implies that
\bean
	\E \Pi\Big( W_2(P, P_0) >\epsilon 
	\mid X_1,\ldots,X_n \Big) \to 0
\eean
for every $\epsilon > 0$.
\end{theorem}

Note that the condition $\Pi( M_2(G) < \infty) = 1$ is easily satisfied for well-known priors.
For example, if we put a Dirichlet process prior for $G$, the tail of $G$ is much lighter than that of its mean, see \cite{doss1982tails}.
Posterior consistency with respect to the total variation can also be easily established using a standard technique.

\section{Proofs}
\label{sec:proof}

Firstly, we introduce a basic set-up which is taken from \cite{fournier2015rate} with slight modification, see also \cite{dereich2013constructive, weed2019sharp}.
For nonnegative integers $l$, let $\cP_l$ be the natural partition of $(-1,1]$ into $2^l$ translations of $(-2^{-l}, 2^{-l}]$, that is, 
\bean
	\cP_l = \Big\{ (-1 + k 2^{1-l}, -1 + (k+1) 2^{1-l}]: k=0, 1, \ldots, 2^l-1 \Big\}.
\eean
Let $B_0 = (-1,1]$ and $B_m = (-2^m, 2^m] \backslash (-2^{m-1}, 2^{m-1}]$ for $m \geq 1$.
Let $\pi_m: \bbR \rightarrow \bbR$ be the function defined as $\pi_m(x) = x / 2^m$, and $\cR_{B_m}P$ be the probability measure on $(-1,1]$ defined as the $\pi_m$-image of $P|_{B_m} / P(B_m)$, that is, for any Borel set $F \subset (-1,1]$,
\bean
	\cR_{B_m}P(F) = \frac{P(\pi_m^{-1}(F) \cap B_m)}{P(B_m)}
\eean
or $\cR_{B_m}P(F)=0$ according as $P(B_m) > 0$ or $P(B_m) = 0$.

To get some insight of overall proofs, we next address how one can obtain the consistency of the empirical distribution with respect to $W_p$.
Suppose for a moment that $P_0$ is supported on $[-1,1]$.
Then, Lemma \ref{lem:D-compact} implies that if $|\bbP_n(F) - P_0(F)|$ is sufficiently small for every $F \in \cP_l$ and $l \leq L$, where $L$ is a large enough constant, then $W_p(\bbP_n, P_0)$ will also be small.
Since there are various tools to bound the deviation $|\bbP_n(F) - P_0(F)|$, \eg\ the inequality by \cite{hoeffding1963probability}, it is not difficult to prove that the empirical distribution converges to $P_0$ in probability with respect to $W_p$, $1 \leq p < \infty$, with the help of Lemma \ref{lem:D-compact}.

In case that $P_0$ has an unbounded support, Lemma \ref{lem:W1-unbounded} can be applied for the Wasserstein consistency of $\bbP_n$.
Indeed, if $|\bbP_n(\pi_m^{-1}(F)) - P_0(\pi_m^{-1}(F))|$ is sufficiently small for every $F \in \cP_l$, $l \leq L$ and $m \leq M$, where $L$ and $M$ are large constants, then $W_p(\bbP_n, P_0)$ will be small.
Note that $L$ and $M$ can be chosen as large but fixed constants, so the consistency of $\bbP_n$ can be similarly proven using a large deviation inequality such as the Hoeffding's inequality.
Here, it plays an important role that $M_p(\bbP_n)$ converges to $M_p(P_0)$ by the law of large numbers, because once the $p$th moment of $\bbP_n$ and $P_0$ is bounded, it is relatively easy to prove the Wasserstein consistency, see the proof of Theorem \ref{thm:consistency} and Lemma \ref{lem:W-set1} for details.

\subsection{Frequently used results from literature}
\label{ssec:lemmas}

The KL condition \eqref{eq:KL-rate} gives a suitable lower bound of the integrated likelihood, that is, the denominator in \eqref{eq:posterior}.
Once this condition holds, the posterior probability of a sequence of subsets $\cF_n$ of $\cF$ can be shown to converge to 1 if the prior probability of $\cF_n^c$ or likelihood is sufficiently small.
The latter can often be expressed through the existence of a certain sequence of uniformly consistent tests.
Lemmas \ref{lem:ghosal2000test} and \ref{lem:ghosal2000prior} are taken from \cite{ghosal2000convergence} with slight modification for the simplicity.
The rate sequence $\epsilon_n$ is assumed that $\epsilon_n \rightarrow 0$ and $n\epsilon_n^2 \rightarrow \infty$.

\medskip
\begin{lemma} \label{lem:ghosal2000test}
Suppose that $\Pi(\cK_n) \geq e^{-n\epsilon_n^2}$ and assume that there exists a sequence of tests $\phi_n$ such that
\bean
	P_0 \phi_n \rightarrow 0
	\quad {\rm and} \quad
	\sup_{P \in \cF_n^c} P(1-\phi_n) \leq e^{-3n\epsilon_n^2}
\eean
for $\cF_n \subset \cF$.
Then, 
  $\Pi\Big( \cF_n^c \mid X_1, \ldots, X_n\Big)
  \rightarrow 0$
in probability.
\end{lemma}

\medskip
\begin{lemma}\label{lem:ghosal2000prior}
Suppose that $\Pi(\cK_n) \geq e^{-n\epsilon_n^2}$ and $\Pi(\cF_n^c) \leq e^{-3n\epsilon_n^2}$ for $\cF_n \subset \cF$.
Then, 
  $\Pi\Big( \cF_n^c \mid X_1, \ldots, X_n\Big)
  \rightarrow 0$
in probability.
\end{lemma}

\medskip
The following lemmas are taken from \cite{fournier2015rate} with slight modification, see also \cite{dereich2013constructive, weed2019sharp}.
Since the statement of Lemma \ref{lem:D-compact} is slightly different from these papers, we provide a detailed proof for the reader's convenience.

\begin{lemma} \label{lem:D-compact}
Assume that two probability measures $P$ and $Q$ are supported on $(-1,1]$.
Then,
\bean
	W_p^p(P,Q) \leq \kappa_p \Bigg(\sum_{l=1}^L 2^{-lp} \sum_{F \in \cP_l} | P(F) - Q(F) | + 2^{-Lp} \Bigg)
\eean
for every $L \geq 1$, where $\kappa_p$ is a constant depending only on $p \geq 1$.
\end{lemma}
\begin{proof}

For a Borel partition $\{A_k: k \geq 1\}$ of a Borel set $A \subset \bbR$ and two finite measures $P$ and $Q$ on $A$ with equal mass, define the finite measure $\overline P$ as
\bean
	\overline P |_{A_k} = \frac{Q(A_k)}{P(A_k)} P |_{A_k}
\eean
if it is well-defined, that is, $P(A_k)=0$ implies $Q(A_k)=0$.
Here, $P |_{A_k}$ and $\overline P |_{A_k}$ denote the restrictions of $P$ and $\overline P$, respectively, onto $A_k$.
We say $\overline P$ is the $\{A_k: k \geq 1\}$-approximation of $P$ to $Q$.
Then, we have the following lemma whose proof is explicitly given in \cite{dereich2013constructive} (pp. 1189--1190).

\begin{lemma} \label{lem:approximation}
Suppose that the $\{A_k: k \geq 1\}$-approximation $\overline P$ of $P$ to $Q$ is well-defined.
Then, there exists a coupling $\xi$ of $P$ and $\overline P$ such that
\bean
	\xi\Big( \big\{ (x,y): x \neq y \big\} \Big) = \frac{1}{2} \sum_{k\geq 1} | P(A_k) - Q(A_k) |.
\eean
\end{lemma}

For $l\geq 0$, let $P_l$ be the $\cP_l$-approximation of $P$ to $Q$.
We only consider the case that $P_l$ is well-defined for all $l \geq 0$.
The other case can be handled with further details, see Proposition 1 of \cite{weed2019sharp}.

Since $P_l(F) = Q(F)$ for $F \in \cP_l$, we have $W_p(P_l, Q) \leq 2^{-(l-1)}$ for every $l \geq 0$.
Furthermore, it is easy to check that, for $F \in \cP_l$, $P_l(F) = P_{l+1}(F)$ and $P_{l+1} |_F$ is the $\{C\in\cP_{l+1}: C\subset F\}$-approximation of $P_l|_F$ to $Q|_F$.
Therefore, by Lemma \ref{lem:approximation}, there exists a coupling $\xi_{l+1}$ of $P_l$ and $P_{l+1}$ such that
\bean
	&& \xi_{l+1} \Big( \{(x,y): x\neq y\} \Big) 
	= \frac{1}{2} \sum_{F \in \cP_l} \sum_{\substack{C: C \subset F \\C\in \cP_{l+1} }} \Big| P_l|_F(C) - Q|_F(C) \Big|
	\\
	&& = \frac{1}{2} \sum_{F \in \cP_l} \sum_{\substack{C: C \subset F \\C\in \cP_{l+1} }} \Big| P_l(C) - Q(C) \Big|
	= \frac{1}{2} \sum_{F \in \cP_l} \sum_{\substack{C: C \subset F \\C\in \cP_{l+1} }} \bigg|  Q(C) - \frac{Q(F)}{P(F)} P(C) \bigg|.
\eean
It follows that there exist random variables $Z_0, Z_1, Z_2, \ldots$ in a same probability space, say $(S, \cS, \mu)$, such that 
\bean
	&& \mu \Big(|Z_{l+1} - Z_l| \leq 2^{-(l-1)} \Big) = 1,
	\\
	&& \mu \Big( Z_{l+1} \neq Z_l \Big ) = \frac{1}{2} \sum_{F \in \cP_l} \sum_{\substack{C: C \subset F \\C\in \cP_{l+1} }} \left| Q(C) - \frac{Q(F)}{P(F)} P(C) \right|
\eean
and $Z_l$ is marginally distributed as $P_l$.
Let $N = \inf \{ l: Z_{l+1} \neq Z_l\}$, where the infimum of the empty set is set to be infinity.
Then, conditional on the event $\{N=l\}$ with $l < L$, where $L$ is a fixed positive integer, we have
\bean
	|Z_0 - Z_L| \leq \sum_{l'=l}^{L-1} |Z_{l'} - Z_{l'+1}| \leq 2^{-(l-2)}
\eean
with probability one.
It follows that
\bean
	\bbE |Z_0 - Z_L|^p 
	&\leq& \sum_{l=0}^{L-1} \bbE \Big[|Z_0 - Z_L|^p ~\Big|~ N=l \Big] \mu\Big( N = l \Big)
	\\
	&\leq& \sum_{l=0}^{L-1} 2^{-(l-2)p} \mu\Big( Z_{l+1} \neq Z_l \Big)
	\\
	&=& \frac{1}{2} \sum_{l=0}^{L-1} 2^{-(l-2)p} \sum_{F \in \cP_l} \sum_{\substack{C: C \subset F \\C\in \cP_{l+1} }} \left| Q(C) - \frac{Q(F)}{P(F)} P(C) \right|.
\eean
Therefore,
\bean
	W_p^p (P, Q) &\leq& 2^{p-1} \Big(W_p^p (P, P_L) + W_p^p(P_L, Q) \Big)
	\\
	&\leq& 2^{p-1} \bigg( \frac{1}{2} \sum_{l=0}^{L-1} 2^{-(l-2)p} \sum_{F \in \cP_l} \sum_{\substack{C: C \subset F \\C\in \cP_{l+1} }} \left| Q(C) - \frac{Q(F)}{P(F)} P(C) \right| + 2^{-p(L-1)} \bigg)	
\eean
Since
\bean
	\left| Q(C) - \frac{Q(F)}{P(F)} P(C) \right|
	&=& \frac{1}{P(F)} \Big| Q(C) P(F) - Q(F) P(C)\Big|
	\\
	&\leq& \frac{1}{P(F)} \Big[ P(F) |Q(C) - P(C)| + P(C) |P(F) - Q(F)| \Big]
	\\
	&=& \frac{P(C)}{P(F)} | P(F) - Q(F)| + |P(C) - Q(C)|,
\eean
we have
\bean
	\frac{W_p^p(P, Q)}{2^{p-1}}
	&\leq& \frac{1}{2} \sum_{l=0}^{L-1} 2^{-(l-2)p} \sum_{F \in \cP_l} \sum_{\substack{C: C \subset F \\C\in \cP_{l+1} }} \left[ \frac{P(C)}{P(F)} | P(F) - Q(F)| + |P(C) - Q(C)| \right] + 2^{-p(L-1)}
	\\
	&=& \frac{1}{2} \sum_{l=0}^{L-1} 2^{-(l-2)p} \left[ \sum_{F\in\cP_l} |P(F) - Q(F)| + \sum_{F\in\cP_{l+1}} |P(F) - Q(F)| \right] + 2^{-p(L-1)}
	\\
	&=& \frac{1}{2} \sum_{l=1}^{L-1} 2^{-(l-2)p} \sum_{F\in\cP_l} |P(F) - Q(F)| 
	+ \frac{1}{2} \sum_{l=1}^L 2^{-(l-3)p} \sum_{F\in\cP_l} |P(F) - Q(F)| + 2^{-p(L-1)}
	\\
	&\leq& \frac{1+2^p}{2} \sum_{l=1}^L 2^{-(l-2)p} \sum_{F \in \cP_l} |P(F) - Q(F)| + 2^{-p(L-1)},
\eean
where the second equality holds because $\sum_{F\in\cP_0} |P(F) - Q(F) | = 0$.
\end{proof}

\medskip
\begin{lemma} \label{lem:W1-unbounded}
For two probability measures $P$ and $Q$ on $\bbR$,
\be\begin{split}\label{eq:W1-unbounded}
	W_p^p(P, Q) \leq \sum_{m \geq 0} 2^{mp} \bigg[ 2^{p-1} |P(B_m) - Q(B_m)| 
	+ \Big( P(B_m) \wedge Q(B_m) \Big) W_p^p( \cR_{B_m} P,  \cR_{B_m} Q) \bigg].
\end{split}\ee
\end{lemma}
\begin{proof}
The proof is explicitly given in \cite{fournier2015rate} (pp. 714--715).
\end{proof}

\subsection{Proof of Theorem \ref{thm:moment}}

Since the KL condition \eqref{eq:KL-support-n} holds, the posterior distribution is consistent with respect to the \Levy-Prokhorov metric $d_P$, see Theorem 6.25 of \cite{ghosal2017fundamentals}.
Therefore, there exists a real sequence $\epsilon_{1n} \downarrow 0$ such that 
\bean
	\Pi\Big(d_P(P, P_0) > \epsilon_{1n} \mid X_1, \ldots, X_n\Big) \longrightarrow 0
	\quad \text{in probability.}
\eean
To see this, let $N_0=1$, and for every $m \geq 1$, choose $N_m > N_{m-1}$ such that
\bean
	\E \left[\Pi\Big(d_P(P, P_0) > \frac{1}{m+1} \mid X_1, \ldots, X_n\Big) \right] \leq \frac{1}{m+1}
	\quad \text{for every $n \geq N_m$}.
\eean
Define $\epsilon_{1n} = (m+1)^{-1}$ if $N_m \leq n < N_{m+1}$.
Then, $\epsilon_{1n} \rightarrow 0$ and for $N_m \leq n < N_{m+1}$, we have
\bean
	\E \left[\Pi\Big(d_P(P, P_0) > \epsilon_{1n} \mid X_1, \ldots, X_n\Big)\right] \leq \frac{1}{m+1} \longrightarrow 0
\eean
as $n \rightarrow \infty$.

Now, suppose that \eqref{eq:moment-consistency} holds.
Then, in a similar way, we can construct a sequence $\epsilon_{2n} \downarrow 0$ such that
\bean
	\Pi\Big( |M_p(P) - M_p(P_0) | > \epsilon_{2n} \mid X_1, \ldots, X_n\Big) \longrightarrow 0
	\quad \text{in probability.}
\eean
Let
\bean
	\cF_n = \Big\{P \in \cF: d_P(P, P_0) \leq \epsilon_{1n}, ~ |M_p(P) - M_p(P_0)| \leq \epsilon_{2n} \Big\}
\eean
and $P_n \in \cF_n$ such that
\bean
	W_p(P_n, P_0) \geq \sup_{P\in \cF_n} W_p(P, P_0) - \epsilon_{1n}.
\eean
Note that $(P_n)$ is a non-random sequence of probability measures such that $d_P(P_n, P_0) \rightarrow 0$ and $M_p(P_n) \rightarrow M_p(P_0)$.
It follows that $W_p(P_n, P_0) \rightarrow 0$. 
Since $\Pi(\cF_n \mid X_1, \ldots, X_n) \rightarrow 1$ in probability, we conclude that \eqref{eq:consistency-W} holds.

Conversely, suppose that \eqref{eq:consistency-W} holds.
Then, similarly as before, we can construct a sequence $\epsilon_{3n} \downarrow 0$ such that
\bean
	\E \left[\Pi\Big(W_p(P, P_0) > \epsilon_{3n} \mid X_1, \ldots, X_n\Big) \right] \longrightarrow 0.
\eean
Let
\bean
	\cF_n' = \Big\{P\in\cF: W_p(P, P_0) \leq \epsilon_{3n} \Big\}.
\eean
and $P_n' \in \cF_n'$ such that
\bean
	|M_p(P_n') - M_p(P_0)| \geq \sup_{P \in \cF_n'} |M_p(P) - M_p(P_0)| - \epsilon_{3n}.
\eean
Again, $(P_n')$ is a non-random sequence with $W_p(P_n', P_0) \rightarrow 0$, so we have $|M_p(P_n') - M_p(P_0)| \rightarrow 0$.
Since $\Pi(\cF_n' \mid X_1, \ldots, X_n) \rightarrow 1$ in probability, we conclude that \eqref{eq:moment-consistency} holds.

\qed

\subsection{Proof of Theorem \ref{thm:consistency}}

We first provide a simple proof relying on a stronger moment condition than the one in the statement of Theorem \ref{thm:consistency}.
For this, we assume that $M_{2p}(P_0) \leq K$ and 
  $$\Pi\Big(P:\, M_{2p}(P)\leq K \mid X_1,\ldots,X_n\Big)
  \to 1\quad\text{in probability.}$$

In view of the characterization of posterior consistency in the Wasserstein distance of Theorem \ref{thm:moment}, we will establish that 
  $
  \Pi \Big(\cC_j  \mid 
  X_1, \ldots, X_n \Big) \rightarrow 0$
in probability for $j=1,2$, where $\cC_1$ and $\cC_2$ are the following two convex sets
\bean
	\cC_1 &=& \left\{P:\, M_p(P)-M_p(P_0) >\epsilon,\quad M_{2p}(P) \leq K\right\}
	\\
	\cC_2 &=& \left\{P:\, M_p(P_0)-M_p(P) >\epsilon,\quad M_{2p}(P)\leq K\right\}
\eean
To this aim, it suffices to show that $\inf_{P\in \cC_1} H(P_0,P)>0$ and $\inf_{P\in \cC_2} H(P_0,P)>0$.
For $P \in \cC_1 \cup \cC_2$, we have
\bean
	|M_p(P) - M_p(P_0)|^2
	&=& \left|\int |x|^p\, \Big(\sqrt{p(x)/p_0(x)}+1\Big)\, \Big(1-\sqrt{p(x)/p_0(x)}\Big)\,p_0(x) dx \right|^2
	\\
	&\leq& H^2(P_0,P)\,\,\int |x|^{2p}\,\Big(1+\sqrt{p(x)/p_0(x)} \Big)^2\,p_0(x) dx
\eean
by the Cauchy--Schwarz inequality.
The integral of the right term is itself upper bounded by
$$K+2\int |x|^{2p}\,\sqrt{p(x)\,p_0(x)} dx+K\leq 4K$$
by virtue of $\sqrt{p\,p_0}\leq \half (p+p_0)$.
Hence, we get $H(P_0,P)\geq \epsilon/(2\sqrt{K})$ for $P \in \cC_1 \cup \cC_2$.
\qed

\medskip
Now, we will prove Theorem \ref{thm:consistency} without the moment condition of order $2p$.

\begin{lemma} \label{lem:W-set1}
For positive constants $\epsilon, \delta$ and $K$ assume that
\be\label{eq:condition-B1}
	P(B_m) + Q(B_m) \leq K 2^{-(p+\delta)m} \quad \text{for $m \geq 0$},
\ee
and
\be \label{eq:condition-F1}
	\Big| P(\pi_m^{-1}(F) \cap B_m) - Q(\pi_m^{-1}(F) \cap B_m) \Big| \leq \epsilon
	\quad \text{for $m \leq M, F \in \cP_l, l \leq L$,}
\ee
where $M$ and $L$ are positive integers.
Then, 
\bean
	W_p^p(P, Q) \leq K' \left[ 2^{-\delta M} + 2^{-Lp} + 2^{Mp} L \epsilon \right],
\eean
where $K'$ is a constant depending only on $\delta, K$ and $p$.
\end{lemma}
\begin{proof}
Since $W_p^p( \cR_{B_m} P,  \cR_{B_m} Q) \leq 2^p$ and \eqref{eq:condition-B1} holds, the summation  in the right hand side of \eqref{eq:W1-unbounded} over $m > M$ is bounded by $c_1 2^{-\delta M}$, where $c_1$ is a constant depending only on $\delta, K$ and $p$.
Therefore, $W_p^p(P, Q)$ is bounded by
\bean	
	\sum_{m=0}^M 2^{mp} \bigg[ 2^{p-1} |P(B_m) - Q(B_m)| +  \Big( P(B_m) \wedge Q(B_m) \Big) W_p^p( \cR_{B_m} P,  \cR_{B_m} Q) \bigg] + c_1 2^{-\delta M}
\eean
by Lemma \ref{lem:W1-unbounded}.
Note that 
\bean
	&& \Big| \cR_{B_m}P(F) - \cR_{B_m}Q(F) \Big| 
	= \left| \frac{P(\pi_m^{-1}(F) \cap B_m)}{P(B_m)} - \frac{Q(\pi_m^{-1}(F) \cap B_m)}{Q(B_m)} \right|
	\\
	&& = \frac{| P(\pi_m^{-1}(F) \cap B_m) Q(B_m) - P(B_m) Q(\pi_m^{-1}(F) \cap B_m) |}{P(B_m) Q(B_m)}
	\\
	&& \leq \frac{1}{P(B_m) Q(B_m)} \bigg[ P(\pi_m^{-1}(F) \cap B_m) \Big|P(B_m) - Q(B_m) \Big| 
	\\
	&& \qquad\qquad\qquad\qquad + P(B_m) \Big|P(\pi_m^{-1}(F) \cap B_m) - Q(\pi_m^{-1}(F) \cap B_m)\Big| \bigg].
\eean
By  Lemma \ref{lem:D-compact} and the last display, we have
\be\begin{split} \label{eq:decompose}
	& \Big( P(B_m) \wedge Q(B_m) \Big) W_p^p( \cR_{B_m} P,  \cR_{B_m} Q)
	\\
	& \leq \kappa_p \Big( P(B_m) \wedge Q(B_m) \Big) \left[ \sum_{l=1}^L 2^{-lp} \sum_{F \in \cP_l} | \cR_{B_m} P(F) - \cR_{B_m} Q(F) | + 2^{-Lp} \right]
	\\
	& \leq \kappa_p \bigg[ \sum_{l=1}^L 2^{-lp} \left\{ \Big| P(B_m) - Q(B_m) \Big| + \sum_{F \in \cP_l} \Big|P(\pi_m^{-1}(F) \cap B_m) - Q(\pi_m^{-1}(F) \cap B_m) \Big| \right\} 
	\\
	& \qquad + 2^{-Lp} \Big( P(B_m) \wedge Q(B_m) \Big) \bigg]
	\\
	& \leq \kappa_p \bigg[ \Big| P(B_m) - Q(B_m) \Big|  + 2^{-Lp} \Big( P(B_m) \wedge Q(B_m) \Big)
	\\
	& \qquad  + \sum_{l=1}^L 2^{-lp} \sum_{F \in \cP_l} \Big|P(\pi_m^{-1}(F) \cap B_m) - Q(\pi_m^{-1}(F) \cap B_m) \Big| \bigg].
\end{split}\ee
It follows  that
\bean
	W_p^p(P, Q) &\leq& c_1 2^{-\delta M} +
	 \sum_{m=0}^M 2^{mp} \bigg\{ (2^{p-1} + \kappa_p) \Big| P(B_m) - Q(B_m) \Big| + \kappa_p 2^{-Lp} \Big( P(B_m) \wedge Q(B_m) \Big) \bigg\}
	\\
	&& + \kappa_p \sum_{m=0}^M 2^{mp} \sum_{l=1}^L 2^{-lp} \sum_{F \in \cP_l}
	\Big|P(\pi_m^{-1}(F) \cap B_m) - Q(\pi_m^{-1}(F) \cap B_m) \Big| 
	\\
	&\leq& c_1 2^{-\delta M} + \epsilon (2^{p-1} + \kappa_p) \sum_{m=0}^M 2^{mp} + K \kappa_p 2^{-Lp} \sum_{m=0}^M 2^{-\delta m} 
	+ \kappa_p \epsilon \sum_{m=0}^M 2^{mp} \sum_{l=1}^L  2^{-(p-1)l}
	\\
	&\leq& c_2 \Big(2^{-\delta M} + 2^{Mp} \epsilon + 2^{-Lp} + 2^{Mp} L \epsilon \Big),
\eean
where the second inequality holds by \eqref{eq:condition-B1}, \eqref{eq:condition-F1} and that the cardinality of $\cP_l$ is $2^l$.
Here, $c_2$ is a constant depending only on $\delta, K$ and $p$.
\end{proof}

By \eqref{eq:moment-prob}, we have that
\bean
	P_0(B_m) \leq 2^{p+\delta} K 2^{-(p+\delta)m}
	\quad \text{for $m \geq 0$}.
\eean
and $\Pi(\cF_0 \mid X_1, \ldots, X_n) \rightarrow 1$ in probability, where
\bean
	\cF_0 = \Big\{P: P(B_m) \leq 2^{p+\delta} K 2^{-(p+\delta)m} ~ \text{for all $m \geq 0$} \Big\}.
\eean
Suppose that a sufficiently small $\epsilon > 0$ is given.
We will prove that for some function $g: (0,\infty) \rightarrow (0, \infty)$, with $g(\epsilon) \downarrow 0$ as $\epsilon \downarrow 0$, 
\be\label{eq:g-consistency}
	\Pi\bigg( W_p^p(P, P_0) \geq g(\epsilon) ~\Big|~ X_1, \ldots, X_n\bigg) \longrightarrow 0
	\quad \text{in probability}.
\ee
Let $L$ and $M$ be the largest integer less than or equal to $\log_2 \epsilon^{-1}$ and $(\log_2 \epsilon^{-1})/(2p)$, respectively.
Then,
\be \label{eq:g-eps}
	2^{-\delta M} + 2^{-Lp} + 2^{Mp} L \epsilon
	\leq 2^\delta \epsilon^{\delta/(2p)} + 2^p \epsilon^p + \sqrt{\epsilon} \log_2 \epsilon^{-1}.
\ee
Let
\bean
	\cF_{m,F,+} &=& \Big\{ P : P(\pi_m^{-1}(F) \cap B_m) - P_0 (\pi_m^{-1}(F) \cap B_m)  < \epsilon \Big\}
	\\
	\cF_{m,F,-} &=& \Big\{ P :  P_0(\pi_m^{-1}(F) \cap B_m) - P (\pi_m^{-1}(F) \cap B_m)  < \epsilon \Big\}.
\eean
Then, by Lemma \ref{lem:W-set1} and \eqref{eq:g-eps}, there exists a constant $c_1$, depending only on $\delta, K$ and $p$, such that
\bean
	P \in \cF_0
	\quad \text{and} \quad
	P \in \bigcap_{m\leq M} \bigcap_{l \leq L} \bigcap_{F \in \cP_l} ( \cF_{m,F,+} \cap\cF_{m,F,-} ) \equiv \cF_\epsilon
\eean
implies that 
\bean
	W_p^p(P, P_0) \leq c_1 \Big(2^\delta \epsilon^{\delta/(2p)} + 2^p \epsilon^p + \sqrt{\epsilon} \log_2 \epsilon^{-1} \Big) \equiv g(\epsilon).
\eean
Certainly, $g(\epsilon) \downarrow 0$ as $\epsilon \downarrow 0$.
Since $\Pi(\cF_0^c \mid X_1, \ldots, X_n) \rightarrow 0$ in probability and the KL condition \eqref{eq:KL-support-n} holds, by Schwartz's theorem (see Theorem 6.25 of \cite{ghosal2017fundamentals} if $\Pi$ depends on $n$), it is sufficient for \eqref{eq:g-consistency} to construct a sequence $\phi_n$ of tests such that
\be \label{eq:test-consistency}
	P_0 \phi_n + 
	\sup_{P \in \cF_\epsilon^c} P(1-\phi_n) \leq e^{-cn} 
\ee
for some constant $c > 0$ and every large enough $n$.

Let
\bean
	\phi_{m,F,+} &=& \left\{\begin{array}{ll}
		1 & ~~\text{if $\bbP_n(\pi_m^{-1}(F) \cap B_m) - P_0(\pi_m^{-1}(F) \cap B_m) > \epsilon / 2$}
		\\
		0 & ~~\text{otherwise}
	\end{array}\right.
	\\
	\phi_{m,F,-} &=& \left\{\begin{array}{ll}
		1 & ~~\text{if $\bbP_n(\pi_m^{-1}(F) \cap B_m) - P_0(\pi_m^{-1}(F) \cap B_m) < - \epsilon / 2 $}
		\\
		0 & ~~\text{otherwise}.
	\end{array}\right.
\eean
Then, by the Hoeffding's inequality, 
\bean
	(P_0 \phi_{m,F,+}) \vee (P_0 \phi_{m,F,-})
	\leq e^{-n\epsilon^2/2}.
\eean
Also, for $P \in \cF^c_{m,F,+}$,
\bean
	P(1 - \phi_{m,F,+}) \leq P\Big( \bbP_n(\pi_m^{-1}(F) \cap B_m) - P(\pi_m^{-1}(F) \cap B_m) \leq -\epsilon/2 \Big) \leq e^{-n\epsilon^2/2}
\eean
by the Hoeffding's inequality.
Similarly, for $P \in \cF^c_{m,F,-}$,
\bean
	P(1 - \phi_{m,F,-}) \leq e^{-n\epsilon^2/2}.
\eean
Therefore, if we define
\bean
	\phi_n = \max_{m \leq M} \max_{l \leq L} \max_{F\in\cP_l} (\phi_{m,F,+} \vee \phi_{m,F,-}),
\eean
then, 
\bean
	P_0 \phi_n \leq \sum_{m \leq M} \sum_{l \leq L} \sum_{F \in \cP_l} P_0 (\phi_{m,F,+} + \phi_{m,F,-})
	\leq 2^{L+1} (L+1) (M+1) e^{-n\epsilon^2/2}.
\eean
Since $L, M$ and $\epsilon$ does not depend on $n$, $\phi_n$ satisfies  \eqref{eq:test-consistency} for some $c > 0$ and large enough $n$, which completes the proof.
\qed

\subsection{Proof of Theorem \ref{thm:moment-rate}}

For a given sequence $\delta_n$, let
\bean
	\cC_{n,1} &=& \{ p: M_p(P) - M_p(P_0) > \delta_n, M_{2p}(P) \leq K \}
	\\
	\cC_{n,2} &=& \{ p: M_p(P_0) - M_p(P) > \delta_n, M_{2p}(P) \leq K \}
\eean
and $\cC_n = \cC_{n,1} \cup \cC_{n,2}$.
Then, it can be shown that
\bean
	\inf_{p \in \cC_n} H^2(p, p_0) \geq \frac{\delta_n^2}{4K}
\eean
as in the first proof of Theorem \ref{thm:consistency}.
For any measurable set $\cC$, let $\Pi_n^\cC$ be the posterior distribution restricted and renormalized onto $\cC$, that is,
\bean
	\Pi_n^\cC(A) = \frac{1}{L_n(\cC)} \int_{A} \prod_{i=1}^n \frac{p}{p_0}(X_i) d\Pi(p)
	\quad \text{for all measurable $A\subset \cC$}
\eean
and let $\bar p_n^\cC(x) = \int_\cC p(x) d\Pi_n^\cC(p)$, where
\bean
	L_n(\cC) = \int_\cC \prod_{i=1}^n \frac{p}{p_0}(X_i) d\Pi(p).
\eean
Since
\bean
	\frac{L_n(\cC)}{L_{n-1}(\cC)} = \frac{\bar p_{n-1}^{\cC}}{p_0}(X_n),
\eean
we have
\bean
	\E\Big[ L_n^{1/2}(\cC) \mid \cG_{n-1} \Big] = L_{n-1}^{1/2}(\cC) \Big( 1 - \half H^2 (p_0, \bar p_{n-1}^{\cC}) \Big),
\eean
where $\cG_{n-1}$ is the $\sigma$-algebra generated by $X_1, \ldots, X_{n-1}$.
Since $\cC_{n,j}$ is convex, we have $H^2 (p_0, \bar p_{n-1}^{\cC_{n,j}}) \geq \delta_n^2/(4K)$ for $j=1, 2$.
Therefore, 
\bean
	\E L_n^{1/2} (\cC_{n,j}) \leq \left(1 - \frac{\delta_n^2}{8K} \right)^n
	\leq e^{-c_1 n \delta_n^2}
\eean
for all large enough $n$, where $c_1>0$ is a constant depending only on $K$.
It follows that $L_n(\cC_{n,j})$ is upper bounded by $e^{-c_2 n\delta_n^2}$ with probability tending to 1 for some constant $c_2$.
Thus, if we take $\delta_n = K' \epsilon_n$ for large enough $K'$, the proof is complete.

\subsection{Proof of Theorem \ref{thm:Wp-rate}}

\begin{lemma} \label{lem:W-set}
For positive constants $\alpha, \delta, \epsilon$ and $K$, suppose that
\be\label{eq:condition-B}
	P(B_m) + Q(B_m) \leq K 2^{-(2p+\delta)m} \quad \text{for $m \geq 0$},
\ee
and
\be \label{eq:condition-F}
	\Big| P(\pi_m^{-1}(F) \cap B_m) - Q(\pi_m^{-1}(F) \cap B_m) \Big| \leq \frac{K (2+\alpha)^{-mp}  \epsilon}{(l+1)^2}
	\quad \text{for $m \leq M $, $F \in \cP_l$, $l \leq L$.}
\ee
 Then,
\bean
	W_p^p(P, Q) \leq K' \left[2^{-Lp} + \epsilon + 2^{-(p+\delta)M} \right],
\eean
where $K'$ is a constant depending only on $\alpha, K$ and $p$.
If $p > 1$, condition \eqref{eq:condition-F} can be replaced by a slightly weaker condition that
\be\label{eq:condition-F2}
	\Big| P(\pi_m^{-1}(F) \cap B_m) - Q(\pi_m^{-1}(F) \cap B_m) \Big| \leq K (2+\alpha)^{-mp}  \epsilon
	\quad \text{for $m \leq M $, $F \in \cP_l$, $l \leq L$.}
\ee
\end{lemma}
\begin{proof}
By \eqref{eq:condition-B} and that $W_p^p( \cR_{B_m} P,  \cR_{B_m} Q) \leq 2^p$, the summation in the right hand side of \eqref{eq:W1-unbounded} over $m > M$ is bounded by a constant multiple of $2^{-(p+\delta)M}$.
Since $B_0 \in \cP_0$,
\bean
	\sum_{m=0}^M 2^{mp} \big| P(B_m) - Q(B_m)\big| 
	&=& \sum_{m=0}^M 2^{mp} \big| P(\pi_m^{-1}(B_0) \cap B_m) - Q(P(\pi_m^{-1}(B_0) \cap B_m) \big| 
	\\
	&\leq& K\epsilon \sum_{m=0}^M \Big( \frac{2}{2+\alpha} \Big)^{mp}
	= K\epsilon \frac{1-(1+\alpha/2)^{-(M+1)p}}{1-(1+\alpha/2)^{-p}},
\eean
where the inequality holds by \eqref{eq:condition-F} with $l=0$.
Therefore, 
\bean
	W_p^p(P, Q) 
	&\leq& \sum_{m=0}^M 2^{mp} \Big( P(B_m) \wedge Q(B_m) \Big) W_p^p( \cR_{B_m} P,  \cR_{B_m} Q) + K' \Big(\epsilon + 2^{-(p+\delta)M} \Big)
\eean
by Lemma \ref{lem:W1-unbounded}, where $K'$ is a constant depending only on $\alpha, K$ and $p$.
By \eqref{eq:decompose}, the summation in the last display is bounded by
\be \begin{split} \label{eq:tech-l}
	\kappa_p \sum_{m=0}^M 2^{mp}
	& \bigg[ \Big| P(B_m) - Q(B_m) \Big| 
	 + 2^{-Lp} \Big( P(B_m) \wedge Q(B_m) \Big)
	 \\	
	& + \sum_{l=1}^L 2^{-lp} \sum_{F \in \cP_l} \Big|P(\pi_m^{-1}(F) \cap B_m) - Q(\pi_m^{-1}(F) \cap B_m) \Big| \bigg]
\end{split} \ee
Since the cardinality of $\cP_l$ is $2^l$ and $\sum_{l=1}^\infty (l+1)^{-2} < \infty$, the first assertion follows from \eqref{eq:condition-B} and \eqref{eq:condition-F}.

If $p > 1$ and \eqref{eq:condition-F} is replaced by \eqref{eq:condition-F2}, we have
\bean
	&& \sum_{l=1}^L 2^{-lp} \sum_{F \in \cP_l} \Big|P(\pi_m^{-1}(F) \cap B_m) - Q(\pi_m^{-1}(F) \cap B_m) \Big|
	\\
	&& \leq \sum_{l=1}^L K 2^{-l(p-1)} (2+\alpha)^{-mp} \epsilon
	\leq \frac{K}{2^{p-1} -1} (2+\alpha)^{-mp} \epsilon.
\eean
Therefore, we have the same conclusion with a different constant $K'$.
\end{proof}

\begin{lemma} \label{lem:binomial-concentration}
If $X_1, \ldots, X_n \iidsim P$, then
\bean
	P(\bbP_n(B) \leq P(B) - \epsilon) &\leq& \exp\left( -\frac{n\epsilon^2}{2P(B)} \right)
	\\
	P(\bbP_n(B) \geq P(B) + \epsilon) &\leq& \exp \left( -\frac{n\epsilon^2}{2\{P(B) + \epsilon/3\}} \right)
\eean
for every $n \geq 1$ and $\epsilon \geq 0$.
\end{lemma}
\begin{proof}
See Theorem 1 of \cite{janson2016large}.
\end{proof}

\medskip
Before proving Theorem \ref{thm:Wp-rate}, we state and prove a similar one.
Theorem \ref{thm:Wp-rate-general} below is devised for eliminating the logarithmic term $\log \epsilon_n^{-1}$ in Theorem \ref{thm:Wp-rate}.
Proofs of Theorem \ref{thm:Wp-rate} and \ref{thm:Wp-rate-general} are quite similar, so we do not provide all details to avoid the repetition.
We provide a detailed proof only for Theorem \ref{thm:Wp-rate-general} because this contains the most technical part caused by the factors $(l+1)^{-2}$ and $(l+1)^{-4}$.
These factors appear for handling the last term in \eqref{eq:tech-l}.
If $p > 1$, we need not consider these factors by the second assertion of Lemma \ref{lem:W-set}.
For $p=1$, we can avoid the technical factors $(l+1)^{-2}$ and $(l+1)^{-4}$, with an additional logarithmic factor in the rate.
If we want to eliminate the term $\log \epsilon_n^{-1}$, the statement would become more complicated as Theorem \ref{thm:Wp-rate-general}.
For conciseness we decided to include Theorem \ref{thm:Wp-rate} in the main texts rather than Theorem \ref{thm:Wp-rate-general}.

\begin{theorem} \label{thm:Wp-rate-general}
Assume that the prior $\Pi$ satisfies the KL condition \eqref{eq:KL-rate} for a sequence $\epsilon_n$ with $\epsilon_n \downarrow 0$ and $\epsilon_n \geq \sqrt{(\log n)/n}$.
Furthermore, assume that there exist positive constants $K$ and $\delta$ such that
\bean
	P_0(\pi_m^{-1}(F) \cap B_m) \leq \frac{K}{(l+1)^4} 2^{-(2p+\delta)m}
	\quad \text{for $m \geq 0, F \in \cP_l, l \geq 0$}
\eean
and $\Pi(\cF_n^c \mid X_1, \ldots, X_n) \rightarrow 0$ in probability, where
\bean
	\cF_n = \bigcap_{m \geq 0} \bigcap_{l \leq L} \bigcap_{F \in \cP_l} 
	\bigg\{ P: P(\pi_m^{-1}(F) \cap B_m) \leq \frac{K}{(l+1)^4} 2^{-(2p+\delta)m} \bigg\}
\eean
and $L$ is the largest integer less than or equal to $(\log_2 \epsilon_n^{-1})/p$.
Then, for some constant $K' > 0$, 
\be \label{eq:W-rate}
	\Pi\bigg( W_p^p(P, P_0) \geq K' \epsilon_n ~\Big|~ X_1, \ldots, X_n\bigg) \longrightarrow 0
	\quad \text{in probability}.
\ee
\end{theorem}
\begin{proof}
Let $M$ be the largest integer less than or equal to $(p+\delta)^{-1} \log_2 \epsilon_n^{-1}$.
Let $\alpha >0$ be a sufficiently small constant such that $(1+\alpha/2)^{2p} < 2^\delta$.
For $m \leq M$ and $F \in \cP_l$ with $l \leq L$, let
\bean
	\cF_{m,F,+} &=& \bigg\{P \in \cF_n: P(\pi_m^{-1}(F) \cap B_m) - P_0 (\pi_m^{-1}(F) \cap B_m) \leq \frac{2K_1}{(l+1)^2} (2+\alpha)^{-mp}\epsilon_n \bigg\}
	\\
	\cF_{m,F,-} &=& \bigg\{P \in \cF_n: P_0(\pi_m^{-1}(F) \cap B_m) - P (\pi_m^{-1}(F) \cap B_m) \leq \frac{2K_1}{(l+1)^2} (2+\alpha)^{-mp}\epsilon_n \bigg\},
\eean
where $K_1>0$ is a large constant described below.
Then, by Lemma \ref{lem:W-set},
\bean
	P \in \bigcap_{m \leq M} \bigcap_{l \leq L} \bigcap_{F \in \cP_l} (\cF_{m,F,+} \cap \cF_{m,F,-}) \equiv \cF_n'
\eean
implies that $W_p^p(P, P_0) \leq K_2 \epsilon_n$ for some constant $K_2$.
Since $\Pi(\cK_n) \geq e^{-n\epsilon_n^2}$, by Lemma \ref{lem:ghosal2000test}, it is sufficient for \eqref{eq:W-rate} to construct a sequence $\phi_n$ of tests such that
\be \label{eq:test-rate}
	P_0 \phi_n \rightarrow 0
	\quad {\rm and} \quad
	\sup_{P \in (\cF'_n)^c} P(1-\phi_n) \leq e^{-3n\epsilon_n^2}
\ee
for every large enough $n$.

For $m \leq M$ and $F\in\cP_l$ with $l \leq L$, let
\bean
	\phi_{m,F,+} &=& \left\{\begin{array}{ll}
		1 & ~~\text{if $\bbP_n(\pi_m^{-1}(F) \cap B_m) - P_0(\pi_m^{-1}(F) \cap B_m) > \frac{K_1}{(l+1)^2} (2+\alpha)^{-mp} \epsilon_n $}
		\\
		0 & ~~\text{otherwise}
	\end{array}\right.
	\\
	\phi_{m,F,-} &=& \left\{\begin{array}{ll}
		1 & ~~\text{if $\bbP_n(\pi_m^{-1}(F) \cap B_m) - P_0(\pi_m^{-1}(F) \cap B_m) < - \frac{K_1}{(l+1)^2} (2+\alpha)^{-mp} \epsilon_n $}
		\\
		0 & ~~\text{otherwise}.
	\end{array}\right.
\eean
Then, by Lemma \ref{lem:binomial-concentration},
\bean
	P_0 \phi_{m,F,+} \leq \exp\left[ \frac{- K_1^2 (l+1)^{-4} (2+\alpha)^{-2mp} n \epsilon_n^2}{2\{P_0(\pi_m^{-1}(F) \cap B_m) + K_1 (l+1)^{-2} (2+\alpha)^{-mp} \epsilon_n /3 \}}\right].
\eean
If $P_0(\pi_m^{-1}(F) \cap B_m)> K_1 (l+1)^{-2} (2+\alpha)^{-mp} \epsilon_n /3$, then
\bean
	&& P_0 \phi_{m,F,+} 
	\leq \exp\left[ -\frac{K_1^2 (l+1)^{-4} (2+\alpha)^{-2mp} n \epsilon_n^2}{4P_0(\pi_m^{-1}(F) \cap B_m)} \right] 
	\\
	&& \leq \exp\left[ -\frac{K_1^2 (2+\alpha)^{-2mp} n \epsilon_n^2}{4K 2^{-(2p+\delta)m} } \right] 
	= \exp\left[ -\frac{K_1^2 \beta^m n\epsilon_n^2}{4K} \right],
\eean
where $\beta = 2^{\delta} (1+\alpha/2)^{-2p}  > 1$.
Otherwise,
\bean
	&& P_0 \phi_{m,F,+} 
	\leq \exp\left[ -\frac{3}{4} K_1 (l+1)^{-2} (2+\alpha)^{-mp} n \epsilon_n \right] 
	\\
	&& \leq \exp\left[ -\frac{3}{4} K_1 n \Big(\frac{1}{p}\log_2 \epsilon_n^{-1} + 1\Big)^{-2} \epsilon_n^{1 + \frac{p \log_2(2+\alpha)}{p+\delta}} \right]
	\leq e^{ -K_1 n \epsilon_n^2 }
\eean
for large enough $n$, where the second inequality holds because $m \leq M \leq (p+\delta)^{-1} \log_2 \epsilon_n^{-1}$ and $l \leq L \leq (\log_2 \epsilon_n^{-1})/p$, and the third inequality holds because 
\bean
	\frac{p \log_2(2+\alpha)}{p+\delta} = \frac{\log_2(2+\alpha)^p}{\log_2 2^{p+\delta}}
	\quad \text{and} \quad \frac{(2+\alpha)^p}{ 2^{p+\delta}} = \frac{(1+\alpha/2)^p}{2^\delta} < 1.
\eean
Since $\epsilon_n \geq \sqrt{(\log n)/n}$, we have
\bean
	\sum_{m \leq M} \sum_{l \leq L} \sum_{F \in \cP_l} P_0 \phi_{m,F,+} 
	\leq ML2^L \max_{m \leq M} \max_{l \leq L} \max_{F \in \cP_l} P_0 \phi_{m,F,+} 
	\\
	\leq \frac{1}{p(p+\delta)} \Big(\log_2 \epsilon_n^{-1}\Big)^2 \bigg( \frac{1}{\epsilon_n} \bigg)^{1/p} \max_{m \leq M} \max_{l \leq L} \max_{F \in \cP_l} P_0 \phi_{m,F,+} \longrightarrow 0
\eean
as $n \rightarrow \infty$ provided that $K_1$ is large enough.
Also, if $K_1$ is sufficiently large, for $P \in \cF_{m,F,+}^c$ with $F \in \cP_l$,
\bean
	P(1 - \phi_{m,F,+}) 
	= P \bigg(\bbP_n(\pi_m^{-1}(F) \cap B_m) - P(\pi_m^{-1}(F) \cap B_m) \leq - \frac{K_1}{(l+1)^2} (2+\alpha)^{-mp} \epsilon_n \bigg)
	\\
	\leq \exp\left[ -\frac{K_1^2 (l+1)^{-4} (2+\alpha)^{-2mp} n \epsilon_n^2}{2P(\pi_m^{-1}(F) \cap B_m)} \right] 
	\leq \exp\left[ -\frac{K_1^2 \beta^m}{2K} n\epsilon_n^2  \right]
	\leq e^{-3n\epsilon_n^2}
\eean
for large enough $n$, where the first inequality holds by Lemma \ref{lem:binomial-concentration}.
A similar inequalities for $P_0 \phi_{m,F,-}$ and $P(1-\phi_{m,F,-})$ can also be obtained.
Therefore, if we define
\bean
	\phi_n = \max_{m \leq M} \max_{l \leq L} \max_{F\in\cP_l} (\phi_{m,F,+} \vee \phi_{m,F,-})
\eean
and $K_1$ is sufficiently large, then $\phi_n$ satisfies \eqref{eq:test-rate} for all large enough $n$.
This completes the proof.
\end{proof}

\medskip
\noindent {\it Proof of Theorem \ref{thm:Wp-rate} for $p > 1$.}
We first claim that if
\bean
	P_0(B_m) \leq K 2^{-(2p+\delta)m}
	\quad \text{for $m \geq 0$}
\eean
and $\Pi(\cF_0^c \mid X_1, \ldots, X_n) \rightarrow 0$ in probability, where
\bean
	\cF_0 = \bigcap_{m \geq 0}
	\Big\{ P: P(B_m) \leq K 2^{-(2p+\delta)m} \Big\},
\eean
then \eqref{eq:W-rate} holds for some constant $K'$.
The proof of this claim is the same to that of Theorem \ref{thm:Wp-rate-general} if we replace $\cF_n$ by $\cF_0$ and eliminate the factors $(l+1)^{-2}$ and $(l+1)^{-4}$ in all equations, which is possible due to the second assertion of Lemma \ref{lem:W-set}.

Once we adjust the constant $K$, two conditions of the claim is satisfied by \eqref{eq:moment-prob}.
Hence the proof is complete.
\qed

\medskip
\noindent {\it Proof of Theorem \ref{thm:Wp-rate} for $p=1$.}
If \eqref{eq:condition-B} and \eqref{eq:condition-F2} hold with $p=1$, then it holds that
\be \label{eq:W1-bound}
	W_1(P, Q) \leq K' \left[2^{-L} + L\epsilon + 2^{-(1+\delta)M} \right].
\ee
This can be proved as in Lemma \ref{lem:W-set}.
The only difference is that the last term in \eqref{eq:tech-l} is bounded as
\bean
	\kappa_p \sum_{m=0}^M 2^{m} \sum_{l=1}^L 2^{-l} \sum_{F \in \cP_l} \Big| P(\pi_m^{-1}(F) \cap B_m) - Q(\pi_m^{-1}(F) \cap B_m) \Big|
	\\
	\leq K \kappa_p \epsilon \sum_{m=0}^M \sum_{l=1}^L \Big( \frac{2}{2+\alpha} \Big)^{m} \leq K' L \epsilon,
\eean
where $K'$ is a constant depending only on $\alpha, K$ and $p$.

As in the proof of Theorem \ref{thm:Wp-rate}, we next claim that if
\bean
	P_0(B_m) \leq K 2^{-(2+\delta)m}
	\quad \text{for $m \geq 0$}
\eean
and $\Pi(\cF_0^c \mid X_1, \ldots, X_n) \rightarrow 0$ in probability, where
\bean
	\cF_0 = \bigcap_{m \geq 0}
	\Big\{ P: P(B_m) \leq K 2^{-(2+\delta)m} \Big\},
\eean
then 
\bean
	\Pi\bigg( W_p^p(P, P_0) \geq K' \epsilon_n \log \epsilon_n^{-1} ~\Big|~ X_1, \ldots, X_n\bigg) \longrightarrow 0
	\quad \text{in probability}
\eean
for some constant $K'$.
To prove this, define $L, M$ and $\alpha$ as in Theorem \ref{thm:Wp-rate-general} with $p=1$.
Also, for $m \leq M$ and $F \in \cP_l$ with $l \leq L$, let
\bean
	\cF_{m,F,+} &=& \bigg\{P \in \cF_0: P(\pi_m^{-1}(F) \cap B_m) - P_0 (\pi_m^{-1}(F) \cap B_m) \leq 2K_1 (2+\alpha)^{-m}\epsilon_n \bigg\}
	\\
	\cF_{m,F,-} &=& \bigg\{P \in \cF_0: P_0(\pi_m^{-1}(F) \cap B_m) - P (\pi_m^{-1}(F) \cap B_m) \leq 2K_1 (2+\alpha)^{-m}\epsilon_n \bigg\},
\eean
where $K_1>0$ is a large constant as in the proof of Theorem \ref{thm:Wp-rate-general}.
Then, by \eqref{eq:W1-bound},
\bean
	P \in \bigcap_{m \leq M} \bigcap_{l \leq L} \bigcap_{F \in \cP_l} (\cF_{m,F,+} \cap \cF_{m,F,-}) \equiv \cF_n'
\eean
implies that $W_p^p(P, P_0) \leq K_2 \epsilon_n \log \epsilon_n^{-1}$ for some constant $K_2$.
Once we change the definition of $\phi_n$ as
\bean
	\phi_n = \max_{m \leq M} \max_{l \leq L} \max_{F\in\cP_l} (\phi_{m,F,+} \vee \phi_{m,F,-}),
\eean
where
\bean
	\phi_{m,F,+} &=& \left\{\begin{array}{ll}
		1 & ~~\text{if $\bbP_n(\pi_m^{-1}(F) \cap B_m) - P_0(\pi_m^{-1}(F) \cap B_m) > K_1 (2+\alpha)^{-m} \epsilon_n $}
		\\
		0 & ~~\text{otherwise}
	\end{array}\right.
	\\
	\phi_{m,F,-} &=& \left\{\begin{array}{ll}
		1 & ~~\text{if $\bbP_n(\pi_m^{-1}(F) \cap B_m) - P_0(\pi_m^{-1}(F) \cap B_m) < - K_1 (2+\alpha)^{-m} \epsilon_n $}
		\\
		0 & ~~\text{otherwise},
	\end{array}\right.
\eean
the remaining proof of the claim is the same to that of Theorem \ref{thm:Wp-rate-general}.

Once we adjust the constant $K$, two conditions of the claim is satisfied by \eqref{eq:moment-prob}.
Hence the proof is complete.
\qed

\subsection{Proof of Theorem \ref{thm:W_infty}}

Let $\cF_\epsilon = \{P \in \cF_0: W_\infty(P, P_0) \leq \epsilon\}$.
We will show that for every small enough $\epsilon \geq K_1 \sqrt{(\log n)/n}$ and $n \geq n_0$, there exists a test $\phi$ such that
\be\label{eq:W_inf-test}
	P_0 \phi \leq e^{-K_2 n\epsilon^2}
	\quad \textrm{and} \quad
	\sup_{P \in \cF_{2\epsilon}^c} P(1-\phi) \leq e^{-K_2n\epsilon^2},
\ee
where $K_1, K_2$ and $n_0$ are constants depending only on $c_0$.
Since $\Pi(\cK_n) \geq e^{-n \epsilon_n^2}$, \eqref{eq:W_inf-test} and Lemma \ref{lem:ghosal2000test} guarantees \eqref{eq:W_inf-rate} for large enough constant $K>0$.

Let $\epsilon > 0$ be given.
Let $N$ be the smallest integer greater than or equal to $\epsilon^{-1}$.
Let $I_j = [(j-1)\epsilon, j\epsilon)$ for $j=1, \ldots, N-1$ and $I_N = [(N-1)\epsilon, 1]$.
Let $I_{jk} = \cup_{l=j}^{j+k-1} I_l$ for $j=1, \ldots, N$ and $k = 1, \ldots, N-j+1$.
Let $\scrI$ and $\scrB$ be the collections of every interval $I_{jk}$ and every finite union of $I_{jk}$, respectively.
Note that the cardinalities of $\scrI$ and $\scrB$ are $N(N+1)/2$ and $2^N-1$, respectively.

We first claim that for $P \in \cF_0$,
\be \label{eq:I_jk-dev}
	P(I_{jk}) - P_0(I_{jk}) \leq \frac{c_0 \epsilon}{2} \;\; \textrm{for every $j$ and $k$ implies that $W_\infty(P_0, P) \leq 2\epsilon$.}
\ee
If $B$ is either $[0,1]$ or $[0, (N-1)\epsilon)$, it is obvious that $P(B) \leq P_0(B^{\epsilon})$.
Also, for $B = I_{jk}$ for some $(j,k)$, with $B \neq [0,1]$ and $B \neq [0, (N-1)\epsilon)$, 
\bean
	P(B) \leq P_0(B) + \frac{c_0 \epsilon}{2} < P_0(B^{\epsilon}).
\eean
Thus, $P(B) \leq P_0(B^\epsilon)$ for every $B \in \scrI$.
For $B\in\scrB-\scrI$, we have $B = \cup_{l=1}^L B_l$ for some $L \geq 2$, where $B_l \in \scrI$ and $B_l$'s are $\epsilon$-separated.
Thus,
\bean
	P_0(B^{\epsilon}) \geq \sum_{l=1}^L P_0(B_l) + (L-1) c_0\epsilon.
\eean
It follows by \eqref{eq:I_jk-dev} that
\bean
	P(B) = \sum_{l=1}^L P(B_l) \leq \sum_{l=1}^L P_0(B_l) + \frac{Lc_0\epsilon}{2} \leq \sum_{l=1}^L P_0(B_l) + (L-1) c_0\epsilon \leq P_0(B^{\epsilon}).
\eean
Thus, we have $P(B) \leq P_0(B^{\epsilon})$ for every $B \in \scrB$.
Next, for any Borel subset $A$ of $[0,1]$, let $J = \{j: A\cap I_j \neq \emptyset\}$ and $C = \cup_{j \in J} I_j$.
Then, we have $C \in \scrB$ and $A \subset C \subset A^{\epsilon}$.
Therefore,
\bean
	P(A) \leq P(C) \leq P_0(C^{\epsilon})  \leq P_0(A^{2\epsilon}).
\eean
This proves \eqref{eq:I_jk-dev}.

By \eqref{eq:I_jk-dev}, 
Then, $\cF_{2\epsilon} \supset \cap_{j,k} \cF_{jk}$, where
\bean
	\cF_{jk} = \{P: P(I_{jk}) - P_0(I_{jk}) \leq c_0\epsilon/2\}.
\eean
Define test functions $\phi_{jk}$ as $\phi_{jk}=1$ if
\bean
	\bbP_n(I_{jk}) > P_0(I_{jk}) + \frac{c_0 \epsilon}{4}
\eean
and $\phi_{jk}=0$ otherwise.
Then, for every $P \in \cF_{jk}^c$, we have
\bean
	&& P(1-\phi_{jk}) = P\left\{ \bbP_n(I_{jk}) \leq P_0(I_{jk}) + \frac{c_0 \epsilon}{4} \right\}
	\\
	&& = P\left\{ \bbP_n(I_{jk}) \leq P(I_{jk}) + P_0(I_{jk}) - P(I_{jk}) + \frac{c_0 \epsilon}{4} \right\}
	\\
	&& \leq P\left\{ \bbP_n(I_{jk}) < P(I_{jk}) - \frac{c_0 \epsilon}{4} \right\}
	\leq \exp\left[ - \frac{nc_0^2\epsilon^2}{8}\right],
\eean
where the last inequality holds by the Hoeffding's inequality.
Let $\phi = \max_{j,k} \phi_{jk}$.
Applying the Hoeffding's inequality again, we have
\be \label{eq:test}
	P_0 \phi \leq \sum_{j,k} P_0 \phi_{jk} \leq \frac{N(N+1)}{2} \exp\left[ - \frac{nc_0^2\epsilon^2}{8}\right] \leq \exp\left[ 2\log(\epsilon^{-1} + 1) - \frac{nc_0^2\epsilon^2}{8}\right].
\ee
Therefore, we can choose constants $K_1, K_2 > 0$ and $n_0$ such that if $\epsilon \geq K_1  \sqrt{(\log n)/n}$, then the right hand side of \eqref{eq:test} is bounded by $\exp(-K_2 n \epsilon^2)$ for every $n \geq n_0$.
This completes the proof of \eqref{eq:W_inf-test}.
\qed

\subsection{Proof of Theorem \ref{thm:dpm-specific}}

\begin{lemma} \label{lem:gamma-bound}
There exist universal constants $c_1, c_2 > 0$ such that
\bean
	\frac{c_1}{\epsilon} \leq \Gamma(\epsilon) \leq \frac{c_2}{\epsilon}
\eean
for every $\epsilon \in (0,1]$.
\end{lemma}
\begin{proof}
Note that
\bean
	\Gamma(\epsilon) = \int_0^\infty x^{\epsilon-1} e^{-x} dx 
	= \int_0^1 x^{\epsilon-1} e^{-x} dx + \int_1^\infty x^{\epsilon-1} e^{-x} dx.
\eean
It is easy to show that there exist constants $a_1$ and $a_2$ such that
\bean
	a_1 \leq \int_1^\infty x^{\epsilon-1} e^{-x} dx \leq a_2
\eean
for every $\epsilon \in (0,1]$.
The assertion follows because $e^{-1} \leq e^{-x} \leq 1$ for every $x \in (0,1]$ and $\int_0^1 x^{\epsilon-1} dx = \epsilon^{-1}$.
\end{proof}

\begin{lemma} \label{lem:beta}
Suppose that $X \sim {\rm Beta}(\alpha\epsilon, \alpha(1-\epsilon))$, $\alpha\epsilon \leq 1$ and $\alpha(1-\epsilon) \geq 1$.
Then,
\bean
	P(X > t) \leq C_\alpha(1 - t^{\alpha\epsilon}),
\eean
where $C_\alpha$ is a constant depending only on $\alpha$.
\end{lemma}
\begin{proof}
Let $p$ be the pdf of $X$, that is,
\bean
	p(x) = \frac{\Gamma(\alpha)}{\Gamma(\alpha\epsilon) \Gamma(\alpha(1-\epsilon))}
	x^{\alpha\epsilon-1} (1-x)^{\alpha(1-\epsilon)-1}.
\eean
By Lemma \ref{lem:gamma-bound},
\bean
	P(X > t) = \int_t^1 p(x) dx \leq c_\alpha \epsilon \int_t^1 x^{\alpha\epsilon-1} dx
	= \frac{c_\alpha}{\alpha} (1 - t^{\alpha\epsilon}),
\eean
where $c_\alpha$ is a constant depending only on $\alpha$.
\end{proof}

\begin{lemma} \label{lem:set-rate}
Suppose that $\Pi(\cK_n) \geq e^{-n\epsilon_n^2}$ and $\epsilon_n \geq \sqrt{(\log n)/n}$. Then, there exists a universal constant $K>0$ such that
\bean
	\Pi\Big( \big| P(B_m) - P_0(B_m) \big| \leq K \epsilon_n ~ \forall m \leq  C \log \epsilon_n^{-1}  \mid X_1, \ldots, X_n \Big) \rightarrow 1 \quad\text{in probability}
\eean
for every $C > 0$.
\end{lemma}
\begin{proof}
Let $C >0$ be given.
For eacn $n$ and $m \leq C\log \epsilon_n^{-1}$, let
\bean
	\psi_m = I\Big( \big| \bbP_n(B_m) - P_0(B_m) \big| > K \epsilon_n/2 \Big),
\eean
where $K$ is a universal constant described below.
Using the Hoeffding's inequality, it is not difficult to prove that
\bean
	P_0 \psi_m \lesssim e^{-K^2n\epsilon_n^2/2}
	\quad {\rm and} \quad
	\sup_{P \in \cG_m^c} P(1-\psi_n) \lesssim e^{-K^2 n\epsilon_n^2/2},
\eean
where $\cG_m = \{P: |P(B_m) - P_0(B_m)| \leq K\epsilon_n \}$.
Let
\bean
	\phi_n = \max_{m \leq C\log \epsilon_n^{-1}} \psi_m.
\eean
Then, we have
\bean
	P_0 \phi_n \lesssim C \log \epsilon_n^{-1} e^{-K^2n\epsilon_n^2/2} \rightarrow 0
	\quad {\rm and} \quad
	\sup_{P \in \cF_n^c} P(1-\phi_n) \lesssim e^{-K^2 n\epsilon_n^2/2},
\eean
where 
\bean
	\cF_n = \bigcap_{m \leq C \log \epsilon_n^{-1}} \cG_m.
\eean
Thus, the proof is complete by Lemma \ref{lem:ghosal2000test} provided that $K^2/2 \geq 3$.
\end{proof}

\begin{lemma} \label{lem:dpm}
Let $\epsilon_n$ be a sequence such that $\epsilon_n \rightarrow 0$ and $\epsilon_n \geq \sqrt{\log n / n}$.
Let $H$ be the normal distribution with mean $\mu_H$ and variance $\sigma^2_H$.
For a Dirichlet process mixture prior \eqref{eq:dpm} with $\alpha > 1$ and $\sigma=\sigma_n \rightarrow 0$, suppose that $\Pi(\cK_n) \geq e^{-n\epsilon_n^2}$.
Also, for some $p \in [1,\infty)$, assume that $P_0(B_m) \leq K 2^{-pm}$ for every $m \geq 0$, and that $\epsilon_n \leq A n^{-p/(2+2p)}$ for every $n$, where $K$ and $A$ are constants.
Then,
\bean
	\Pi \Big( P(B_m) \leq K' 2^{-pm} ~ \forall m \geq 0 \mid X_1, \ldots, X_n\Big) \rightarrow 1 \quad \text{in probability,}
\eean
where $K'$ is a large enough constant.
\end{lemma}
\begin{proof}
Let $\widetilde \epsilon_n = L \epsilon_n$, and define $\widetilde \cK_n$ as $\cK_n$ after replacing $\epsilon_n$ by $\widetilde\epsilon_n$, where $L$ is a large constant described below.
Then, $\Pi(\widetilde \cK_n) \geq \Pi(\cK_n) \geq e^{-n\epsilon_n^2}$.
Note that
\bean
	G(B) \sim {\rm Beta} \Big(\alpha H(B), \alpha \big(1-H(B)\big) \Big)
\eean
for any Borel set $B$.
Also, $\alpha H(\overline B_m) \leq 1$ and $\alpha (1 - H(\overline B_m)) \geq 1$ for every large enough $m$, where
\bean
	\overline B_m = (-\infty, -2^{m-1}] \cup (2^{m-1}, \infty).
\eean
Thus, by Lemma \ref{lem:beta}, for any $K' \geq 2^p$,
\be\begin{split}\label{eq:G_B_prob}
	& \Pi \big( G(B_m) > K' 2^{-pm} \big) 
	\leq \Pi \big( G(\overline B_m) > K' 2^{-pm} \big) 
	\\
	& \leq \Pi \big( G(\overline B_m) > 2^{-pm} \big) 
	\leq C_\alpha \big(1 - 2^{-pm\alpha H(\overline B_m)} \big)
\end{split}\ee
for every large enough $m$, where $C_\alpha$ is a constant depending only on $\alpha$.
Note that $1-\Phi_\sigma(x) \leq e^{-x^2/(2\sigma^2)}/2$, where $\Phi_\sigma$ is the cdf of $\phi_\sigma$, so we have
\bean
	H(\overline B_m) \leq \frac{1}{2} \left[ 
	\exp\left\{ -\frac{1}{2}\left(\frac{2^{m-1}-\mu_H}{\sigma_H}\right)^2 \right\} 
	+ \exp\left\{ -\frac{1}{2}\left(\frac{2^{m-1}+\mu_H}{\sigma_H}\right)^2 \right\} 
	\right].
\eean
Since $1 - e^{-x} \leq x$, the right hand side of \eqref{eq:G_B_prob} is bounded by
\bean
	&& \frac{\log 2}{2}  p \alpha C_\alpha m \left[ 	
	\exp\left\{ -\frac{1}{2}\left(\frac{2^{m-1}-\mu_H}{\sigma_H}\right)^2 \right\} 
	+ \exp\left\{ -\frac{1}{2}\left(\frac{2^{m-1}+\mu_H}{\sigma_H}\right)^2 \right\} 
	\right]
	\\
	&& \lesssim p \alpha C_\alpha \exp\{ -C_{\sigma_H} 2^{2m} \}
\eean
for every large enough $m$, where $C_{\sigma_H}$ is a constant depending only on $\sigma_H$.

Note that $P$ is the convolution of $G$ and $N(0, \sigma_n^2)$.
If $Y_1 = Y_2 + Y_3$, where $Y_2$ and $Y_3$ are independent random variables following $G$ and $N(0, \sigma_n^2)$, respectively, then
\bean
	P(\overline B_m) = \Pr(|Y_1| > 2^{m-1}) &\leq& \Pr(|Y_2| > 2^{m-2}) + \Pr(|Y_3| > 2^{m-2}) 
	\\
	&\leq& G(\overline B_{m-1}) + 2 \big(1- \Phi_{\sigma_n}(2^{m-2}) \big).
\eean
Hence,
\bean
	&& P(B_m) \leq P(\overline B_m) \leq G(\overline B_{m-1}) + 2 \big(1- \Phi_{\sigma_n}(2^{m-2}) \big)
	\\
	&& \leq G(\overline B_{m-1}) + \exp\left( -\frac{2^{2(m-2)}}{2 \sigma_n^2} \right)
	\leq G(\overline B_{m-1}) + 2^{-pm},
\eean
where the last inequality holds for every large enough $m$.
It follows for any constants $C > 0$ and large enough $n$ that
\bean
	&& \Pi \big( P(B_m) > (K'+1) 2^{-pm} ~ \text{for some $m \geq C \log \widetilde\epsilon_n^{-1}$} \big)
	\\
	&& \leq \Pi \big( G(\overline B_{m-1}) > K' 2^{-pm} ~ \text{for some $m \geq C \log \widetilde\epsilon_n^{-1}$} \big)
	\\
	&& = \Pi \big( G(\overline B_m) > K' 2^{-p(m+1)} ~ \text{for some $m \geq C \log \widetilde\epsilon_n^{-1} -1$} \big)
	\\
	&& \leq \sum_{m \geq C \log \widetilde\epsilon_n^{-1}-1} \Pi \big( G(\overline B_m) > 2^{-pm} \big)
	\\
	&& \lesssim p \alpha C_\alpha \sum_{m \geq C \log \widetilde\epsilon_n^{-1}-1} \exp\left( -C_{\sigma_H} 2^{2m} \right)
	\lesssim p \alpha C_\alpha \exp\left( - \frac{C_{\sigma_H}}{4} \widetilde\epsilon_n^{-2C\log2} \right).
\eean
If we take $C = (p\log 2)^{-1}$, the right hand side of the last display is bounded by 
\bean
	p\alpha C_\alpha \exp\left( -\frac{C_{\sigma_H}}{4} \widetilde\epsilon_n^{-2/p} \right).
\eean
Since $\epsilon_n \leq A n^{-p/(2+2p)}$, $n\epsilon_n^2$ is bounded by a constant multiple of $\epsilon_n^{-2/p}$ for every $n$.
Hence, if $L$ is large enough, we have that
\bean
	\Pi \Big( P(B_m) > (K'+1) 2^{-pm} ~ \text{for some $m \geq C \log \widetilde \epsilon_n^{-1}$} \mid X_1, \ldots, X_n \Big) \rightarrow 0 ~\text{in probability}
\eean
by Lemma \ref{lem:ghosal2000prior}.

Note that by Lemma \ref{lem:set-rate},
\bean
	\Pi\Big( \big| P(B_m) - P_0(B_m) \big| \leq K'' \widetilde\epsilon_n ~ \forall m \leq  C \log \widetilde\epsilon_n^{-1}  \mid X_1, \ldots, X_n \Big) \rightarrow 1 \quad\text{in probability},
\eean
where $K''$ is a constant.
Since
\bean
	P_0(B_m) + K'' \widetilde\epsilon_n \leq (K+ K'') 2^{-pm}
\eean
for every $m \leq C\log \widetilde\epsilon_n^{-1}$, the proof is complete.
\end{proof}

\medskip
It is shown in the proof of Theorem 2 in \cite{ghosal2007posterior} that $\Pi(\cK_n) \geq e^{-n\epsilon_n^2}$ with $\epsilon_n = c n^{-2/5} (\log n)^2$ for some constant $c > 0$.
For any $p < 4$, note that $\epsilon_n \leq n^{-p/(2+2p)}$ for all large enough $n$.
Hence, the proof of Theorem \ref{thm:dpm-specific} is complete by \eqref{eq:moment-prob} and Lemma \ref{lem:dpm}.
\qed

\subsection{Proof of Theorem \ref{thm:convolution}}
Denote $p_G(x) = \int k_\sigma(x-z) dG(z)$.
We use the result of \cite{nguyen2013convergence}.
It is shown in the proof of Theorem 2 in \cite{nguyen2013convergence} that
\bean
	&& W_2^2(G, G_0) 
	\\
	&& \leq C\inf_{\delta\in(0,1)} \left[ \delta^2 + \|p_G - p_{G_0}\|_1^{2(s-2)/(1+2s)} 
	\bigg\{\int_{-1/\delta}^{1/\delta} \tilde k(t)^{-2} dt \bigg\}^{(s-2)/(1+2s)}  \right]
\eean
for any $s >2$ and $G$ with $M_2(G) < \infty$, where $C$ is a constant depending only on $s$.
Note that Theorem 2 of \cite{nguyen2013convergence} assumed that $G$ and $G_0$ are discrete probability measures with bounded supports, but finiteness of the second moment suffices as discussed therein.
The right hand side of the last display tends to zero as $\|p_G - p_{G_0}\|_1 \to 0$.
It follows that for every $\epsilon > 0$,
\bean
	\Pi\Big( W_2^2(G, G_0) > \epsilon \mid X_1, \ldots, X_n\Big) \to 0.
\eean
in probability.
Since $W_2^2(P_G, P_{G_0}) \leq W_2^2(G, G_0)$, the proof is complete.
\qed

\section*{Acknowledgements}
The authors are grateful for the comments of reviewers on an earlier version of the paper.
M. Chae was supported by the National Research Foundation of Korea (NRF) grant funded by the Korea government (MSIT) (No. 2020R1F1A1A01054718).
P. De Blasi is supported by MIUR, PRIN Project 2015SNS29B and acknowledges “Dipartimenti di Eccellenza” Grant 2018-2020.

\bibliographystyle{imsart-nameyear}
\bibliography{bib-short}

\end{document}